\documentclass[11pt,a4paper]{amsart}
\usepackage{amsmath,amssymb, amsbsy}
\usepackage{subfigure}
\usepackage[english]{babel}
\usepackage{graphicx}
\usepackage{amscd}
\usepackage{amsmath}
\usepackage{amsfonts}
\usepackage{amssymb}
\usepackage{amsthm}
\usepackage{latexsym}
\usepackage{xcolor}

\usepackage{todonotes}
\numberwithin{equation}{section}

\newtheorem{theorem}{Theorem}[section]
\newtheorem{proposition}[theorem]{Proposition}
\newtheorem{lemma}[theorem]{Lemma}
\newtheorem{corollary}{Corollary}[section]

\newtheorem{definition}[theorem]{Definition}

\newtheorem{remark}{Remark}[section]

\def\neweq#1{\begin{equation}\label{#1}}
\def\endeq{\end{equation}}

\newcommand{\R}{\mathbb{R}}
\newcommand{\N}{\mathbb{N}}

\newcommand{\eps}{\varepsilon}

\newcommand{\SA}{{\mathbb S}^{N-1}_\alpha}
\newcommand{\SAP}{{\mathbb S}^{N-1}_{\alpha,+}}

\newcommand{\Ha}{\mathcal H_\alpha^{N-1}}

\begin{document}

\title[  On solutions to a class of  degenerate equations with the Grushin operator]{  On solutions to a class of  degenerate equations with the Grushin operator}

\author{Laura Abatangelo, Alberto Ferrero, Paolo Luzzini}

\address{\hbox{\parbox{5.7in}{\medskip\noindent{Laura Abatangelo, \\
Politecnico di Milano, \\
        Dipartimento di Matematica, \\
        Piazza Leonardo da Vinci 32, 20133 Milano, Italy. \\
        {\em E-mail address: } {\tt laura.abatangelo@polimi.it}
        \\[5pt]
Alberto Ferrero, \\
Università del Piemonte Orientale, \\
        Dipartimento di Scienze e Innovazione Tecnologica, \\
        Viale Teresa Michel 11, 15121 Alessandria, Italy. \\
        {\em E-mail address: } {\tt alberto.ferrero@uniupo.it}  \\[5pt]
Paolo Luzzini, \\
Università del Piemonte Orientale, \\
        Dipartimento di Scienze e Innovazione Tecnologica, \\
        Viale Teresa Michel 11, 15121 Alessandria, Italy. \\
        {\em E-mail address: } {\tt paolo.luzzini@uniupo.it}
}}}}

\date{\today}

\begin{abstract}
The Grushin Laplacian $- \Delta_\alpha $ is a degenerate elliptic operator in   $\mathbb{R}^{h+k}$ that degenerates on
$\{0\} \times \mathbb{R}^k$. We consider weak solutions of $- \Delta_\alpha u= Vu$ in an open bounded connected domain $\Omega$ with $V \in W^{1,\sigma}(\Omega)$ and $\sigma > Q/2$, where $Q = h + (1+\alpha)k$ is the so-called homogeneous dimension of $\mathbb{R}^{h+k}$. By means of an Almgren-type monotonicity formula we identify the exact asymptotic blow-up profile of solutions on degenerate points of $\Omega$. As an application we derive strong unique continuation properties for solutions.
\end{abstract}

\maketitle


\vspace{.2cm}

\vspace{.2cm}

\section{Introduction}

For $N \in \mathbb{N}$, $N\ge 2$, $\alpha \in \N$, we  denote by $ \Delta_\alpha $ the operator defined by
\begin{equation*}
    \Delta_\alpha  := \Delta_x + |x|^{2\alpha} \Delta_y
\end{equation*}
where $x\in \R^h$, $y\in \R^k$, $h,k\ge 1$, $h+k=N$, and $\Delta_x$ and $\Delta_y$ are respectively the classical Laplace operators with respect to the $x$ and $y$ variables. Here and throughout the paper, $\N=\{0,1,2,3,\dots\}$ denotes the set of nonnegative integers.

Nowadays, $ \Delta_\alpha $  is  known as the Grushin (or Baouendi-Grushin) operator  and was introduced in a preliminary version by Baouendi \cite{Ba} and Grushin \cite{Gru70, Gru71}.
As one can immediately realize, this operator is not uniformly elliptic, due to the presence of the term $|x|^{2\alpha}$: there is a whole submanifold where the operator is degenerate, which is
\begin{equation}\label{eq:def-sigma}
\Sigma :=  \{(0,y)\in \R^h \times \R^k: y \in \R^k\} \subset \R^N.
\end{equation}

However, since $\alpha\in \N$, the Grushin operator can be written as a finite sum of squares of  smooth vector fields, i.e.  for some $n\ge 1$, $ \Delta_\alpha  =\sum_{i=1}^n Z_i^2$ (see Section \ref{s:functional-setting} for more details on this kind of representation), where $\{Z_i\}_{i=1,\ldots,n}$ satisfy the H\"ormander condition: if one considers the commutators $[Z_i,Z_j]=Z_i Z_j - Z_j Z_i $ for any $i\neq j$, then the set of all $Z_i$ and their iterated commutators up to a finite order generates $\R^N$ at any point of $\R^N$ (they generate a Lie algebra of maximum rank at any point).
 Under this condition, a famous result by Hormander~\cite{Hor1967} implies that the operator is {\em hypoelliptic}: for any open subset $\Omega \subset \R^N$, if $f\in C^{\infty}(\Omega)$ then the distributional solution $u$ of $ \Delta_\alpha  u=f$ is also $C^\infty(\Omega)$.

  In principle one can consider real $\alpha >0$, but in general the H\"ormander condition fails  because the generating vector fields are not smooth. This is the main reason why in several cases one considers the Grushin operator with $\alpha\in\N$. Incidentally, we mention that there is also a large literature investigating the case of non-smooth vector fields, where the H\"ormander condition does not apply. In this regard, so far from being exhaustive, we only mention the seminal work by Franchi and Lanconelli \cite{FrLa82}.

 %\bigskip

In the present paper we consider $\alpha\in\N$ and the equation
\begin{equation} \label{eq:u}
  - \Delta_\alpha  \, u =V(x,y) u \qquad \text{in } \Omega,
\end{equation}
where $\Omega$ is a connected bounded open set in $\R^N$ and $V$ is a potential. The term $Vu$ can be viewed as a zero-th order perturbation of the standard Grushin operator.  The potential  $V$ satisfies suitable regularity assumptions which will be specified in the next section. Here we just mention that in particular the eigenvalue equation $  - \Delta_\alpha  \, u =\lambda u$ will be included in our  analysis. Since $ \Delta_\alpha $ is uniformly elliptic far from the degenerate set, the most interesting case is $\Omega$ intersecting $\Sigma$. Moreover, $ \Delta_\alpha $ is  invariant with respect to translation in the variable $y \in \R^k$, so that we will always assume that $\Omega$ contains the origin.

The main objects of the present paper are asymptotic behaviors and unique continuation principles for weak solutions to \eqref{eq:u} (see \eqref{eq:var} for the definition of weak solution) around points of the degenerate set $\Omega \cap \Sigma$, which is supposed to be non empty, as already mentioned.

\subsection{Asymptotic behavior of solutions at the origin}

As it is well-known, under suitable assumptions on the potential, weak solutions to the equation for the classical Laplacian
\[
-\Delta u = V(x,y) u \quad \text{in }\Omega
\]
has the following asymptotic behavior at the origin: there exist $\gamma\in \N$ and a spherical harmonic $\Psi\in L^2(\mathbb S^{N-1})$ such that
\begin{align*}
&-\Delta_{\mathbb S^{N-1}} \Psi = \gamma(N-2+\gamma) \Psi,\\
  &  r^{-\gamma} u(r x) \to
    %\sum\limits_{\substack{\beta\in\N^N\\|\beta|=k}}\frac{D u^\beta(P)}{\beta!}x^\beta=
    |x|^\gamma \Psi\left(\frac{x}{|x|}\right) \quad \text{as }r\to0 \text{ in $H^1(B_R)$ for any $R>0$.}
\end{align*}
For a detailed statement of this kind of result, see for instance \cite[Theorem 1.3]{FFT2} under pointwise assumptions on the potential and \cite[Theorem 1.1]{FFT1} under more general assumptions including integral conditions on the potential and its weak derivatives.
Besides its own interest, identifying the precise blow-up asymptotic profile of a solution at a point, namely the origin without losing generality, is the key step to determine the asymptotic behavior of eigenvalues when the domain is affected by special singular perturbations. Indeed, as several recent results establish, when a small set concentrating at a point is removed from the original domain, the eigenfunctions' vanishing order at that point plays a fundamental role to determine the rate of convergence of the related eigenbranches for the Laplace operator (see e.g. \cite{AbBoLeMu21, AbFeHiLe19, AbLeMu22, AbLeMu24}), as well as   for different operators (see  Felli and Romani \cite{FeRo23} and \cite{AbFeNo20}). The same can be said about moving mixed boundary conditions, as showed in  Felli, Noris and Ognibene \cite{FeNoOg22}, and \cite{AbFeLe20, AbOg24}.

This kind of problems is one of the main motivation which led us to the study of asymptotic behavior of solutions to linear problems for the Grushin operator. Indeed, motivated by the increasing interest in the spectral theory and geometry of degenerate elliptic operators and sublaplacians (see e.g. Chen, Chen and Li \cite{ChChLi20}, Frank and Helffer \cite{FrHe24}, and \cite{LaLuMu21, LuPrSt24}), in future works we plan to investigate the asymptotic behavior of Grushin eigenvalues under singular perturbation of the domain.

In  the Grushin case, away from the singular set, the operator is uniformly elliptic, so that solutions share essentially the same characteristics as solutions to equations involving the standard Laplacian. For this reason  we will focus on the asymptotic behavior of solutions at a point lying on  the degenerate set $\Sigma$, that we can suppose to be the origin. In particular, we want to investigate how the degenerate manifold plays a role.

In general, this kind of asymptotic results can be obtained by an established procedure relying on an Almgren-type monotonicity formula (see for instance \cite{FFFN, FF1, FFT1, FFT2, FFT3}). The so-called Almgren-type frequency function is the ratio between the quadratic form attached to the equation and the $L^2$-norm of the trace of solutions respectively taken over the ball of radius $r$ and its boundary. Both of them are normalized with certain powers of the radius $r$ depending on the dimension.
 In the setting of the Grushin operator, as already pointed out in  Garofalo \cite{Garofalo}, when balls are taken away from the singular set, the ellipticity of the operator prevails and the correct dimension is exactly $N$. On the contrary, at points on the degenerate manifold the so-called homogeneous dimension
 \[
 Q := h + (1+\alpha)k
 \]
 takes over and the balls need to be taken with respect to a suitable distance from the origin induced by $ \Delta_\alpha $,  which is usually called  gauge norm:
 \begin{equation} \label{eq:distance}
   d_\alpha(x,y):=\left(|x|^{2(\alpha+1)}+(\alpha+1)^2 |y|^2\right)^{\frac{1}{2(\alpha+1)}}
   \qquad \text{for any } (x,y)\in \R^h \times \R^k \, .
 \end{equation}
The  dimension $Q$ is then the homogeneous dimension of $\R^N$ with respect to the anisotropic dilation:
\begin{equation} \label{eq:delta-lambda}
  \delta_\lambda(x,y) := (\lambda x, \lambda^{\alpha+1} y) \qquad \text{for any } (x,y) \in \R^h \times \R^k ,  \lambda>0\, .
\end{equation}
The operator $ \Delta_\alpha $ possesses the following scaling property with respect to $\delta_\lambda$:
\begin{equation*} % \label{eq:scaling-property}
    \Delta_\alpha  (u \circ \delta_\lambda) = \lambda^2 ( \Delta_\alpha   u) \circ \delta_\lambda
\end{equation*}
while for the distance $d_\alpha$ we have
\begin{equation} \label{eq:d-delta}
  d_\alpha(\delta_\lambda(x,y))=\lambda d_\alpha(x,y) \qquad \text{for any } (x,y)\in \R^h \times \R^k \, .
\end{equation}
Once we defined the distance function $d_\alpha$, we can now introduce the following \textit{balls centered at the origin} of radius $r>0$ and the corresponding \textit{spheres} by
\begin{equation} \label{eq:B-r-alpha}
    B_r^\alpha:=\left\{(x,y)\in \R^h \times \R^k:d_\alpha(x,y)<r \right\}
    \, , \quad \partial B_r^\alpha:=\left\{(x,y)\in \R^h \times \R^k:d_\alpha(x,y)=r \right\} \, .
\end{equation}
This definition combined with \eqref{eq:d-delta} yields
\begin{equation} \label{eq:delta-B}
  \delta_\lambda (B_r^\alpha)=B_{\lambda r}^\alpha  \quad \text{and} \quad
  \delta_\lambda (\partial B_r^\alpha)=\partial B_{\lambda r}^\alpha\, .
\end{equation}
In the same way, through $d_\alpha$, we can define balls $B_r^\alpha(z_0)$ and spheres $\partial B_r^\alpha(z_0)$ centered at a generic point $z_0 \in \R^N$.

The role of the homogeneous dimension  is highlighted once more in the  first result of the paper, that  classifies the asymptotic behavior of any weak solution to \eqref{eq:u} near $0$:

\begin{theorem}\label{t:main}
 Let $\Omega\subset\R^{h+k}$ a bounded connected open set containing $0$. Let $\alpha \in \N$ and let $u \in   W^{1,2}_{\alpha,\, {\rm loc}}(\Omega)$ be a nontrivial weak solution to \eqref{eq:u} with $V \in W^{1,\sigma}_{{\rm loc}}(\Omega)$ and $\sigma > Q/2$. Then there exists $j\in\N$ such that
\begin{equation} \label{eq:as-for}
\rho^{-\ell} \, u(\delta_\rho(x,y)) \to (d_\alpha(x,y))^\ell \ \Psi\left(\frac{x}{d_\alpha(x,y)},\frac{y}{[d_\alpha(x,y)]^{\alpha+1}}\right) \quad \text{as } \rho \to  0^+
\end{equation}
in $ W^{1,2}_\alpha (B_r^\alpha)$ and uniformly in $B_r^\alpha$ for any $r>0$, where
    \begin{equation}\label{eq:car-ell}
      \ell=\dfrac{-(Q-2)+\sqrt{(Q-2)^2+4\mu_j(\alpha+1)^2}}{2},
    \end{equation}
$\mu_j$ is an eigenvalue of the polar operator $-\mathcal L_\Theta$ attached to the Grushin Laplacian (see definition in \eqref{GrPo2}) and $\Psi \not\equiv 0 $ is one of its relative eigenfunctions.
\end{theorem}

  For the definition of the weighted Sobolev space $ W^{1,2}_\alpha (\Omega)$ we refer to Section \ref{s:functional-setting}.
In the previous theorem, the polar operator $\mathcal{L}_\Theta$ acts on functions defined on the unit Grushin sphere
$ \SA = \{ (x,y)\in\R^{h+k}:  d_\alpha(x,y)=1\}$ and plays the same role as the Laplace-Beltrami operator on the unit sphere for the standard Laplacian. The explicit form of  $\mathcal{L}_\Theta$, together with its eigenvalues and eigenfunctions, has been computed by
Garofalo and Shen \cite{GaSh94} for the case $h \geq 2, k=1, \alpha=1$, and recently by De Bie and Lian \cite{DeLi24} for $h \geq 2, k\geq 1, \alpha \in \N$  (see also Liu \cite{Li18}). Up to the authors' knowledge, the explicit computation of the eigenvalues and eigenfunctions of  $\mathcal{L}_\Theta$ in the case $h=1$ has not been done yet.
Indeed, if $h=1$ it is not possible to find, for every choice of $m\in \N$, a Grushin harmonic polynomial in the variables $(x,y) \in \R^{1+k}$ that is $\delta_\lambda$-homogeneous of order $m$, so that one can not use all the theory of orthogonal polynomials to construct a basis of eigenfunctions for   $\mathcal{L}_\Theta$.

The proof of Theorem \ref{t:main} is based, as we mentioned, on an Almgren monotonicity approach. As tools, we need a Pohozaev-type integral identity together with a careful regularity analysis for weak solutions to \eqref{eq:u}.

\begin{remark}  As mentioned above, for any $h\ge 2$, $k\ge 1$ and $\alpha \in \N$ the spectrum of $-\mathcal L_\Theta$ was completely determined  (see Garofalo and Shen \cite{GaSh94} for the case $h \geq 2, k=1, \alpha=1$ and De Bie and Lian \cite{DeLi24} for the general case).

 In particular, in these cases,  all the eigenvalues of $-\mathcal L_\Theta$ are  of  the form $n (n + Q - 2)/(\alpha+1)^2$ where $n \in \N$. We observe that if we insert an eigenvalue $ \mu_j = n (n + Q - 2)/(\alpha+1)^2$ in \eqref{eq:car-ell}, we obtain $\ell = n$ thus showing that the exponent in the asymptotic formula \eqref{eq:as-for} is always an integer.

It remains open the question of determining the eigenvalues of $-\mathcal L_\Theta$ in the remaining cases $h=1$, $k\ge 1$ and $\alpha \in\N$.

\end{remark}

\subsection{Unique continuation principles}

 Relevant consequences of Theorem \ref{t:main} are unique continuation principles.  The first one shows that if a solution has infinite vanishing order at a point, then it must vanish everywhere. Namely we have the following.

\begin{corollary} \label{c:1}
Under the assumptions of Theorem \ref{t:main}, let $z_0\in\Omega$ and $u\in  W^{1,2}_{\alpha,\, {\rm loc}}(\Omega)$ be a weak solution to \eqref{eq:u} such that $u(z)=O(|z-z_0|^n)$ as $z\to z_0$ for any $n\in\N$. Then $u\equiv 0$ in $\Omega$.
\end{corollary}
 %We remark that if $z_0$ is away from the singular set, the conclusion follows by unique continuation principles for uniformly elliptic operators, see for example \cite{So}. If $z_0$ lies on the singular set, we reach the thesis invoking Theorem \ref{t:main}.
  In a similar way we can also recover the following stronger result.
\begin{corollary} \label{c:2}
Under the assumptions of Theorem \ref{t:main}, let $z_0\in\Omega$ and $u\in  W^{1,2}_{\alpha,\, {\rm loc}}(\Omega)$ be a weak solution to \eqref{eq:u} such that% \textcolor{red}{CONTROLLARE}
\[
\int_{ B_r^\alpha(z_0)} |x|^{2\alpha} \, u^2 \, dxdy = O(r^n)\quad \text{as }r\to  0^+
\]
for any $n\in\N$. Then $u\equiv 0$ in $\Omega$.
\end{corollary}

We note also that the  degenerate set of the Grushin operator has zero measure in $\R^{N}$.  In the case $\Omega \setminus \Sigma$ is connected, this implies the validity of a ``weak'' unique continuation principle  by classical unique continuation for elliptic operators: if a weak solution to \eqref{eq:u} $u$ is identically zero on an open subset of $\Omega$ then $u\equiv 0$ on the whole $\Omega$.

We can strengthen it in Corollary \ref{c:positive-measure}, which is not a consequence of the unique continuation properties stated in Corollary \ref{c:1} and Corollary \ref{c:2} nor on other results contained in previous works like  Garofalo  \cite{Garofalo},  Garofalo and Shen \cite{GaSh94}, Banerjee, Garofalo and Manna \cite{BaGaMa20}, and  De Bie and Lian \cite{DeLi24}. It is a direct consequence of the asymptotic analysis of Theorem \ref{t:main}, as one can see in the proof of the corollary. See also Remark \ref{r:cor} for more details on the novelty of the corollary.

\begin{corollary} \label{c:positive-measure}
Under the assumptions of Theorem \ref{t:main}, let $u\in  W^{1,2}_{\alpha,\, {\rm loc}}(\Omega)$ be a weak solution to \eqref{eq:u} such that $u\equiv 0$ on a subset $\omega\subseteq\Omega$ with nonzero measure in $\R^{h+k}$. Then $u\equiv 0$ in $\Omega$.
\end{corollary}

\begin{remark} \label{r:cor}  We observe that the proofs of Corollaries \ref{c:1} and \ref{c:2} when $z_0 \notin  \Sigma$ and the proof of Corollary \ref{c:positive-measure} can be obtained from the classical unique continuation principle for uniformly elliptic operators (see de Figueiredo \cite{deFiGo} and Sogge \cite{So}) only when $\Omega \setminus \Sigma$ is connected. Indeed, one has to apply the classical unique continuation principle outside an arbitrarily small neighborhood of the singular set $\Sigma$. But when $h=1$, the set $\Sigma$ is a hyperplane and $\Omega \setminus \Sigma$ is an open set made of two connected components. In the case of Corollaries \ref{c:1} and \ref{c:2}, the classical unique continuation principle for uniformly elliptic operators only implies vanishing of the solution $u$ on the connected component containing $z_0$. 

Concerning Corollary \ref{c:positive-measure}, if  for example the set $\omega$, defined in its statement, is entirely contained in one of the two connected components of $\Omega \setminus \Sigma$ or the intersection with the other connected component has zero measure, then the classical unique continuation principle only implies that the solution $u$ of \eqref{eq:u} vanishes on one of the two connected components but not necessarily on the other one. 

We can conclude that Corollaries \ref{c:1}, \ref{c:2} and \ref{c:positive-measure} are new results and, as the reader can check, their proofs cannot do without the asymptotic analysis contained in Theorem \ref{t:main}.
\end{remark}

%We note that in this latter principle the assumptions rely on the special {\em distance function} attached to the operator. As it is recalled below in Section 2, this is in fact responsible both of the weight $\psi_\alpha$ and of the renormalized $N-1$ Lebesgue measure $d\mathcal H_\alpha^{N-1}$.

All the aforementioned   corollaries complement the earlier unique continuation results on the Grushin operator in the literature. One over all is Garofalo \cite{Garofalo}, where a strong unique continuation principle is proved under the assumption that the potential $V$ satisfies suitable pointwise growth conditions.
%a bound of the type:
%\[
%|V| \leq C\psi_\alpha,
%\]
%where $\psi_\alpha$ is a suitable angle function vanishing on $\Sigma$ (see \eqref{eq:psi-alpha} for its definition).
In a subsequent work by Garofalo and Shen \cite{GaSh94},  the authors were able to remove the aforementioned bound on $V$ replacing it with an integrability condition on the potential $V$,
but only in the case  $\alpha=1$, $h\geq2$ and $k=1$. Indeed, by means of Carleman estimates they establish a strong unique continuation property under certain $L^p$ condition on $V$ if the solution vanishes of infinite order at a point of the singular set in the $L^2$ mean. This result has been extended and improved very recently by De Bie and Lian \cite{DeLi24} taking into account  the case  $\alpha \in\N$, $h\geq2$ and $k\geq 1$. Note that our results are not implied by those in
De Bie and Lian \cite{DeLi24}.

For example, looking at Table 1 in De Bie and Lian \cite{DeLi24}, we see that in the case $h \ge 4$ and $k\ge 2$, they prove the validity of the unique continuation principle for $V \in L^r_{{\rm loc}}(\Omega)$ with $r>(\alpha+1)(h+k-2)+1$.
In order to construct an example of a potential $V$ satisfying our assumption on $V$ contained in Theorem \ref{t:main} but not their one, we have to look for $h,k,\alpha$ and $\sigma > \frac Q2$,  $\sigma<N$ such that $\frac{NQ}{2N-Q} < \sigma^* = \frac{N\sigma}{N-\sigma} < (\alpha+1)(h+k-2) + 1$. The left term in these inequalities is strictly less than the right one if and only if
\begin{equation} \label{eq:ineq-h}
   \alpha h^2 - (\alpha^2 k+2\alpha+1) h - \alpha(\alpha+1) k^2 + (2\alpha^2-\alpha-1) k > 0 \, ,
\end{equation}
which is verified for $h$ large enough once we have fixed $\alpha$ and $k$. If we choose for simplicity the smallest values of $\alpha$ and $k$ admissible in this case, i.e. $\alpha=1$ and $k=2$, the inequality becomes $h^2-5h-8>0$ whose least positive integer solution is $h=7$. With this choice of $h,k,\alpha$ we choose $\sigma = 5.6 > 5.5 = \frac{Q}{2}$ and the following potential $V(x,y)=(|x|^2 + |y|^2)^{-0.3}$ that is singular at the origin and radial with respect to the Euclidean distance. Denoting by $B_1$ the unit ball with respect to the Euclidean distante and choosing it as a domain for \eqref{eq:u}, by direct computation, one can verify that $V \in W^{1,\sigma}(B_1)$ but $V \not \in L^r_{{\rm loc}}(B_1)$ when $r>(\alpha+1)(h+k-2)+1=15$. Similar examples can be found for any $\alpha\ge 1$ and $k\ge 2$ provided that $h$ satisfies \eqref{eq:ineq-h}.

Finally, we observe that our unique continuation principles are always new when $h=1$ since this case was never considered in \cite{DeLi24,GaSh94}.

The paper is organized as follows. Section \ref{s:functional-setting} is devoted to establish the correct functional setting and  some notation. Section \ref{s:spherical} describes the Grushin operator in spherical coordinates. Section \ref{s:monotonicity} deals with the monotonicity argument that is the fundamental tool for proving the main results of the paper. The same section contains the proofs of Theorem \ref{t:main}, and Corollaries \ref{c:1}-\ref{c:positive-measure}. Section \ref{s:regularity} is devoted to some fundamental properties of weighted Sobolev spaces and to the regularity of solutions of equations with the operator $ \Delta_\alpha $.  It is worth noting that the results from Section  \ref{s:regularity}, although used in the previous section, are presented later for the sake of clarity.     Finally, the paper is equipped with a detailed appendix containing some basic scaling properties and a rigorous proof of a Pohozaev type identity.

\bigskip

\section{ Functional setting} \label{s:functional-setting}

We recall some basic facts on the Grushin Laplacian, see for example Garofalo \cite{Garofalo}.

Let us define the variable coefficients matrix
\begin{equation*}
   A_\alpha(x,y)=
   \left(
   \begin{tabular}{c|c}
     $I_{\R^h}$ & $0$  \\
     \hline
      $0$       &  $|x|^{2\alpha} I_{\R^k}$
   \end{tabular}
   \right)
\end{equation*}
and the following \textit{Grushin gradient} $\nabla_\alpha$ defined by $\nabla_\alpha v=(\nabla_x v, |x|^\alpha \nabla_y v)$ where $\nabla_x$ and the $\nabla_y$ denotes the gradient vectors with respect to $x$ and $y$ respectively.

Letting $v$ be a sufficiently smooth function defined on a domain $\Omega$ of $\R^N$ with sufficiently smooth boundary, the definitions of $A_\alpha$ and $\nabla_\alpha$ formally yield
$$
    \Delta_\alpha  v={\rm div}(A_\alpha \nabla v)
$$
so that by integration by parts we have
\begin{align} \label{eq:int-parts}
  & \int_\Omega (- \Delta_\alpha  v) v \, dxdy=\int_\Omega -{\rm div}(A_\alpha \nabla v) v \, dxdy=\int_\Omega A_\alpha \nabla v \cdot \nabla v \, dxdy-\int_{\partial \Omega} v (A_\alpha \nabla v)\cdot \nu d\mathcal H^{N-1} \\[7pt]
 \notag  & \qquad =\int_\Omega  |\nabla_\alpha v|^2 \, dxdy-\int_{\partial \Omega} v (A_\alpha \nabla v)\cdot \nu \, d\mathcal H^{N-1} \, ,
\end{align}
where $\mathcal H^{N-1}$ denotes the $(N-1)$-dimensional Hausdorff measure in $\R^N$ and $\nu$ the outer unit normal vector to $\partial \Omega$. Identity \eqref{eq:int-parts} suggests that a suitable class of weighted Sobolev spaces should be the natural setting for defining the notion of weak solution for an equation with the Grushin operator. To define this kind of class of functional spaces, we first need to introduce the following notations. If $J=(h_1,\dots,j_h)\in \N^h$ and $x=(x_1,\dots,x_h)\in \R^h$, we define
\begin{equation*}
   x^J:=x_1^{j_1}\cdot \dots \cdot x_h^{j_h} \, , \qquad |J|:=j_1+\dots +j_h
\end{equation*}
and for any $\alpha\in \N$ and $J=(j_1,\dots,j_h)\in \N$ with $|J|=\alpha$, we define the following extension of the usual binomial coefficient
\begin{equation*}
    \binom{\alpha}{J}:=\frac{\alpha!}{(j_1)! \cdot \hphantom{} \dots \hphantom{} \cdot (j_h)!} \, .
\end{equation*}
With this notation one can easily verify that
\begin{equation} \label{eq:binom}
   |x|^{2\alpha} \, \Delta_y =\sum_{\ell=1}^{k} \quad  \sum_{J\in \N^h , \ |J|=\alpha}   \binom{\alpha}{J} (x^J)^2  \, \frac{\partial^2}{\partial y_\ell^2} \, .
\end{equation}
Looking at \eqref{eq:binom}, it seems natural to define the following vector fields
\begin{equation} \label{eq:vector-fields}
  X_j:=\frac{\partial}{\partial x_j} , \quad j\in \{1,\dots,h\} \, ,
      \qquad \qquad Y_{J,\ell}:=\sqrt{\binom{\alpha}{J}} \, x^J \frac{\partial}{\partial y_\ell} , \quad
      J\in \N^h , \, |J|=\alpha, \, \ell\in \{1,\dots,k\} \, .
\end{equation}
Indeed, in this way we may write
\begin{equation} \label{eq:hormander}
   \Delta_\alpha =\sum_{j=1}^{h} X_j^2 + \sum_{\ell=1}^{k} \quad  \sum_{J\in \N^h , \ |J|=\alpha} Y_{J,\ell}^2 \, .
\end{equation}
We are ready to introduce the following class of weighted Sobolev spaces denoted by $ W^{1,p}_\alpha (\Omega)$ which are defined by
\begin{align*} % \label{eq:def-H1}
   W^{1,p}_\alpha (\Omega):=& \Big\{v\in L^p(\Omega): X_j v\in L^p(\Omega) \quad \forall \, j\in \{1,\dots,h\},  \\
  & \notag \qquad Y_{J,\ell} \, v \in L^p(\Omega) \quad \forall \, J \in \N^h, \, |J|=\alpha, \, \forall  \, \ell \in \{1,\dots,k\} \Big\} \, .
\end{align*}
In particular we have that $|\nabla_\alpha v|\in L^p(\Omega)$ if $v\in  W^{1,p}_\alpha (\Omega)$.

We clarify that $X_j v$ and $Y_{J,\ell} \, v$ have to be interpreted in the weak sense. In particular we say that $w=Y_{J,\ell} \,  v$ in the weak sense if $w\in L^1_{{\rm loc}}(\Omega)$ and
\begin{equation} \label{eq:weak-Yj}
     \int_\Omega w \varphi \, dxdy=-\int_\Omega v Y_{J,\ell} \, \varphi \, dxdy \qquad \text{for any } \varphi\in C^\infty_c(\Omega) \, .
\end{equation}
The definition of $X_j v$ in the weak sense is standard being it simply a partial derivative.

The space $ W^{1,p}_\alpha (\Omega)$ endowed with the norm
\begin{equation*}
    \|u\|_{ W^{1,p}_\alpha (\Omega)}:=\left(\int_\Omega |\nabla_\alpha u|^p dxdy+\int_\Omega |u|^p dxdy\right)^{\frac 1p}
\end{equation*}
is a Banach space. In particular, when $p=2$ we have that the space $ W^{1,2}_\alpha (\Omega)$ endowed with the scalar product
\begin{equation*}
   (u,v)_{ W^{1,2}_\alpha (\Omega)}:=\int_\Omega \nabla_\alpha u \cdot \nabla_\alpha v \, dxdy +\int_\Omega uv \, dxdy
   \qquad \text{for any } u,v\in  W^{1,2}_\alpha (\Omega)
\end{equation*}
is a Hilbert space.

We denote by $ W^{1,p}_{\alpha,0} (\Omega)$ the subspace of $ W^{1,p}_\alpha (\Omega)$ defined as the closure in $ W^{1,p}_\alpha (\Omega)$ of $C^\infty_c(\Omega)$ where $C^\infty_c(\Omega)$ denotes the space of $C^\infty$-functions with compact support in $\Omega$.
We also denote by  $ W^{1,p}_{\alpha,c} (\Omega)$ the subspace of $ W^{1,p}_\alpha (\Omega)$-functions with compact support in $\Omega$.

Since we will also need to apply regularity results for weak solutions of equations with the Grushin operator, we also introduce the following class of second order weighted Sobolev spaces $ W^{2,p}_\alpha (\Omega)$ defined by
\begin{align*} % \label{eq:def-H2}
   W^{2,p}_\alpha (\Omega):=& \Big\{v\in L^p(\Omega): X_j v, \ Y_{J,\ell}\, v, \  X_{j_1} X_{j_2} v,  \
  Y_{J_1,\ell_1} Y_{J_2,\ell_2} \, v , \ X_j Y_{J,\ell} \, v , \ Y_{J,\ell} \, X_j v\in L^p(\Omega) \\
   & \quad \forall \, j, j_1, j_2 \in \{1,\dots,h\},
   \notag \quad \forall \, J, J_1, J_2 \in \N^h, \,  |J|=|J_1|=|J_2|=\alpha, \, \forall  \, \ell, \ell_1, \ell_2 \in \{1,\dots,k\} \Big\} \, .
\end{align*}
%\begin{align} \label{eq:def-H2}
%  H^2_G(\Omega):=& \Big\{v\in L^2(\Omega): X_j v, \ Y_{J,\ell}\, v, \  X_{j_1} X_{j_2} v,  \
%  Y_{J_1,\ell_1} Y_{J_2,\ell_2} \, v , \ X_j Y_{J,\ell} \, v , \ Y_{J,\ell} \, X_j v\in L^2(\Omega) \\
%   & \quad \forall \, j, j_1, j_2 \in \{1,\dots,h\},
%   \notag \qquad \forall \, J, J_1, J_2 \in \N^h, \, |J|=\alpha, \, \forall  \, \ell, \ell_1, \ell_2 \in \{1,\dots,k\} \Big\} \, .
%\end{align}
The space $ W^{2,p}_\alpha (\Omega)$ endowed with the norm
\begin{align*}
    \|u\|_{ W^{2,p}_\alpha (\Omega)}:=& \left\{\int_\Omega |u|^p \, dxdy + \int_\Omega |\nabla_\alpha u|^p \, dxdy+\sum_{j_1, j_2=1}^{h}
       \int_\Omega |X_{j_1} X_{j_2} u|^p dxdy  \right. \\
     &
     +\sum_{\ell_1, \ell_2=1}^{k} \quad  \sum_{J_1, J_2 \in \N^h , \ |J_1|=|J_2|=\alpha} \int_\Omega
      |Y_{J_1,\ell_1} Y_{J_2,\ell_2} \, u|^p  dxdy
    \\
     &  \left. +\sum_{j=1}^{h}\sum_{\ell=1}^{k} \quad  \sum_{J\in \N^h , \ |J|=\alpha} \int_\Omega
      \Big[|X_j Y_{J,\ell} \, u|^p +|Y_{J,\ell} \, X_j u|^p \Big]dxdy \right\}^{\frac 1p} \quad
      \text{for any } u \in  W^{2,p}_\alpha (\Omega)
\end{align*}
is a Banach space. Similarly, we can define the spaces of higher order $ W^{m,p}_\alpha (\Omega)$ for any integer $m\ge 1$.

The basic properties of $ W^{1,2}_\alpha (\Omega)$ and $ W^{1,2}_{\alpha,0} (\Omega)$ are collected in Section \ref{s:regularity}. For other properties of the Grushin Laplacian $ \Delta_\alpha $ and for a complete collection of the tools employed in the proofs of the main results see again Section \ref{s:regularity} and the Appendix.

We now proceed by introducing the main assumptions and by providing the rigorous notion of weak solution for equation \eqref{eq:u}. Let us  recall that the homogeneous dimension is $Q=h+(\alpha+1)k$. For more details on the role of the number $Q$ see the Appendix.
Let $\Omega$ be a domain in $\R^N$ and let us assume that
\begin{equation} \label{eq:V}
    V \in W^{1,\sigma}_{{\rm loc}}(\Omega) \, ,   \qquad    \sigma>\frac{Q}{2} \, .%W^{1,\infty}_{{\rm loc}}(\Omega \setminus \Sigma) \cap W^{1,1}_{{\rm loc}}(\Omega) \cap L^\sigma(\Omega) \quad \text{and } \quad |\nabla V(x,y) \cdot X_G(x,y)| \le C  \qquad \text{a.e. in } \Omega
\end{equation}
We say that $u\in  W^{1,2}_{\alpha, \rm loc} (\Omega)$ is a weak solution of \eqref{eq:u} if
\begin{equation}\label{eq:var}
  \int_\Omega \nabla_\alpha u \cdot \nabla_\alpha v \, dxdy=\int_\Omega V(x,y) uv \, dxdy
  \qquad \text{for any } v \in  W^{1,2}_{\alpha,c} (\Omega) \, .
\end{equation}
We conclude this section with a list of notations which will be used throughout the paper. We denote by  $X_G$ the vector field $X_G(x,y):=(x,(\alpha+1)y)$ defined for any $(x,y)\in \R^N=\R^h\times \R^k$.
This vector field will be also identified with the corresponding differential operator defined by
\begin{equation} \label{eq:def-XG}
   X_G:= \sum_{i=1}^{h} x_i \frac{\partial}{\partial x_i}+(\alpha+1) \sum_{j=1}^{k} y_j \frac{\partial}{\partial y_j} \, .
\end{equation}
This means that if $v=v(x,y)$ is a function then $X_G \, v:=X_G \cdot \nabla v$.
The differential operator $X_G$ leaves unchanged the distance function:
\begin{equation*}
   X_G d_\alpha = d_\alpha \qquad \text{in } \R^N \setminus \{0\} \, .
\end{equation*}
We introduce the following $(N-1)$-dimensional measure denoted by $\mathcal H_\alpha^{N-1}$ and defined by
\begin{equation*}
  \mathcal H_\alpha^{N-1}(M):=\int_M \frac{1}{|\nabla d_\alpha|} \, d\mathcal H^{N-1}
\end{equation*}
for any measurable set $M\subseteq \R^N$ with finite $(N-1)$-dimensional Hausdorff measure.

  Finally,  according to Garofalo \cite{Garofalo}, we also introduce the  angular  function
\begin{equation} \label{eq:psi-alpha}
    \psi_\alpha(x,y):=\frac{|x|^{2\alpha}}{(d_\alpha(x,y))^{2\alpha}} \qquad \text{for any } (x,y)\in \R^N \setminus \{0\} \, .
\end{equation}

\section{The Grushin operator in spherical coordinates} \label{s:spherical}

 In this section we introduce suitable spherical coordinates adapted to the Grushin structure and  a decomposition of $ \Delta_\alpha $ in such a system. Note that these spherical coordinates have been already introduced in full generality in D'Ambrosio and Lucente \cite{DaLu03} (see also Garofalo and Shen \cite{GaSh94} and De Bie and Lian \cite{DeLi24}). Then we study the spectral structure of the spherical operator.

With the only purpose of using a unified notation, we only treat  in detail the cases $h,k\ge 2$ and we only give a brief explanation when at least one between $h$ and $k$ equals $1$.

\medskip

{\bf The case $h,k \ge 2$.} Denoting by $r = |x|$, $t = |y|$ the radial variables in $\mathbb{R}^h$ and $\mathbb{R}^k$ and by $\theta \in \mathbb{R}^{h-1}, \eta\in \mathbb{R}^{k-1}$ the corresponding angular variables, and writing the operators $\Delta_x$ and $\Delta_y$  in spherical coordinates in  $\mathbb{R}^h$ and $\mathbb{R}^k$ respectively,
 given a sufficiently smooth function $u$  we obtain
 \begin{equation}\label{GP}
     \Delta_\alpha  u =  \left(\frac{\partial^2 u}{\partial r^2} +\frac{h-1}{r} \, \frac{\partial u}{\partial r} + \frac{1}{r^2} \, \Delta_\theta u\right)
+r^{2\alpha}\left(  \frac{\partial^2 u}{\partial t^2} +\frac{k-1}{t} \, \frac{\partial u}{\partial t} + \frac{1}{t^2} \, \Delta_\eta u\right)
 \end{equation}
where $\Delta_\theta$ and $\Delta_\eta$ denote the Laplace-Beltrami operators on the $(h-1)$-dimensional and $(k-1)$-dimensional spheres.

Next, we  plan to use as a global radial variable the natural norm associated with the Grushin operator, that is
 $ \rho = d_\alpha(x,y)$ where $d_\alpha$ was defined in \eqref{eq:distance}.
 To do that, the right change of variables is:
 \begin{equation}\label{chvar}
 (r,t) = \Phi(\rho,\varphi) := \left(\rho (\sin \varphi)^\frac{1}{\alpha+1}, \frac{\rho^{\alpha+1}}{\alpha+1}\cos \varphi\right)
 \end{equation}
 with $\rho \in (0,+\infty)$, $\varphi \in \left(0,\frac \pi 2\right)$. The Jacobian matrix of the above transformation is
 \[
 J_\Phi(\rho, \varphi) = \begin{pmatrix} \frac{\partial r }{\partial \rho} & \frac{\partial r }{\partial \varphi}\\
 \frac{\partial t }{\partial \rho} & \frac{\partial t }{\partial \varphi}
  \end{pmatrix} =
 \begin{pmatrix}(\sin \varphi)^\frac{1}{\alpha+1} & \frac{\rho}{\alpha+1}(\sin \varphi)^{-\frac{\alpha}{\alpha+1}}\cos \varphi\\
\rho^{\alpha}\cos \varphi& -\frac{\rho^{\alpha+1}}{\alpha+1}\sin \varphi
 \end{pmatrix}
 \]
 and its determinant is:
 \begin{align} \label{eq:det-rho-phi}
 \det J_\Phi(\rho,\varphi) &=  -\frac{\rho^{\alpha+1}}{\alpha+1} \, (\sin \varphi)^{-\frac{\alpha}{\alpha+1}} \, .
 \end{align}
Clearly
\begin{align*}
J_{\Phi^{-1}}(r,t) = \frac{1}{\det  J_\Phi(\rho,\varphi) }\begin{pmatrix}
\frac{\partial t}{\partial \varphi} & -\frac{\partial r}{\partial \varphi} \\
-\frac{\partial t}{\partial \rho} & \frac{\partial r}{\partial \rho}
\end{pmatrix} = \begin{pmatrix} \left(\sin \varphi\right)^\frac{2\alpha+1}{\alpha+1}&\rho^{-\alpha}\cos \varphi \\
\frac{\alpha+1}{\rho}\, (\sin \varphi)^{\frac{\alpha}{\alpha+1}}\cos \varphi& -\frac{\alpha+1}{\rho^{\alpha+1}}\sin \varphi
 \end{pmatrix} \, ,
\end{align*}
 so that
 \begin{align} \label{eq:p-r-t}
 \frac{\partial u }{\partial r} = \frac{\partial u }{\partial \rho} \left(\sin \varphi\right)^\frac{2\alpha+1}{\alpha+1} + \frac{\partial u }{\partial \varphi }\frac{\alpha+1}{\rho}(\sin \varphi)^{\frac{\alpha}{\alpha+1}}\cos \varphi \, ,
 \qquad
  \frac{\partial u }{\partial t } = \frac{\partial u }{\partial \rho} \, \rho^{-\alpha}\cos \varphi-\frac{\partial u }{\partial \varphi }\frac{\alpha+1}{\rho^{\alpha+1}}\sin \varphi .
 \end{align}
 A possible way to proceed, with the purpose of writing the Grushin operator  in polar coordinates, is to use \eqref{GP} and to compute all the second order derivatives of $u$ with respect to the variables $r$ and $t$ and express them in terms of second order derivatives in the variables $\rho$ and $\varphi$, as we already did for first order derivatives.
 Here, we proceed differently with the purpose of simplifying computations and  recovering a proper variational formulation for the spherical operator $\mathcal L_\Theta$ appearing in the representation \eqref{GrPo} below.

 Given two functions $u$ and $v$ sufficiently smooth, thanks to \eqref{eq:p-r-t} we may write in polar coordinates the following product:
 \begin{align} \label{eq:gradu-gradv}
    & \nabla_\alpha u \cdot \nabla_\alpha v=\frac{\partial u}{\partial r} \frac{\partial v}{\partial r}+\frac{1}{r^2}
    \, \nabla_\theta u \cdot \nabla_\theta v+r^{2\alpha} \frac{\partial u}{\partial t} \frac{\partial v}{\partial t}+\frac{r^{2\alpha}}{t^2}
    \, \nabla_\eta u \cdot \nabla_\eta v \\[10pt]
 \notag   & \qquad = \psi_\alpha  \frac{\partial u}{\partial \rho} \frac{\partial v}{\partial \rho}
  +\frac{(\alpha+1)^2}{\rho^2}\, \psi_\alpha \left[\frac{\partial u}{\partial \varphi} \frac{\partial v}{\partial \varphi}
   +\frac{1}{(\alpha+1)^2 (\sin \varphi)^2} \, \nabla_\theta u \cdot \nabla_\theta v+ \frac{1}{(\cos \varphi)^2} \, \nabla_\eta u \cdot \nabla_\eta v  \right] \, .
 \end{align}

 Then, assuming that $u$ is a sufficiently smooth function defined in a neighborhood of the origin $B_R^\alpha$ and $v \in C^\infty_c(B_R^\alpha)$, recalling that $\SA$ denotes the unit Grushin sphere $\partial B_1^\alpha$, we may write
 \begin{align*}
    & \int_{B_R^\alpha} (- \Delta_\alpha  u) v \, dxdy = \int_{B_R^\alpha} \nabla_\alpha u \cdot \nabla_\alpha v \, dxdy
       = \int_{\SA} \left( \int_0^R \psi_\alpha  \frac{\partial u}{\partial \rho} \frac{\partial v}{\partial \rho} \,
          \rho^{Q-1} \, d\rho \right) d\Ha  \\[10pt]
    & \quad + \int_{0}^{R} \frac{(\alpha+1)^2}{\rho^2} \, \rho^{Q-1} \left\{ \int_{\left(0,\frac \pi 2\right) \times \mathbb{S}^{h-1}\times \mathbb{S}^{k-1}}   \psi_\alpha \left[\frac{\partial u}{\partial \varphi} \frac{\partial v}{\partial \varphi}
   +\frac{1}{(\alpha+1)^2 (\sin \varphi)^2} \, \nabla_\theta u \cdot \nabla_\theta v   \right.  \right. \\[10pt]
  & \quad  \left. \left. + \frac{1}{(\cos \varphi)^2} \, \nabla_\eta u \cdot \nabla_\eta v  \right] \frac{1}{(\alpha+1)^k} \, (\sin \varphi)^{\frac{h-1-\alpha}{\alpha+1}} (\cos \varphi)^{k-1}
     \, d\varphi \, d\mathcal H^{h-1}(\theta) \, d\mathcal H^{k-1}(\eta)  \right\} \, d\rho
 \end{align*}
 where we used the fact that
\begin{equation} \label{eq:volume-integral}
   dxdy=\frac{\rho^{Q-1}}{(\alpha+1)^k} \, (\sin \varphi)^{\frac{h-1-\alpha}{\alpha+1}} (\cos \varphi)^{k-1}
    d\rho \, d\varphi \, d\mathcal H^{h-1}(\theta) \, d\mathcal H^{k-1}(\eta)=\rho^{Q-1}  d\rho \, d\Ha
\end{equation}
and where \eqref{eq:volume-integral} follows from \eqref{chvar} and \eqref{eq:det-rho-phi} .

Integrating by parts we obtain
\begin{align*}
    & \int_{B_R^\alpha} (- \Delta_\alpha  u) v \, dxdy = \int_{B_R^\alpha} - \psi_\alpha\left(\frac{\partial^2 u}{\partial \rho^2}+\frac{Q-1}{\rho}\frac{\partial u }{\partial \rho}+\frac{(\alpha+1)^2}{\rho^2}\mathcal{L}_\Theta u \right)
    v \, dxdy
\end{align*}
so that we may conclude that
\begin{equation}\label{GrPo}
 \Delta_\alpha  = \psi_\alpha\left(\frac{\partial^2}{\partial \rho^2}+\frac{Q-1}{\rho}\frac{\partial  }{\partial \rho}+\frac{(\alpha+1)^2}{\rho^2}\mathcal{L}_\Theta \right) \, ,
\end{equation}
where $\mathcal{L}_\Theta$ is the surface operator on the Grushin unit sphere $\SA$ defined by
\begin{equation}\label{GrPo2}
 \mathcal{L}_\Theta := \frac{\partial^2 }{\partial\varphi^2}+ \left(\frac{h-1+\alpha}{\alpha+1}\frac{\cos\varphi}{\sin\varphi}  -(k-1)\frac{\sin\varphi}{\cos\varphi}\right)\frac{\partial  }{\partial \varphi}+\frac{1}{((\alpha+1)\sin\varphi)^2}\Delta_\theta+
 \frac{1}{(\cos\varphi)^2}\Delta_\eta \, .
\end{equation}
Here $\Theta$ denotes an arbitrary element of $\SA$ whose angular coordinates are given by $\theta_1,\dots, \theta_{h-1}$, $\eta_1,\dots, \eta_{k-1},\varphi$.

Note that  Garofalo and Shen in \cite[p.135]{GaSh94} and De Bie and Lian \cite{DeLi24} already derived formula \eqref{GrPo}  respectively in the case $k=1, \alpha=1$ and in the general  case $k,\alpha \in \mathbb{N}\setminus\{0\}$.

%\begin{align} \label{eq:grad-polar}
%  |\nabla_\alpha u|^2=\psi_\alpha  \left(\frac{\partial u}{\partial \rho} \right)^2
%  +\frac{(\alpha+1)^2}{\rho^2}\, \psi_\alpha \left[\left(\frac{\partial u}{\partial \varphi} \right)^2
%   +\frac{1}{(\alpha+1)^2 (\sin \varphi)^2} \, |\nabla_\theta u|^2+ \frac{1}{(\cos \varphi)^2} \, |\nabla_\eta u|^2 \right]
 %  =(\sin \varphi)^{\frac{2\alpha}{\alpha+1}} \left(\frac{\partial u}{\partial \rho} \right)^2
%  +\frac{(\alpha+1)^2}{\rho^2} \, (\sin \varphi)^{\frac{2\alpha}{\alpha+1}} \left(\frac{\partial u}{\partial \varphi} \right)^2+\frac{|\nabla_\theta u|^2}{\rho^2 (\sin \varphi)^{\frac{2}{\alpha+1}}}
 % +\frac{(\alpha+1)^2 \, (\sin \varphi)^{\frac{2\alpha}{\alpha+1}}|\nabla_\eta u|^2}{\rho^{2}(\cos \varphi)^2}  \\[7pt]
 % \notag &
%\end{align}

From the above arguments we also deduce that if $u,v \in C^\infty(\overline{B_R^\alpha})$ with $v$ compactly supported in $B_R^\alpha$, then
\begin{align} \label{eq:f-b-sferica}
   & -\int_{\SA} \psi_\alpha(\Theta) \mathcal L_\Theta u(\rho,\Theta) v(\rho,\Theta) \, d\Ha \\[7pt]
 \notag  & =\int_{\SA} \psi_\alpha(\Theta) \left[\frac{\partial u}{\partial \varphi}
   \frac{\partial v}{\partial \varphi}+\frac{1}{(\alpha+1)^2 (\sin \varphi)^2} \, \nabla_\theta u \cdot \nabla_\theta v+\frac{1}{(\cos \varphi)^2} \, \nabla_\eta u \cdot \nabla_\eta v\right] d\Ha =: b_\alpha(u,v)
\end{align}
for any $\rho \in (0,R)$, which gives the variational formulation for the operator $\mathcal L_\Theta$. Such a variational formulation now needs a suitable functional space on which we can define the bilinear form appearing in the second line of \eqref{eq:f-b-sferica}.

Following the scheme usually used for defining classical Sobolev spaces on the boundary of a domain, see for example Lions and Magenes \cite[Chapter 1, Section 7.3]{LionsMagenes}, we can define the weighted Sobolev space $ W^{1,2}_\alpha (\SA)$. We explain in more details how this procedure may be implemented.

We can construct a partition of unity  $\phi_1,\dots,\phi_n, \phi_{n+1},\dots,\phi_{n+m}$ associated with an open covering $O_1, \dots , O_n, O_{n+1}, \dots, O_{n+m}$ of $\SA$  in $\mathbb{R}^N$ such that
$$
    D:= \SA \cap \left(\Sigma \cup \left\{(x,y) \in \R^{h+k}: 0<\varphi< \frac \pi 4 \right\}\right)
        \subseteq \bigcup_{j=1}^n O_j
$$
and such that $\sum_{j=1}^{n}  \phi_j$ vanishes outside $\SA \setminus D$.

Recalling Lions and Magenes \cite{LionsMagenes}, for any $j \in \{1,\dots , n+m\}$, we now define a family of sufficiently smooth diffeomorphisms $\Phi_j : O_j \to W_j \times (-1,1)$ where $W_j$ is an open set in $\R^{N-1}$, such that
$$
    \Phi_j(B_1^\alpha \cap O_j) = W_j \times (-1,0) \, , \quad
    \Phi_j(\SA \cap O_j) = W_j \times \{0\} \, , \quad
    \Phi_j(O_j \setminus \overline{B_1^\alpha}) = W_j \times (0,1) \, .
$$

In the case $j \in \{1,\dots,n\}$, we have to specialize the choice of the maps $\Phi_j$ being the maps involving the singular set $\SA \cap \Sigma$.

We describe in details one of those maps $\Phi_j$, $j \in \{1,\dots,n\}$, corresponding to case $O_j \cap \SA \subset \SAP$ where $\SAP$ is the subset of $\SA$ of points satisfying $y_k>0$. Proceeding similarly by replacing $y_k$ with the other variables $y_1,\dots,y_{k-1}$ we can cover the entire set $D$.

Let $\eta_{k-1}\in \left(0,\frac \pi 2\right)$ if $k\ge 3$, $\eta_{k-1} \in \left(-\frac \pi 2 , \frac \pi 2\right)$ if $k=2$, the polar coordinate such that if $(x,y)\in \SAP$ and $y'=(y_1,\dots,y_{k-1})$ then $y_k=|y| \cos(\eta_{k-1})$ and $|y'|=|y| \, \sin (\eta_{k-1})$  if $k\ge 3$, and $y_2 = |y| \cos(\eta_1)$ and $y_1 = |y| \sin(\eta_1)$ if $k=2$.

Let us denote by $P_k : \R^{N} \to \R^{N-1}$ the projection $(x,y',y_k) \mapsto (x,y')$ and let $\widetilde D := P_k(D) \cap \widetilde B_1^\alpha$ where $\widetilde B_1^\alpha$ denotes the Grushin unit ball in $\R^{N-1}=\R^h \times \R^{k-1}$.

We can choose $W_j := \widetilde D$ and define over $W_j \times (-1,1)$ the map
\begin{equation} \label{eq:Psi-j}
    \Psi_j(x,y',y_k) := \left(x,y', y_k  + \Gamma(x,y') \right)
    := \left(x,y', y_k + \sqrt{\tfrac{1-|x|^{2(\alpha+1)}-(\alpha+1)^2 \, |y'|^2}{(\alpha+1)^2}} \right)
\end{equation}
and put
\begin{equation} \label{eq:Oj}
O_j := \Psi_j (W_j \times (-1,1)) \quad \text{and} \quad
\Phi_j =\Psi_j^{-1} : O_j \to W_j \times (-1,1) \, .
\end{equation}

%In this way, we can construct a parametric representation of $\SAP$ by using the coordinates $(x,y')$ with $(x,y') \in \widetilde B_1^\alpha$ where $\widetilde B_1^\alpha$ denotes the Grushin unit ball in $\R^{N-1}=\R^h \times \R^{k-1}$. More precisely the parametric representation is given by the map
%$$
%    (x,y') \mapsto \left(x,y',\sqrt{\tfrac{1-|x|^{2(\alpha+1)}-(\alpha+1)^2 \, |y'|^2}{(\alpha+1)^2}}\right)
%     \in \SAP  \, , \qquad (x,y') \in \widetilde B_1^\alpha \, .
%$$
%Let $D$ be the subset of $\SAP$ of points satisfying $0<\varphi<\frac \pi 4$ and let $\widetilde D$ its projection in $\widetilde B_1^\alpha$.

Let $w$ be a function defined over $\SA$. We will say that $w \in  W^{1,2}_\alpha (\SA)$ if
the function
$$
  ( \phi_j \, w)(\Psi_j(x,y',0)) \in  W^{1,2}_\alpha (\widetilde D)  \qquad \text{for any $j \in \{1,\dots,n+m\}$} \, .
$$

\bigskip

{\bf The case $h \ge 2$ and $k=1$.} Identity \eqref{GP} is replaced by
\begin{equation*} % \label{GP-2}
     \Delta_\alpha  u =  \left(\frac{\partial^2 u}{\partial r^2} +\frac{h-1}{r} \, \frac{\partial u}{\partial r} + \frac{1}{r^2} \, \Delta_\theta u\right)
+r^{2\alpha} \, \frac{\partial^2 u}{\partial y^2}
 \end{equation*}
 and \eqref{chvar} remains the same with the only difference that this time $\varphi \in (0,\pi)$ and $t = y$.

 Proceeding as in the case $k\ge 2$ we see that \eqref{GrPo} still holds true with $\mathcal L_{\Theta}$ defined by

 \begin{equation*} %\label{GrPo3}
 \mathcal{L}_\Theta := \frac{\partial^2 }{\partial\varphi^2}
 +\frac{h-1+\alpha}{\alpha+1}\frac{\cos\varphi}{\sin\varphi} \frac{\partial  }{\partial \varphi}+\frac{1}{((\alpha+1)\sin\varphi)^2}\Delta_\theta \, .
\end{equation*}

The construction of $O_j$, $\Psi_j$ and $\Phi_j$ may be done as in the case $k\ge 2$: here $P:\R^N \to \R^{N-1}$ is the projection $(x,y) \mapsto x$, $\widetilde D:=P(D) \cap \widetilde B_1$ where $\widetilde B_1$ is the Euclidean unit ball in $\R^{N-1}$. Then again $W_j := \widetilde D$ and
$$
   \Psi_j(x,y) := \left(x, y + \sqrt{\tfrac{1-|x|^{2(\alpha+1)}}{(\alpha+1)^2}}\right) \, .
$$
The rest, included the definition of $ W^{1,2}_\alpha (\SA)$, is the same as for $k\ge 2$.  Namely, we will say that $w \in  W^{1,2}_\alpha (\SA)$ if
the function
$$
  ( \phi_j \, w)(\Psi_j(x,0)) \in  W^{1,2}(\widetilde D)  \qquad \text{for any $j \in \{1,\dots,n+m\}$} \, .
$$

\bigskip

{\bf The case $h = 1$ and $k\ge 2$.} Identity \eqref{GP} is replaced by
\begin{equation*} %\label{GP-3}
     \Delta_\alpha  u =  \frac{\partial^2 u}{\partial x^2}
+r^{2\alpha}\left(  \frac{\partial^2 u}{\partial t^2} +\frac{k-1}{t} \, \frac{\partial u}{\partial t} + \frac{1}{t^2} \, \Delta_\eta u\right).
 \end{equation*}
 and \eqref{chvar} is replaced by
 \begin{equation}\label{chvar-2}
 (x,t) = \Phi(\rho,\varphi) := \left(\rho (\sin \varphi) |\sin \varphi|^{-\frac{\alpha}{\alpha+1}}, \frac{\rho^{\alpha+1}}{\alpha+1}\cos \varphi\right)
 \end{equation}
 with $\rho \in (0,+\infty)$, $\varphi \in \left(-\frac \pi 2,\frac \pi 2\right)$.

 With this choice of $\rho$ and $\varphi$, we have a new version of \eqref{eq:p-r-t}:
 \begin{align} \label{eq:p-r-t-bis}
  & \frac{\partial u }{\partial x} = \frac{\partial u }{\partial \rho} \left|\sin \varphi\right|^\frac{\alpha}{\alpha+1} \sin \varphi + \frac{\partial u }{\partial \varphi }\frac{\alpha+1}{\rho}|\sin \varphi|^{\frac{\alpha}{\alpha+1}}\cos \varphi \, ,
 \\[10pt]
 \notag &
  \frac{\partial u }{\partial t } = \frac{\partial u }{\partial \rho} \, \rho^{-\alpha}\cos \varphi-\frac{\partial u }{\partial \varphi }\frac{\alpha+1}{\rho^{\alpha+1}}\sin \varphi .
 \end{align}
so that the identity corresponding to \eqref{eq:gradu-gradv} becomes
 \begin{align} \label{eq:gradu-gradv-bis}
    & \nabla_\alpha u \cdot \nabla_\alpha v=\frac{\partial u}{\partial x} \frac{\partial v}{\partial x}+|x|^{2\alpha} \frac{\partial u}{\partial t} \frac{\partial v}{\partial t}+\frac{|x|^{2\alpha}}{t^2}
    \, \nabla_\eta u \cdot \nabla_\eta v \\[10pt]
 \notag   & \qquad = \psi_\alpha  \frac{\partial u}{\partial \rho} \frac{\partial v}{\partial \rho}
  +\frac{(\alpha+1)^2}{\rho^2}\, \psi_\alpha \left[\frac{\partial u}{\partial \varphi} \frac{\partial v}{\partial \varphi}
   + \frac{1}{(\cos \varphi)^2} \, \nabla_\eta u \cdot \nabla_\eta v  \right] \, .
 \end{align}
 Then by \eqref{chvar-2}, \eqref{eq:p-r-t-bis} and \eqref{eq:gradu-gradv-bis}, we deduce that \eqref{GrPo} still  holds true with
 \begin{equation*} %\label{GrPo4}
 \mathcal{L}_\Theta := \frac{\partial^2 }{\partial\varphi^2}
 +\left(\frac{\alpha}{\alpha+1}\frac{\cos\varphi}{\sin\varphi} -(k-1) \frac{\sin \varphi}{\cos \varphi}\right)\frac{\partial  }{\partial \varphi}+\frac{1}{(\cos\varphi)^2} \, \Delta_\eta \, .
\end{equation*}
The construction of $O_j$, $\Psi_j$ and $\Phi_j$ may be done as in the case $h \ge 2$ with $k\ge 2$ or $k=1$: here $P_k:\R^N \to \R^{N-1}$ is the projection $(x,y',y_k) \mapsto (x,y')$, $D$ is defined similarly to the previous cases with the only difference that condition $0<\varphi<\frac \pi 4$ is replaced here by $-\frac \pi 4<\varphi< \frac \pi 4$, $\widetilde D:=P(D) \cap \widetilde B_1$ where $\widetilde B_1$ is the Grushin unit ball in  $\R^{h} \times \R^{k-1}$. Then again $W_j := \widetilde D$ and
$\Psi_j$ is defined as in \eqref{eq:Psi-j}.
The rest, included the definition of $ W^{1,2}_\alpha (\SA)$, is the same as for $h,k\ge 2$.

\bigskip

{\bf The case $h = 1$ and $k = 1$.} In this case, identity \eqref{GP} becomes
\begin{equation*} %\label{GP4}
    \Delta_\alpha  u = \frac{\partial^2 u}{\partial x^2} + |x|^{2\alpha} \, \frac{\partial^2 u}{\partial y^2} \, .
\end{equation*}
We define $\rho$ and $\varphi$ in such a way that
\begin{equation}\label{chvar-4}
   (x,y) = \Phi(\rho,\varphi):= \left(\rho (\sin \varphi) |\sin \varphi|^{-\frac{\alpha}{\alpha+1}}, \frac{\rho^{\alpha+1}}{\alpha+1}\cos \varphi\right)
\end{equation}
with $\rho \in (0,+\infty)$ and $\varphi \in (-\pi,\pi)$.

  With this choice of $\rho$ and $\varphi$, we have also this time the counterpart of \eqref{eq:p-r-t}:
 \begin{align} \label{eq:p-r-t-ter}
  & \frac{\partial u }{\partial x} = \frac{\partial u }{\partial \rho} \left|\sin \varphi\right|^\frac{\alpha}{\alpha+1} \sin \varphi + \frac{\partial u }{\partial \varphi }\frac{\alpha+1}{\rho}|\sin \varphi|^{\frac{\alpha}{\alpha+1}}\cos \varphi \, ,
 \\[10pt]
 \notag &
  \frac{\partial u }{\partial y } = \frac{\partial u }{\partial \rho} \, \rho^{-\alpha}\cos \varphi-\frac{\partial u }{\partial \varphi }\frac{\alpha+1}{\rho^{\alpha+1}}\sin \varphi .
 \end{align}
 Now the identity corresponding to \eqref{eq:gradu-gradv} becomes
 \begin{align} \label{eq:gradu-gradv-ter}
    & \nabla_\alpha u \cdot \nabla_\alpha v=\frac{\partial u}{\partial x} \frac{\partial v}{\partial x}+|x|^{2\alpha} \frac{\partial u}{\partial y} \frac{\partial v}{\partial y} = \psi_\alpha  \frac{\partial u}{\partial \rho} \frac{\partial v}{\partial \rho}
  +\frac{(\alpha+1)^2}{\rho^2}\, \psi_\alpha \frac{\partial u}{\partial \varphi} \frac{\partial v}{\partial \varphi}
 \, .
 \end{align}
 Then by \eqref{chvar-4}, \eqref{eq:p-r-t-ter}, \eqref{eq:gradu-gradv-ter}, we deduce that \eqref{GrPo} still holds true with
 \begin{equation*} %\label{GrPo5}
    \mathcal L_\Theta :=  \frac{\partial^2}{\partial \varphi^2} +\frac{\alpha}{\alpha+1} \frac{\cos \varphi}{\sin \varphi} \, \frac{\partial}{\partial \varphi} \, .
 \end{equation*}
The construction of $O_j$, $\Psi_j$ and $\Phi_j$ may be done as in previous cases: here $P:\R^N \to \R^{N-1}$ is the projection $(x,y) \mapsto x$, $D$ is characterized again by condition $-\frac \pi 4 < \varphi < \frac \pi 4$, $\widetilde D:=P(D) \cap (-1,1)$. Then again $W_j := \widetilde D$ and
$\Psi_j$ is defined as in the case $h\ge 2$ and $k=1$.
The rest, included the definition of $ W^{1,2}_\alpha (\SA)$, is the same as for $h,k\ge 2$.
 Namely, we will say that $w \in  W^{1,2}_\alpha (\SA)$ if
the function
$$
  ( \phi_j \, w)(\Psi_j(x,0)) \in  W^{1,2}(\widetilde D)  \qquad \text{for any $j \in \{1,\dots,n+m\}$} \, .
$$

\medskip

 We are now ready to prove that the weighted Sobolev spaces on $\SA$ we have introduced are Hilbert spaces.
\begin{proposition} \label{p:equivalence} Let $h,k\ge 1$ and let $ W^{1,2}_\alpha (\SA)$ be the weighted Sobolev space defined on the Grushin sphere. Then $ W^{1,2}_\alpha (\SA)$ is a Hilbert space with the scalar product
\begin{equation} \label{eq:norm-b-alpha}
   (w_1,w_2)_{ W^{1,2}_\alpha (\SA)} = b_\alpha(w_1,w_2)+\int_{\SA}  \, w_1 w_2 \, d\Ha
   \, , \qquad \text{for any } w_1,w_2 \in  W^{1,2}_\alpha (\SA)
\end{equation}
where $b_\alpha(\cdot,\cdot)$ is the bilinear form defined by the second line in \eqref{eq:f-b-sferica}.
\end{proposition}

\begin{proof} We only treat in detail the proof in the case $h,k\ge 2$. By definition of $ W^{1,2}_\alpha (\SA)$, we can localize the problem to a single relatively open set $\SA \cap O_j$
with $j \in \{1,\dots,n\}$ being the sets that cover the  degenerate set $\SA \cap \Sigma$.

In order to simplify notations, we still denote by $w$ the function $ \phi_j w$ so that we may assume in the rest of the proof that $w$ is compactly supported in $\SA \cap O_j$. Again, with the only purpose of simplifying notations, we identify $w$ with the function $(x,y') \mapsto w(\Psi_j(x,y',0))$ for $(x,y') \in \widetilde D$.

We have to show the equivalence between the  $ W^{1,2}_\alpha (\widetilde D)$-norm and the $ W^{1,2}_\alpha (\SA)$-norm defined by \eqref{eq:norm-b-alpha}. This can be done by expressing the $ W^{1,2}_\alpha (\widetilde D)$-norm of $w$ in terms of the polar coordinates $\theta_1,\dots,\theta_{h-1},\eta_1,\dots,\eta_{k-1},\varphi$ for $\SAP$.

Let us put $\nabla_\alpha w=(\nabla_x w,|x|^\alpha \nabla_{y'} w)$. Now, putting $r=|x|$ and, $t'=|y'|$ if $k\ge 3$ and $t' = y'$ if $k=2$, we may write
\begin{align} \label{eq:grad-split-1}
   & |\nabla_\alpha w|^2 =\left(\frac{\partial w}{\partial r}\right)^2 + \frac{1}{r^2} \, |\nabla_\theta w|^2
    +r^{2\alpha} \left(\frac{\partial w}{\partial t'}\right)^2 + \frac{r^{2\alpha}}{t'^2} \, |\nabla_{\eta'} w|^2
    \qquad \text{if $k \ge 3$} ,
\end{align}
where $\nabla_{\eta'}$ denotes the gradient on $\mathbb S^{k-2}$ and
\begin{align} \label{eq:grad-split-2}
   & |\nabla_\alpha w|^2 =\left(\frac{\partial w}{\partial r}\right)^2 + \frac{1}{r^2} \, |\nabla_\theta w|^2
    +r^{2\alpha} \left(\frac{\partial w}{\partial t'}\right)^2
    \qquad \text{when $k = 2$} \, .
\end{align}
Proceeding as for the computation of \eqref{eq:p-r-t} and observing that $t'=\frac{1}{\alpha+1} \, \cos \varphi \, \sin(\eta_{k-1})$ and $r=(\sin \varphi)^{\frac{1}{\alpha +1}}$, we may write
\begin{align*}
   & \frac{\partial w}{\partial r}=\frac{(\alpha+1) (\sin \varphi)^{\frac{\alpha}{\alpha +1}}}{\cos \varphi} \,
   \frac{\partial w}{\partial \varphi} + \frac{(\alpha+1)(\sin \varphi)^{\frac{2\alpha+1}{\alpha+1}}
   \sin(\eta_{k-1})}{(\cos \varphi)^2 \, \cos(\eta_{k-1})} \, \frac{\partial w}{\partial \eta_{k-1}}
   % =: a  \, \frac{\partial w}{\partial \varphi} + b \, \frac{\partial w}{\partial \eta_{k-1}}
   \,  ,   \\[10pt]
   & \frac{\partial w}{\partial t'}= \frac{\alpha+1}{\cos \varphi \, \cos(\eta_{k-1})} \,  \frac{\partial w}{\partial \eta_{k-1}}
   % =: r^{-\alpha} c \, \frac{\partial w}{\partial \eta_{k-1}}
   \, ,
\end{align*}
so that
\begin{equation} \label{eq:EST-1}
  \left(\frac{\partial w}{\partial r}\right)^2 +r^{2\alpha} \left(\frac{\partial w}{\partial t'}\right)^2
  = \frac{(\alpha+1)^2 (\sin \varphi)^{\frac{2\alpha}{\alpha +1}}}{(\cos \varphi)^2}
  \left[
  \left( \frac{\partial w}{\partial \varphi} + b \, \frac{\partial w}{\partial \eta_{k-1}}\right)^2
  + c^2 \left( \frac{\partial w}{\partial \eta_{k-1}} \right)^2  \right]
\end{equation}
where we put $b=\frac{\sin \varphi \, \sin(\eta_{k-1})}{\cos \varphi \cos(\eta_{k-1})}$ and $c=\frac{1}{\cos(\eta_{k-1})}$.

It is readily seen that for $0<\varphi<\frac \pi 4$, $b^2 < c^2$ and the matrix
$$
   \begin{pmatrix}
     1-\frac 14 & \ \ b \\
     b          & \ \ \left(1-\frac 14\right) c^2 + b^2
   \end{pmatrix}
$$
is positive definite so that
\begin{equation} \label{eq:EST-2}
   \frac 14 \left[ \left(\frac{\partial w}{\partial \varphi}\right)^2 + c^2 \left( \frac{\partial w}{\partial \eta_{k-1}} \right)^2 \right] \le \left( \frac{\partial w}{\partial \varphi} + b \, \frac{\partial w}{\partial \eta_{k-1}}\right)^2
  + c^2 \left( \frac{\partial w}{\partial \eta_{k-1}} \right)^2
  \le  3 \left[ \left(\frac{\partial w}{\partial \varphi}\right)^2 + c^2 \left( \frac{\partial w}{\partial \eta_{k-1}} \right)^2 \right] \, .
\end{equation}
Combining \eqref{eq:EST-1} and \eqref{eq:EST-2} we obtain
\begin{align} \label{eq:www}
  & \frac{(\alpha+1)^2}{4(\cos \varphi)^2} \, \psi_\alpha \left[ \left(\frac{\partial w}{\partial \varphi}\right)^2 +
    \frac{1}{(\cos(\eta_{k-1}))^2} \, \left( \frac{\partial w}{\partial \eta_{k-1}} \right)^2  \right]  \le    \left(\frac{\partial w}{\partial r}\right)^2 +r^{2\alpha} \left(\frac{\partial w}{\partial t'}\right)^2 \\[10pt]
\notag  & \qquad \le \frac{3(\alpha+1)^2}{(\cos \varphi)^2} \, \psi_\alpha  \left[ \left(\frac{\partial w}{\partial \varphi}\right)^2 + \frac{1}{(\cos(\eta_{k-1}))^2} \, \left( \frac{\partial w}{\partial \eta_{k-1}} \right)^2 \right] \, .
\end{align}
Recalling the definition of $O_j$ in \eqref{eq:Oj}, the fact that the functions of the partition of unity $ \phi_j$ are compactly supported in $O_j$ and that $w$ in \eqref{eq:www} actually stands for $ \phi_j\, w$, we deduce that there exist
two positive constants $C_1$ and $C_2$ independent of $w$ such that
\begin{align} \label{eq:www-bis}
  & C_1 \, \psi_\alpha \left[ \left(\frac{\partial w}{\partial \varphi}\right)^2 +
    \left( \frac{\partial w}{\partial \eta_{k-1}} \right)^2  \right]  \le    \left(\frac{\partial w}{\partial r}\right)^2 +r^{2\alpha} \left(\frac{\partial w}{\partial t'}\right)^2 \le C_2 \, \psi_\alpha  \left[ \left(\frac{\partial w}{\partial \varphi}\right)^2 +  \left( \frac{\partial w}{\partial \eta_{k-1}} \right)^2 \right] \, .
\end{align}
Now, observing that $\frac{1}{r^2} |\nabla_\theta w|^2 = \psi_\alpha \frac{1}{(\sin \varphi)^2} \, |\nabla_\theta w|^2$ and
$|\nabla_\eta w|^2 = \left( \frac{\partial w}{\partial \eta_{k-1}} \right)^2 + \frac{1}{(\sin(\eta_{k-1})^2} \, |\nabla_{\eta'} w|^2$, by \eqref{eq:grad-split-1}, \eqref{eq:grad-split-2} and \eqref{eq:www-bis} we finally obtain, for suitable positive constants $C_1,C_2$, not necessarily the same of \eqref{eq:www-bis},
\begin{align} \label{eq:equivalence}
 & C_1 \psi_\alpha \left[ \left(\frac{\partial w}{\partial \varphi}\right)^2 + \frac{1}{(\alpha+1)^2 (\sin \varphi)^2} \, |\nabla_\theta w|^2 + \frac{1}{(\cos \varphi)^2} \, |\nabla_\eta w|^2 \right]  \le    |\nabla_\alpha w|^2 \\[10pt]
\notag & \qquad \le C_2 \psi_\alpha \left[ \left(\frac{\partial w}{\partial \varphi}\right)^2 + \frac{1}{(\alpha+1)^2 (\sin \varphi)^2} \, |\nabla_\theta w|^2 + \frac{1}{(\cos \varphi)^2} \, |\nabla_\eta w|^2 \right] \, .
\end{align}
Recalling again that the support of $ \phi_j w$ is contained in $O_j$ we have that
\begin{equation} \label{eq:measures}
     D_1 \, dx dy' \, \le d\Ha = \sqrt{1+|\nabla \Gamma(x,y')|^2}\, dxdy' \le  D_2 \, dxdy'
\end{equation}
over $\SA \cap O_j$, for some positive constants $D_1$ and $D_2$ independent of $w$. We recall that $\Gamma$ is the function introduced in \eqref{eq:Psi-j}.

Integrating \eqref{eq:equivalence} and taking into account \eqref{eq:measures}, we immediately see that the norm induced by \eqref{eq:norm-b-alpha} is equivalent to the norm coming from $ W^{1,2}_\alpha (\widetilde D)$.

The same kind of approach may be implemented in all the other cases $h\ge 2$ and $k=1$, $h=1$ and $k\ge 2$, $h=k=1$.
\end{proof}

%We observe that it is readily verified that
%\begin{equation*}
%      \frac{1}{(\alpha+1)^2} \, \int_{\partial B_1^\alpha}  \psi_\alpha |\nabla_\Theta w|^2 d\Ha
%       \le  b_\alpha(w,w)
%        \le  \int_{\partial B_1^\alpha}  \psi_\alpha |\nabla_\Theta w|^2 d\Ha
%        \qquad \text{for any } w\in  W^{1,2}_\alpha (\partial B_1^\alpha)
%\end{equation*}
%where $\nabla_\Theta$ denotes the Riemannian gradient on $\partial B_1^\alpha$, i.e.
%\begin{equation*}
%\end{equation*}

\begin{corollary} \label{c:equivalence} Let $h,k \ge 1$. Then the space $ W^{1,2}_\alpha (\SA)$ endowed with norm induced by \eqref{eq:norm-b-alpha} is compactly embedded $L^2(\SA)$.
  \end{corollary}

\begin{proof} Using a partition of unity subject to an open cover $\{O_j\}$ of $\SA$, as in the proof of Proposition \ref{p:equivalence}, we can identify locally a bounded sequence in   $ W^{1,2}_\alpha (\SA)$ with a bounded sequence in $ W^{1,2}_\alpha (\widetilde D)$ for some bounded domain $\widetilde D \subset \R^{N-1}$. Now, we know that $ W^{1,2}_\alpha (\widetilde D)$ is compactly embedded into $L^2(\widetilde D)$, see Proposition \ref{p:embedding} in Section \ref{s:regularity}, so that our sequence admits a subsequence converging in $L^2(\widetilde D)$. Coming back to $\SA$, we recover the strong convergence of this subsequence in $L^2(\SA \cap O_j)$ for any $j$. This proves strong convergence in $L^2(\SA)$.
\end{proof}

\bigskip

Thanks to Proposition \ref{p:equivalence} we can now give a rigorous variational formulation for the eigenvalue problem associated with the operator $-\mathcal L_\Theta$: we say that $\mu$ is an eigenvalue of $-\mathcal L_\Theta$ corresponding to some eigenfunction $\Psi \in  W^{1,2}_\alpha (\SA)\setminus \{0\}$ if
\begin{equation} \label{eq:b-alpha-eig}
   b_\alpha (\Psi,w)=\mu \int_{\SA} \psi_\alpha \Psi \, w \, d\Ha \qquad
   \text{for any } w\in  W^{1,2}_\alpha (\SA) \, .
\end{equation}
By Corollary \ref{c:equivalence} and standard results for linear self-adjoint compact operators, we deduce that the eigenvalue problem \eqref{eq:b-alpha-eig} admits a sequence of nonnegative eigenvalues diverging to $+\infty$ with finite multiplicity. Moreover, it can be proved that $ W^{1,2}_\alpha (\SA)$ can be endowed with an equivalent scalar product defined by
$$
     b_\alpha(w_1,w_2) + \int_{\SA} \psi_\alpha \, w_1 w_2 \, d\Ha
$$
(see the argument used in the proof of Proposition \ref{p:Poincare}) and, with respect to this scalar product, $ W^{1,2}_\alpha (\SA)$ admits a Hilbertian basis of eigenfunctions.

Throughout the paper we use the following notations for the eigenvalues and the eigenfunctions of $-\mathcal L_\Theta$
\begin{equation} \label{eq:eig-L-Theta}
   0 = \mu_0 < \mu_1 \le \mu_2 \le \dots \le \mu_n \le \dots \, ,
\end{equation}
where each eigenvalue appears in the sequence as many times as its multiplicity, and
\begin{equation}\label{eq:eigf-L-Theta}
   \omega_0 \, , \omega_1 \, \dots \, , \omega_j \, , \dots
\end{equation}
a corresponding sequence of orthonormal eigenfunctions with respect to the scalar product of $L^2_\alpha(\SA)$ given by $(w_1,w_2) \mapsto  \int_{\SA} \psi_\alpha \, w_1 w_2 \, d\Ha$.

\section{The monotonicity formula} \label{s:monotonicity}

 In the present section we present the Almgren-type monotonicity argument and we prove the main results of the paper, namely Theorem \ref{t:main} and Corollaries \ref{c:1}, \ref{c:2}, and \ref{c:positive-measure}.

Let $\Omega \subseteq \R^{h+k}$ be a bounded connected domain that, without losing generality, contains the origin. Let $u$ be a weak solution of \eqref{eq:u} in sense of the variational formulation \eqref{eq:var}. Let us define the functions
\begin{align} \label{eq:D}
   & D(r) := r^{2-Q} \int_{B_r^\alpha} \left(|\nabla_\alpha u|^2-V(x,y)u^2\right) dxdy \, , \\[7pt]
  \label{eq:H} & H(r) :=r^{1-Q} \int_{\partial B_r^\alpha} \psi_\alpha \, u^2 \, d\Ha \, ,
\end{align}
for any $r\in (0,R)$ where $R>0$ is the maximal value for which $B_R^\alpha\subseteq \Omega$ and where $\psi_\alpha$ is the function defined in \eqref{eq:psi-alpha}.  The functions $D$ and $H$ are respectively called the energy and height functions.

Both functions $D$ and $H$ are well defined if we assume that $V$ satisfies \eqref{eq:V}.  Indeed, for $D$ one can check by the Sobolev embedding in Proposition \ref{p:embedding} and by  \eqref{eq:V} that $Vu^2\in  L^1_{\rm loc}(\Omega)$.  Instead, in the definition of $H$, we only need $u\in  W^{1,2}_\alpha (B_r^\alpha)$ since the trace of a function  $u\in   W^{1,2}_{\alpha,\, {\rm loc}}(\Omega)$, which in \eqref{eq:H} is still denoted by $u$, belongs to the weighted space $L^2_\alpha(\partial B_r^\alpha)$ as one can see from Proposition \ref{p:trace-alpha}.  Here  $L^2_\alpha(\partial B_r^\alpha)$ is the space of  square integrable function on $\partial B_r^\alpha$ with respect to the measure $|x|^{2\alpha}d\mathcal{H}^{N-1}$   (see Subsection \ref{s:properties-Sobolev} for the general definition of the space $L^p_\alpha(\partial \Omega)$).

We start by computing the derivative of $H$.

\begin{lemma} \label{l:H'} Let $u \in   W^{1,2}_{\alpha,\, {\rm loc}}(\Omega)$ be a weak solution of \eqref{eq:u} with $V$ satisfying \eqref{eq:V} and let $H$ be the corresponding function defined in \eqref{eq:H}. Then $H\in W^{1,1}_{{\rm loc}}(0,R)$  and
 for a.e. $r \in (0,R)$:
\begin{equation} \label{eq:H'}
    H'(r)=2r^{1-Q} \int_{\partial B_r^\alpha} \frac{\psi_\alpha}{d_\alpha} \, u X_G u \, d\Ha \, .
  % 2r^{1-Q} \int_{\partial B_r^\alpha} u (A_\alpha \nabla u)\cdot \nu \, d\mathcal H^{N-1}
\end{equation}

\end{lemma}

\begin{proof} Since $u \in   W^{1,2}_{\alpha,\, {\rm loc}}(\Omega)$ then we readily see that the function $\psi_\alpha u X_G u \in L^1_{{\rm loc} }(\Omega\setminus \{0\})$ and hence the function $r \mapsto \int_{\partial B_r^\alpha} \frac{\psi_\alpha}{d_\alpha} \, u X_G u  \, d\Ha$ is defined a.e. in $(0,R)$. Then for a.e. $r\in (0,R)$, by Lemma \ref{l:rescale-partial}, we may write
\begin{align} \label{eq:ae-der}
   & 2r^{1-Q} \int_{\partial B_r^\alpha} \frac{\psi_\alpha}{d_\alpha} \, u X_G u \, d\Ha
     = \int_{\partial B_1^\alpha} |x|^{2\alpha} \, 2 u(rx,r^{\alpha+1}y) \nabla u(rx,r^{\alpha+1}y) \cdot (x,(\alpha+1)r^\alpha y) \, d\Ha \, .
\end{align}
Then we integrate \eqref{eq:ae-der} over a generic interval $(r_1,r_2)$ with $r_2>r_1>0$ to obtain after Fubini-Tonelli Theorem
\begin{align*} % \label{eq:ae-der}
   & \int_{r_1}^{r_2} \left(2r^{1-Q} \int_{\partial B_r^\alpha} \frac{\psi_\alpha}{d_\alpha} \, u X_G u \, d\Ha\right) d\Ha
   =\int_{\partial B_1^\alpha} |x|^{2\alpha}
     \left(\int_{r_1}^{r_2} \frac{\partial}{\partial r}\left(u(rx,r^{\alpha+1}y)\right)^2 \right) d\Ha  \\[15pt]
   &   = \int_{\partial B_1^\alpha} |x|^{2\alpha}
     \left[\left(u(r_2x,r_2^{\alpha+1}y)\right)^2-\left(u(r_1x,r_1^{\alpha+1}y)\right)^2  \right] d\Ha  \\[15pt]
   & = r_2^{1-Q} \int_{\partial B_{r_2}^\alpha} r_2^{-2\alpha}|x|^{2\alpha} u^2 \, d\Ha
      -r_1^{1-Q} \int_{\partial B_{r_1}^\alpha} r_1^{-2\alpha}|x|^{2\alpha} u^2 \, d\Ha
      =H(r_2)-H(r_1) \, .
\end{align*}
This readily shows that $H\in W^{1,1}_{{\rm loc}}(0,R)$ and moreover its derivative is given a.e. by \eqref{eq:H'}.
\end{proof}

We also need to represent the derivative of the function $D$. In order to perform the computation of $D'$ we first need  the following Pohozaev-type identity.  For the clarity of exposition we postpone the proof to the Appendix.

\begin{proposition} \label{p:Pohozaev} (Pohozaev-type identity) Let $u\in   W^{1,2}_{\alpha,\, {\rm loc}}(\Omega)$ be a weak solution of \eqref{eq:u} with $V$ satisfying \eqref{eq:V}. Then $u\in  W^{2,2}_{\alpha,{\rm loc}} (\Omega)$, $\nabla u\in L^2_{{\rm loc}}(\Omega;\R^N)$ and the trace on $\partial B_r^\alpha$ of first order derivatives of $u$ belong to $L^2(\partial B_r^\alpha)$ for any $r>0$ such that $\overline{B_r^\alpha}\subset \Omega$. Moreover for a.e. $r\in (0,R)$ we have
\begin{align*}
   & -\frac{Q-2}{2} \int_{B_r^\alpha} |\nabla_\alpha u|^2 \, dxdy+\frac{r}{2} \int_{\partial B_r^\alpha} |\nabla_\alpha u|^2 \, d\Ha - \frac{1}{r} \int_{\partial B_r^\alpha} \psi_\alpha (X_G u)^2 \, d\Ha \\[10pt]
   & \quad = \frac{r}2 \int_{\partial B_r^\alpha} Vu^2 \, d\Ha  -\frac{Q}{2} \int_{B_r^\alpha} Vu^2 \, dxdy
   -\frac 12 \int_{B_r^\alpha} (\nabla V\cdot X_G) u^2 \, dxdy \, .
\end{align*}
\end{proposition}

\begin{remark} We observe that the regularity result stated in Proposition \ref{p:Pohozaev} combined with the integration by parts formula stated in Proposition \ref{p:int-parts} in the Appendix and Lemma \ref{l:H'}, yields
\begin{equation} \label{eq:D-H'}
   D(r)=\frac{r}{2} \, H'(r) \, .
\end{equation}
The same regularity result stated in Proposition \ref{p:Pohozaev} shows that the integral in the right hand side of \eqref{eq:H'}   can be also interpreted in the sense of traces.
\end{remark}

\begin{lemma} Let $u \in   W^{1,2}_{\alpha,\, {\rm loc}}(\Omega)$ be a weak solution of \eqref{eq:u} with $V$ satisfying \eqref{eq:V} and let $D$ be the corresponding function defined in \eqref{eq:D}. Then $D \in W^{1,1}_{{\rm loc}}(0,R)$ and for a.e. $r\in (0,R)$ we have that
\begin{align} \label{eq:D'}
   & D'(r)=-r^{1-Q} \int_{B_r^\alpha} \left(2V+\nabla V\cdot X_G\right) u^2 \, dxdy+2r^{-Q}
   \int_{\partial B_r^\alpha} \psi_\alpha (X_G u)^2 \, d\Ha \, .
\end{align}

\end{lemma}

\begin{proof} By Proposition \ref{p:Pohozaev} we easily deduce that   for a.e. $r\in (0,R)$:
  \begin{align} \label{eq:D'-1}
     & r \int_{\partial B_r^\alpha} \left(|\nabla_\alpha u|^2-Vu^2\right) d\Ha=(Q-2) \int_{B_r^\alpha} |\nabla_\alpha u|^2 \, dxdy-Q \int_{B_r^\alpha} Vu^2 \, dxdy \\[7pt]
     \notag & \qquad -\int_{B_r^\alpha} (\nabla V\cdot X_G) u^2 \, dxdy+\frac 2r \int_{\partial B_r^\alpha} \psi_\alpha (X_G u)^2 \, d\Ha \, .
  \end{align}
  On the other hand, for a.e. $r\in (0,R)$ we have that
  \begin{align} \label{eq:D'-2}
     & D'(r)=(2-Q) r^{1-Q} \int_{B_r^\alpha} \left(|\nabla_\alpha u|^2 -Vu^2\right)dxdy
     +r^{2-Q} \int_{\partial B_r^\alpha} \left(|\nabla_\alpha u|^2 -Vu^2\right)d\Ha
  \end{align}
  thanks to Lemma \ref{l:der-B} in the Appendix.

  Combining \eqref{eq:D'-1} and \eqref{eq:D'-2} we arrive to the conclusion of the proof.
\end{proof}

% UTILIZZIAMO LA SEGUENTE NOTAZIONE $f(u)=au$.

 We will also need the following Poincar\'e-type inequality in $B_r^\alpha$.

\begin{proposition} \label{p:Poincare}  Let $r>0$. There exists a constant $C(h,k,\alpha)$ depending only on $h$, $k$ and $\alpha$ such that
  \begin{equation}\label{eq:Poincare}
     \frac{1}{r^2} \int_{B_r^\alpha} v^2 \, dxdy\le C(h,k,\alpha) \left(\int_{B_r^\alpha} |\nabla_\alpha v|^2 \, dxdy
     +\frac{1}{r}\int_{\partial B_r^\alpha} \psi_\alpha v^2 \, d\Ha \right)
\end{equation}
for any $v\in  W^{1,2}_\alpha (B_r^\alpha)$.
\end{proposition}

\begin{proof} For simplicity let start by proving \eqref{eq:Poincare} for $r=1$. It is enough to prove that the original norm of $ W^{1,2}_\alpha (B_1^\alpha)$ is equivalent to the norm defined by
$$
    \|v\|_\alpha := \left(\int_{B_1^\alpha} |\nabla_\alpha v|^2 \, dxdy
     + \int_{\partial B_1^\alpha} \psi_\alpha v^2 \, d\Ha \right)^{\frac 12} \qquad \text{for any } v \in  W^{1,2}_\alpha (B_1^\alpha)
$$
which means that there exist two positive constants $C_1,C_2$ such that
$$
   \|v\|_\alpha \le C_1  \|v\|_{ W^{1,2}_\alpha (B_1^\alpha)} \, , \qquad
    \|v\|_{ W^{1,2}_\alpha (B_1^\alpha)} \le C_2 \|v\|_\alpha  \qquad \text{for any } v \in  W^{1,2}_\alpha (B_1^\alpha)  \, .
$$
The first estimate is a consequence of the continuity of the trace map, see Proposition \ref{p:trace-alpha} in Section \ref{s:regularity}.
For proving the second estimate we can proceed by contradiction and find a sequence $\{v_n\} \subset  W^{1,2}_\alpha (B_1^\alpha)$
with $\|v_n\|_{ W^{1,2}_\alpha (B_1^\alpha)}=1$ such that $\|v_n\|_\alpha \to 0$ as $n \to +\infty$. Up to subsequences, we can assume that there exists $v \in   W^{1,2}_\alpha (B_1^\alpha)$ such that $v_n \rightharpoonup v$ weakly in $ W^{1,2}_\alpha (B_1^\alpha)$ and by the compact embedding $ W^{1,2}_\alpha (B_1^\alpha) \subset L^2(B_1^\alpha)$, see Proposition \ref{p:embedding} in Section \ref{s:regularity}, $v_n \to v$ strongly in $L^2(B_1^\alpha)$. Moreover, by weak convergence we have that $\int_{B_1^\alpha} \nabla_\alpha v_n \cdot \nabla_\alpha v \, dxdy \to \int_{B_1^\alpha} |\nabla_\alpha v|^2 \, dxdy$. On the other hand we also have
\begin{equation*}
   \left|\int_{B_1^\alpha} \! \nabla_\alpha v_n \!\cdot\! \nabla_\alpha v \, dxdy \right| \leq
      \left(\int_{B_1^\alpha} |\nabla_\alpha v_n|^2 \, dxdy\right)^{\frac 12}
        \left(\int_{B_1^\alpha} |\nabla_\alpha v|^2 \, dxdy\right)^{\frac 12} \! \le  \|v_n\|_\alpha \left(\int_{B_1^\alpha} |\nabla_\alpha v|^2 \, dxdy\right)^{\frac 12}
\end{equation*}
which converges to zero thus proving that $\int_{B_1^\alpha} |\nabla_\alpha v|^2 \, dxdy=0$ so that $v$ is constant in $B_1^\alpha$. On the other hand, the continuity of the trace map also implies that $v_n \rightharpoonup v$ weakly in $L^2_\alpha(\partial B_1^\alpha)$ but $v_n \to 0$ strongly in $L^2_\alpha(\partial B_1^\alpha)$ since $\|v_n\|_{L^2_\alpha(\partial B_1^\alpha)} \le \|v_n\|_\alpha \to 0$. This implies that the constant function $v$ vanishes on $\partial B_1^\alpha$ which means that actually $v=0$ in $B_1^\alpha$.

In particular $v_n \to 0$ in $L^2(B_1^\alpha)$ and recalling that  $\int_{B_1^\alpha} |\nabla_\alpha v_n|^2 dxdy \le \|v_n\|_\alpha^2 \to 0$, this implies $v_n \to 0$ strongly in $ W^{1,2}_\alpha (B_1^\alpha)$, contradicting $\|v_n\|_{ W^{1,2}_\alpha (B_1^\alpha)}=1$.
The equivalence of the two norms readily implies the validity of \eqref{eq:Poincare} in the case $r=1$.

For proving \eqref{eq:Poincare} for any arbitrary radius $r>0$, letting $v\in  W^{1,2}_\alpha (B_r^\alpha)$, we consider the rescaled function $v_r(x,y)=v(rx,r^{\alpha+1}y)$ defined in $B_1^\alpha$ and we apply to it \eqref{eq:Poincare} with $r=1$.
  Then applying Lemma \ref{l:change-B} and Lemma \ref{l:rescale-partial} we arrive to the conclusion in the general case.
\end{proof}

As a consequence of Proposition \ref{p:Poincare} one can also prove the following

\begin{corollary} \label{c:Poincare}  Let $r>0$. Let $1\le q<\infty$ be such that $q\le \frac{2Q}{Q-2}$ when $Q>2$. Then
there exists a constant $C(h,k,\alpha,q)$ depending only on $h$, $k$, $\alpha$ and $q$ such that
  \begin{equation}\label{eq:Poincare-bis}
  \left( \int_{B_r^\alpha} |v|^q \, dxdy\right)^{\frac 2q} \le C(h,k,\alpha,q) \, r^{\frac{2-q}{q}\, Q+2}
    \left(\int_{B_r^\alpha} |\nabla_\alpha v|^2 \, dxdy
     +\frac{1}{r}\int_{\partial B_r^\alpha} \psi_\alpha v^2 \, d\Ha \right)
\end{equation}
for any $v\in  W^{1,2}_\alpha (B_r^\alpha)$.
\end{corollary}

\begin{proof} By Proposition \ref{p:embedding} applied to the case $\Omega=B_1^\alpha$ combined with a rescaling argument yields
\begin{equation} \label{eq:Poincare-ter}
  \left( \int_{B_r^\alpha} |v|^q \, dxdy\right)^{\frac 2q} \le C_1(h,k,\alpha,q) \, r^{\frac{2-q}{q}\, Q+2}
    \left(\int_{B_r^\alpha} |\nabla_\alpha v|^2 \, dxdy
     +\frac{1}{r^2}\int_{B_r^\alpha} v^2 \, dxdy \right) \, .
\end{equation}
Estimate \eqref{eq:Poincare-bis} then follows combining \eqref{eq:Poincare-ter} with \eqref{eq:Poincare}.
\end{proof}

\begin{lemma} \label{l:H>0} Let $u \in    W^{1,2}_{\alpha,\, {\rm loc}} (\Omega)$ be a nontrivial weak solution of \eqref{eq:u} with $V$ satisfying \eqref{eq:V} and let $H$ be the corresponding function defined in \eqref{eq:H}. Then there exist $\bar r\in (0,R)$ such that $H(r)>0$ for any $r\in (0,\bar r)$.
\end{lemma}

\begin{proof} By \eqref{eq:V}, H\"older inequality, Corollary \ref{c:Poincare}, we deduce that  for all $r \in (0,R):$
\begin{align} \label{eq:est-r0}
    \int_{B_r^\alpha} \left[|\nabla_\alpha u|^2-Vu^2\right] dxdy & \ge \left[1-\|V\|_{L^\sigma(B_r^\alpha)} \,
   C\left(h,k,\alpha,\tfrac{2\sigma}{\sigma-1}\right) r^{\frac{2\sigma-Q}{\sigma}} \right]\int_{B_r^\alpha} |\nabla_\alpha u|^2\,  dxdy  \\[7pt]
\notag   & \qquad -\frac 1r \, \|V\|_{L^\sigma(B_r^\alpha)} \,
   C\left(h,k,\alpha,\tfrac{2\sigma}{\sigma-1}\right) r^{\frac{2\sigma-Q}{\sigma}} \int_{\partial B_r^\alpha} \psi_\alpha  u^2 \, d\Ha \, .
\end{align}
Suppose by contradiction that for any $\bar r\in (0,R)$ the exists $r_0 \in (0,\bar r)$ such that $H(r_0)=0$.
Since by \eqref{eq:V} we know that $\sigma>\frac{Q}{2}$, choosing $\bar r$ small enough such that
\begin{equation} \label{eq:bar-r}
   1-\|V\|_{L^\sigma(B_{\bar r}^\alpha)} \,
   C\left(h,k,\alpha,\tfrac{2\sigma}{\sigma-1}\right) \bar r^{\, \frac{2\sigma-Q}{\sigma}}>0 \, ,
\end{equation}
applying \eqref{eq:est-r0} with $r=r_0<\bar r$, observing that by \eqref{eq:H'} we have $H'(r_0)=0$ being the trace of $u$ zero on $\partial B_{r_0}$, by \eqref{eq:D-H'} and \eqref{eq:D} we obtain
\begin{align*}
   & 0=\int_{B_{r_0}^\alpha} \left[|\nabla_\alpha u|^2-Vu^2\right] dxdy
   \ge \left[1-\|V\|_{L^\sigma(B_{r_0}^\alpha)} \,
   C\left(h,k,\alpha,\tfrac{2\sigma}{\sigma-1}\right) r_0^{\frac{2\sigma-Q}{\sigma}} \right]\int_{B_{r_0}^\alpha} |\nabla_\alpha u|^2\,  dxdy \, ,
\end{align*}
which implies that $u$ is constant in $B_{r_0}^\alpha$ and since its trace on $\partial B_{r_0}^\alpha$ is zero then $u\equiv 0$ in $B_{r_0}^\alpha$. Applying a classical unique continuation principle for elliptic operators, see Sogge \cite{So}, we deduce that $u\equiv 0$ on any open set whose closure is contained in $\Omega \setminus \Sigma$. This proves that $u=0$ a.e. in $\Omega$, a contradiction.
\end{proof}

Thanks to Lemma \ref{l:H>0} we may define the Almgren-type function corresponding to nontrivial solutions of equation \eqref{eq:V}:
\begin{equation}\label{eq:N(r)}
    \mathcal N(r) := \frac{D(r)}{H(r)}  \qquad \text{for any } r\in (0,\bar r) \, .
\end{equation}
 We now compute the derivative of $\mathcal{N}$ and prove some estimates.

\begin{lemma} \label{l:N'} Let $u \in   W^{1,2}_{\alpha,\, {\rm loc}}(\Omega)$ be a nontrivial weak solution of \eqref{eq:u} with $V$ satisfying \eqref{eq:V}, let $\mathcal N$ be defined by \eqref{eq:N(r)} with $\bar r$ as in \eqref{eq:bar-r}. Then for a.e. $r\in (0,\bar r)$ we have
\begin{align} \label{eq:N'}
   \mathcal N'(r)&=2\frac{r^{1-2Q}}{(H(r))^2} \, \left\{ \int_{\partial B_r^\alpha} \psi_\alpha \left({X_G u}\right)^2 d\Ha
  \cdot \int_{\partial B_r^\alpha} \psi_\alpha  u^2 \, d\Ha-\left( \int_{\partial B_r^\alpha} \psi_\alpha u {X_G u} \, d\Ha \right)^2 \right\} \\[10pt]
   \notag  & \qquad -\frac{r^{1-Q}}{H(r)}  \int_{B_r^\alpha} (2V + \nabla V\cdot X_G ) u^2 \, dxdy \, .
\end{align}
Moreover, we also have
\begin{equation} \label{eq:N>-1}
  \mathcal N(r)>-1 \qquad \text{for any } r\in (0,\bar r)
\end{equation}
and
\begin{equation}  \label{eq:N'ge}
  \mathcal N'(r)\ge -K(h,k,\alpha,\sigma,\bar r,V) r^{-1+\frac{2\sigma-Q}{\sigma}} \, [\mathcal N(r)+1] \qquad  \text{for a.e. } \, r \in (0,\bar r) 
  %%%  -4C(N,\alpha) \|V\|_{L^\infty(\Omega)} \bar r[\mathcal N(r)+1]\qquad \text{for any } r\in (0,\bar r)
\end{equation}
for some positive constant $K(h,k,\alpha,\sigma,\bar r,V)$ depending only on $h$, $k$, $\alpha$, $\sigma$, $\bar r$ and $V$.
\end{lemma}

\begin{proof} The proof of \eqref{eq:N'} follows by \eqref{eq:H'}, \eqref{eq:D-H'}, \eqref{eq:D'} and the definition of $\mathcal N$.

Let us proceed with the proof of \eqref{eq:N>-1}. Let $\bar r$ be as in \eqref{eq:bar-r}. Then by \eqref{eq:est-r0} and \eqref{eq:bar-r} we have that
\begin{align} \label{eq:D+H}
    D(r)+H(r) & \ge r^{2-Q} \left[ \int_{B_r^\alpha} \left(|\nabla_\alpha u|^2-Vu^2\right)dxdy \right.  \\[7pt]
    \notag  & \qquad \left. +\frac 1r \, \|V\|_{L^\sigma(B_r^\alpha)} \,
   C\left(h,k,\alpha,\tfrac{2\sigma}{\sigma-1}\right) r^{\frac{2\sigma-Q}{\sigma}} \int_{\partial B_r^\alpha} \psi_\alpha u^2 \, d\Ha \right]  \\[7pt]
  \notag & \ge r^{2-Q} \left[1-\|V\|_{L^\sigma(B_r^\alpha)} \,
   C\left(h,k,\alpha,\tfrac{2\sigma}{\sigma-1}\right) r^{\frac{2\sigma-Q}{\sigma}} \right]\int_{B_r^\alpha} |\nabla_\alpha u|^2\,  dxdy   > 0
\end{align}
for any $r \in (0,\bar r)$. Dividing the last inequality by $H(r)$, \eqref{eq:N>-1} readily follows.

Let us proceed with the proof of \eqref{eq:N'ge}. By \eqref{eq:V}, H\"older inequality, Corollary \ref{c:Poincare} and \eqref{eq:D+H}, we infer that there exists a constant $C(h,k,\alpha,\sigma,\bar r)$ depending only on $h$, $k$, $\alpha$, $\sigma$ and $\bar r$ such that
\begin{align} \label{eq:pr-est}
   & \left|\int_{B_r^\alpha} (2V + \nabla V\cdot X_G ) u^2 \, dxdy \right| \le
   C\left(h,k,\alpha,\sigma,\bar r\right) r^{\frac{2\sigma-Q}{\sigma}} \, \|V\|_{W^{1,\sigma}(B_r^\alpha)}
   \, r^{Q-2} \, [D(r)+H(r)]
\end{align}
for any $r\in (0,\bar r)$. Combining \eqref{eq:pr-est} with \eqref{eq:N'} and using the H\"older inequality, we obtain
\begin{align} \label{eq:N'-basso}
   & \mathcal N'(r)\ge -\frac{r^{1-Q}}{H(r)}  \int_{B_r^\alpha} (2V + \nabla V\cdot X_G ) u^2 \, dxdy
   \ge -K(h,k,\alpha,\sigma,\bar r,V) r^{-1+\frac{2\sigma-Q}{\sigma}}[\mathcal N(r)+1]
\end{align}
for any $r\in (0,\bar r)$ where we put $K(h,k,\alpha,\sigma,\bar r,V)=C\left(h,k,\alpha,\sigma,\bar r\right) \, \|V\|_{W^{1,\sigma}(B_{\bar r}^\alpha)}$. This completes the proof of \eqref{eq:N'ge}.
\end{proof}

\bigskip

\begin{lemma} \label{l:bound} Suppose that all the assumptions of Lemma \ref{l:N'} hold true. Then $\mathcal N$ is bounded as $r\to 0^+$.
\end{lemma}

\begin{proof} Let us consider \eqref{eq:N'ge} and for simplicity denote by $K$ the constant $K(h,k,\alpha,\sigma,\bar r,V)$ introduced in the statement of Lemma \ref{l:N'} and let us put $\delta=\frac{2\sigma-Q}{\sigma}>0$. Integrating \eqref{eq:N'ge} in $(r,\bar r)$ we obtain
\begin{equation*}
  \log(\mathcal N(\bar r)+1)-\log(\mathcal N(r)+1)\ge \tfrac{K}{\delta} \left(r^\delta-\bar r^{\, \delta}\right) \qquad \text{for any } r\in (0,\bar r) \, .
\end{equation*}
This implies
\begin{equation*}
   \mathcal N(r)+1\le [\mathcal N(\bar r)+1] e^{\tfrac{K}{\delta} \left(\bar r^{\, \delta}-r^\delta\right)} \le [\mathcal N(\bar r)+1] e^{\tfrac{K\bar r^{\, \delta}}{\delta}} \qquad \text{for any } r\in (0,\bar r) \, .
\end{equation*}
This proves that $\mathcal N$ is bounded from above in a right neighborhood of $0$.
The fact that $\mathcal N$ is also bounded from below follows by \eqref{eq:N>-1}.
\end{proof}

\begin{lemma} \label{l:def-ell} Suppose that all the assumption of Lemma \ref{l:N'} hold true. Then $\mathcal N$ admits a finite limit as $r\to 0^+$ which we denote by $\ell$.
\end{lemma}

\begin{proof} By Lemma \ref{l:bound} we know that $\mathcal N$ is bounded and by \eqref{eq:N'} and \eqref{eq:N'-basso}, we infer that $\mathcal N'$ is the sum of a nonnegative function $f(r)$  (the first line in the right hand side of  \eqref{eq:N'}) and of a function $g(r)$  (the second line in the right hand side of  \eqref{eq:N'}) which is integrable in $(0,\bar r)$.

Hence, the integral functions $\int_r^{\bar r} f(t)\, dt$ and $\int_r^{\bar r} g(t)\, dt$ both admit a limit as $r\to 0^+$. Therefore, $\int_r^{\bar r} \mathcal N'(t) \, dt$ also admits a limit as $r\to 0^+$. This shows that $\mathcal N$ admits a limit as $r \to 0^+$. This limit is necessarily finite thanks to Lemma \ref{l:bound}.
\end{proof}

Proceeding as in the following papers \cite{FFFN,FF1,FFT1,FFT2,FFT3} with suitable adaptation to  the  present problem one can prove the following.

\begin{lemma} \label{l:H-2ell}
 Suppose that all the  assumptions of Lemma \ref{l:N'} hold true and let $\ell$ be as in Lemma \ref{l:def-ell}. Then the following conclusions hold true:
 \begin{equation} \label{eq:exist-lim-H}
    \lim_{r \to 0^+} r^{-2\ell} H(r)
 \end{equation}
 exists and it is finite and moreover for any $\eta>0$ there exist $K(\eta)>0$ and $r_\eta>0$, both depending on $\eta$, such that
 \begin{equation} \label{eq:est-H-ge}
      H(r) \ge K(\eta) \, r^{2\ell +\eta} \qquad \text{for any } r \in (0,r_\eta) \, .
 \end{equation}
\end{lemma}

\begin{proof} We first observe that by the proof of Lemma \ref{l:def-ell}, we have that $\mathcal N' \in L^1(0,\bar r)$ and
by \eqref{eq:N'-basso} we have that
\begin{equation} \label{eq:est-present-article}
   \mathcal N(r)-\ell=\int_0^r \mathcal N'(t) \, dt\ge -C_1 \, r^\delta \qquad \text{for any } r\in (0,\bar r)
\end{equation}
with $\delta=\frac{2\sigma-Q}{\sigma}$ accordingly with the notation of Lemma \ref{l:bound}.

Once we have \eqref{eq:est-present-article}, one can follow exactly the same argument used in the proof of \cite[Lemma 5.6]{FFT2} to complete the proof of \eqref{eq:exist-lim-H} and \eqref{eq:est-H-ge}.
\end{proof}

We now perform a blow-up argument  of a  nontrivial solutions $u$ of equation \eqref{eq:V} by defining the following family of rescaled functions:
\begin{equation} \label{eq:u-eps}
u_\eps (x,y) = \dfrac{u(\eps x,\eps^{\alpha+1}y)}{\sqrt{H(\eps)}},  \qquad  \eps >0\, .
% = \dfrac{u(\eps x,\eps^{\alpha+1}y)}{\sqrt{\eps^{1-Q}\int_{\partial B_\eps^\alpha} {u_\eps}^2\psi_\alpha\,d\Ha}}.
\end{equation}
By \eqref{eq:psi-alpha} and Lemma \ref{l:rescale-partial} we readily see that
\begin{align} \label{eq:normalization}
    \int_{\partial B_1^\alpha} \psi_\alpha u_\eps^2 \, d\Ha=1  \, .
\end{align}

 Before proving the following blow-up result, we recall that by $(\rho,\Theta)$ we denote the radial and spherical  coordinates introduced in Section \ref{s:spherical}.
\begin{lemma} \label{l:blow-up} Suppose that all the assumptions of Lemma \ref{l:N'} hold true and let $\ell$ be limit defined in Lemma \ref{l:def-ell}. Then there exists $j\in\N$ such that
    \begin{equation} \label{eq:def-ell}
    \ell=\dfrac{-(Q-2)+\sqrt{(Q-2)^2+4\mu_j(\alpha+1)^2}}{2}
    \end{equation}
    where $\mu_j$ is an eigenvalue of the polar operator $-\mathcal L_\Theta$ with $\mathcal L_\Theta$ defined by \eqref{GrPo2}.

    Moreover, for any sequence $\eps_n \downarrow 0$ there  exist a subsequence $\eps_{n_j}$ and an eigenfunction $\Psi$ of the operator $-\mathcal L_\Theta$ associated  with $\mu_j$ such that $\int_{\partial B_1^\alpha} \psi_\alpha \, \Psi^2 \, d\Ha=1$ and  as $j \to +\infty$
    \begin{equation} \label{eq:u-c}
    u_{\eps_{n_j}}(\rho,\Theta) \to \rho^{\ell} \, \Psi(\Theta) \qquad \rho>0, \, \Theta \in \partial B_1^\alpha \, ,
    \end{equation}
    weakly in $ W^{1,2}_\alpha (B_1^\alpha)$, strongly in $ W^{1,2}_\alpha (B_r^\alpha)$ and uniformly in $B_r^\alpha$ for any $r\in(0,1)$.
\end{lemma}

\begin{proof} Throughout the proof of this lemma we denote by $C$ any positive constant independent of $\varepsilon$ which is not necessarily always the same in each formula of this proof.

We first prove that $\{u_\eps\}_{\varepsilon \in (0,\bar r)}$ is bounded in $ W^{1,2}_\alpha (B_1^\alpha)$. By \eqref{eq:u-eps}, \eqref{eq:Poincare} and \eqref{eq:D+H} we obtain
\begin{equation} \label{eq:bound-W12}
    \|u_\varepsilon\|_{ W^{1,2}_\alpha (B_1^\alpha)}^2 \le C[\mathcal N(\varepsilon)+1]
\end{equation}
for any $\varepsilon \in (0,\bar r)$ with $\bar r$ as in \eqref{eq:N(r)}. Boundedness in $ W^{1,2}_\alpha (B_1^\alpha)$ then follows from Lemma \ref{l:bound}.

Therefore, for any sequence $\eps_n$ there exists a subsequence $\eps_{n_j}$ such that
    \begin{equation} \label{eq:conv-Boundary}
         u_{\eps_{n_j}} \rightharpoonup \bar u \quad \text{in } W^{1,2}_\alpha (B_1^\alpha).
    \end{equation}
  Moreover, we claim that
  \begin{equation} \label{eq:conv-Boundary2}
    u_{\eps_{n_j}} \to \bar u \neq 0 \quad \text{in }L^2_\alpha(\partial B_1^\alpha) \, .
  \end{equation}
For simplicity, in order to prove \eqref{eq:conv-Boundary2},  we only wrote $u_{\eps_{n_j}}$
and $\bar u$ on $\partial B_1^\alpha$ to actually denote their traces ${\rm Tr}_\alpha \, u_{\eps_{n_j}}$ and ${\rm Tr}_\alpha \, \bar u$. First of all we observe that $u_{\eps_{n_j}} \to \bar u$ a.e. with respect to the measure $\mathcal H^{N-1}$. Indeed, denoting by $\Sigma_\eta=\{(x,y):|x|<\eta\}$, $\eta>0$, a neighborhood of the set $\Sigma$, we have that
$\{u_\varepsilon\}_{\eps \in (0,\bar r)}$ is bounded in the classical Sobolev space (without weight)
$W^{1,2}(B_1^\alpha \setminus \overline \Sigma_\eta)$.
% and hence along a suitable subsequence we have that $u_{\eps_{n_j}} \to \bar u$ a.e. in $\partial B_1^\alpha \setminus \overline \Sigma_\eta$, for any $\eta>0$.
 Hence by compactness of the classical trace map from $W^{1,2}(B_1^\alpha \setminus \overline \Sigma_\eta)$ to $L^2(\partial B_1^\alpha \setminus \overline \Sigma_\eta)$, along a suitable subsequence of $\{u_{\eps_{n_j}}\}$, the a.e. convergence in $\partial B_1^\alpha \setminus \overline \Sigma_\eta$ holds true. Being the pointwise limit always the same, i.e. the function $\bar u$ found in \eqref{eq:conv-Boundary}, by using a Cantor diagonal argument we may find a suitable subsequence, that for simplicity we still denote by $\{u_{\eps_{n_j}}\}$, such that $u_{\eps_{n_j}} \to \bar u$ a.e. in $\partial B_1^\alpha \setminus \Sigma$ and hence a.e. in $\partial B_1^\alpha$ being $\mathcal H^{N-1}(\partial B_1^\alpha \cap \Sigma)=0$.
Now, by Corollary \ref{c:trace} we have that $\{u_\varepsilon\}_{\eps \in (0,\bar r)}$ is bounded in $L^q_\alpha(\partial B_1^\alpha)$ for some $q>2$ and this, together with the a.e. convergence and the Vitali convergence theorem, yields
the strong convergence in $L^2_\alpha(\partial B_1^\alpha)$ contained in \eqref{eq:conv-Boundary2} along a suitable subsequence.

In particular, by \eqref{eq:normalization} we infer that $\int_{\partial B_1^\alpha} \psi_\alpha \, \bar u^2 \, d\Ha=1$ and hence $\bar u\neq 0$.

Next, putting $V_\eps(x,y)=V(\eps x, \eps^{\alpha+1 }y)$, we observe that
\begin{equation} \label{eq:u-eps-equation}
   - \Delta_\alpha  u_\varepsilon =\varepsilon^2 \, V_\varepsilon \, u_\varepsilon  \qquad \text{in } B_1^\alpha
\end{equation}
for any $\varepsilon \in (0,\bar r)$. By Lemma \ref{l:change-B} and direct computation we see that
\begin{equation} \label{eq:conv-V}
    \|\varepsilon^2 \, V_\varepsilon\|_{W^{1,\sigma}(B_1^\alpha)}\le \varepsilon^{\frac{2\sigma-Q}{\sigma}} \,
     \|V\|_{W^{1,\sigma}(\Omega)}
\end{equation}
for any $\varepsilon \in (0,\min\{1,\bar r\})$. Up to shrink $\bar r$, it is not restrictive to assume in the rest of the proof that $\bar r<1$.

Hence, recalling that $\sigma>Q/2$ thanks to \eqref{eq:V}, by \eqref{eq:bound-W12} and Proposition \ref{p:reg-2} (i), we deduce that $\{u_\varepsilon\}_{\varepsilon \in (0,\bar r)}$ is bounded in $ W^{2,2}_\alpha (B_r^\alpha)$ for any $r\in (0,1)$.

Applying the compact embedding result of Proposition \ref{p:embedding} (i) we deduce that
\begin{equation} \label{eq:strong-conv}
    u_{\varepsilon_{n_j}} \to \bar u \qquad \text{in }  W^{1,2}_\alpha (B_r^\alpha)
\end{equation}
for any $r\in (0,1)$.

Moreover, we can pass to the limit in the weak formulation of \eqref{eq:u-eps-equation} over $B_1^\alpha$ (see \eqref{eq:var} for more details) and by H\"older inequality, \eqref{eq:conv-V} and \eqref{eq:conv-Boundary},  we obtain
    \begin{equation}\label{eq:baru}
    - \Delta_\alpha  \, \bar u =0 \quad \text{in } B_1^\alpha
    \end{equation}
    in a weak sense. Moreover, by Proposition \ref{p:reg} we also have that $\bar u \in C^\infty(B_1^\alpha)$.

    Furthermore, by \eqref{eq:N(r)} and   Lemmas \ref{l:change-B}, \ref{l:rescale-partial} we have
    \begin{align}\label{eq:Ninepsr}
     \mathcal N (\eps r) &= \dfrac{(\eps r)^{2-Q} \int_{B_{\eps r}^\alpha} (|\nabla_\alpha u|^2 - V u^2)\,dxdy}{(\eps r)^{1-Q} \int_{\partial B_{\eps r}^\alpha} \psi_\alpha \, u^2\,d\Ha } \\[7pt]
    \notag & = \dfrac{ r^{2-Q}  \int_{B_{ r}^\alpha} |\nabla_\alpha u_\eps|^2 \,dxdy}{r^{1-Q} \int_{\partial B_{ r}^\alpha} \psi_\alpha \, u_\eps^2\,d\Ha } -
   \frac{(\varepsilon r)^{2-Q}  \int_{B_{\varepsilon r}^\alpha}  V u^2 \, dxdy}{(\varepsilon r)^{1-Q}
    \int_{\partial B_{\varepsilon r}^\alpha} \psi_\alpha \, u^2\,d\Ha} \, .
    \end{align}
    About the second term in the latter equation, by \eqref{eq:prel-est}, \eqref{eq:V} we obtain
    \begin{align} \label{eq:residual-term}
    & \left| \frac{(\varepsilon r)^{2-Q}  \int_{B_{\varepsilon r}^\alpha}  V u^2 \, dxdy}{(\varepsilon r)^{1-Q}
    \int_{\partial B_{\varepsilon r}^\alpha} \psi_\alpha \, u^2\,d\Ha}\right|
     \le \frac{\frac 12 C(h,k,\alpha,\sigma,\bar r) (\varepsilon r)^{\frac{2\sigma-Q}{\sigma}} \|V\|_{W^{1,\sigma}(\Omega)} [ D(\varepsilon r)+H(\varepsilon r)]}{H_{u}(\varepsilon r) }\\[7pt]
    \notag & \qquad \le 2^{-1} C(h,k,\alpha,\sigma,\bar r) (\varepsilon r)^{\frac{2\sigma-Q}{\sigma}} \|V\|_{W^{1,\sigma}(\Omega)} \,[ \mathcal N(\eps r) + 1]= o(1) \quad \text{as }\eps \to 0
    \end{align}
    uniformly with respect to $r$ in a neighborhood of $0$. From \eqref{eq:Ninepsr}, \eqref{eq:residual-term}, \eqref{eq:conv-Boundary},   \eqref{eq:conv-Boundary2} and \eqref{eq:strong-conv}, we deduce
    \begin{equation}\label{eq:Nbaru}
            \ell=\lim_{j \to +\infty}  \mathcal N(\eps_{n_j} r) = \lim_{j \to +\infty} \dfrac{ r^{2-Q}  \int_{B_{ r}^\alpha} |\nabla_\alpha u_{\eps_{n_j}}|^2 \,dxdy}{r^{1-Q} \int_{\partial B_{ r}^\alpha} \psi_\alpha \, u_{\eps_{n_j}}^2\,d\Ha }  \, ,
    \end{equation}

    \begin{equation} \label{eq:Nbaru-2}
        \lim_{j \to +\infty}  r^{2-Q}  \int_{B_{ r}^\alpha} |\nabla_\alpha u_{\eps_{n_j}}|^2 \,dxdy
        =r^{2-Q}  \int_{B_{ r}^\alpha} |\nabla_\alpha \bar u|^2 \,dxdy=:D_{\bar u}(r)
    \end{equation}
      and
    \begin{equation} \label{eq:Nbaru-3}
         \lim_{j \to +\infty} r^{1-Q} \int_{\partial B_{ r}^\alpha} \psi_\alpha \, u_{\eps_{n_j}}^2\,d\Ha=
         r^{1-Q} \int_{\partial B_{ r}^\alpha} \psi_\alpha \, \bar u^2\,d\Ha=:H_{\bar u}(r)
    \end{equation}
    for any $r\in(0,1)$. We observe that $H_{\bar u}(r)>0$ for any $r\in (0,1)$. Indeed if there exists $r_0 \in (0,1)$
    such that $H_{\bar u}(r_0)=0$ then combining \eqref{eq:Nbaru}-\eqref{eq:Nbaru-3} we see that also $D_{\bar u}(r_0)=0$ since otherwise the limit in the right hand side of \eqref{eq:Nbaru} cannot be finite. This implies $\bar u\equiv 0$ in $B_{r_0}^\alpha$ and applying the classical unique continuation property for elliptic operator in $B_1^\alpha \setminus \Sigma$ we infer that $\bar u\equiv 0$ in $B_1^\alpha$, a contradiction.

    Now we may define the Almgren-type quotient corresponding to $\bar u$ by
    \[
    \mathcal N_{\bar u} (\rho) : = \frac{D_{\bar u}(\rho)}{H_{\bar u}(\rho)}=\dfrac{ \rho^{2-Q}  \int_{B_{ \rho}^\alpha} |\nabla_\alpha \bar u|^2 \,dxdy}{\rho^{1-Q} \int_{\partial B_{ \rho}^\alpha} \psi_\alpha \, \bar u^2\,d\Ha }, \quad  \rho \in (0,1).
    \]

    Combining \eqref{eq:Nbaru} with \eqref{eq:Nbaru-2}-\eqref{eq:Nbaru-3}, we infer
    \begin{equation*}
       \ell = \lim_{j \to +\infty}  \mathcal N(\eps_{n_j} \rho) = \mathcal N_{\bar u} (\rho) \qquad
        \text{for any } \rho \in (0,1) \, ,
    \end{equation*}
    which shows that $\mathcal N_{\bar u}$ is a constant function.

    On the other hand, following the same outline as in Lemma \ref{l:N'}, it results
    \begin{equation}\label{eq:N'baru}
        \mathcal N_{\bar u}'(\rho)
    = \frac{2 \rho^{1-2Q}}{(H_{\bar u}(\rho))^2} \left\{ \int_{\partial B_\rho^\alpha} \psi_\alpha
     \left({X_G \bar u}\right)^2 d\Ha
  \cdot \int_{\partial B_\rho^\alpha} \psi_\alpha  \bar u^2 \, d\Ha-\left( \int_{\partial B_\rho^\alpha} \psi_\alpha \bar u {X_G \bar u} \, d\Ha \right)^2 \right\}
    \end{equation}
    for any $\rho\in(0,1)$. Being $\mathcal N_{\bar u}$ constant from \eqref{eq:N'baru} we deduce that we have an equality in a H\"older inequality which implies that $X_G \bar u$ and $\bar u$ are parallel vectors as elements of the Hilbert space $L^2_\alpha(\partial B_\rho^\alpha)$ for any $\rho\in(0,1)$.  Hence there exists a function $h(\rho)$ such that
    \begin{equation}\label{eq:parallelvectors}
        X_G \bar u(\rho,\Theta)=h(\rho)\bar u(\rho,\Theta) \quad \text{for any } \rho \in(0,1), \ \Theta \in \partial B_1^\alpha.
    \end{equation}
    We observe that multiplying both sides of \eqref{eq:parallelvectors} by $\psi_\alpha \bar u$ and integrating over $\partial B_\rho^\alpha$ with respect to the measure $\mathcal H_\alpha^{N-1}$, by \eqref{eq:H'} we deduce that $h(\rho)=\rho H_{\bar u}'(\rho)/(2H_{\bar u}(\rho))$ which shows that $h \in L^1_{\rm loc}(0,1)$.

    By \eqref{eq:vect-field} in the Appendix, we have that $\frac{\partial \bar u}{\partial \rho}= \frac{1}{\rho} \, X_G \bar u$. Integrating \eqref{eq:parallelvectors} we obtain
    \[
    \bar u(\rho,\Theta) = \bar u\left(\tfrac 12,\Theta\right) e^{\int_{1/2}^\rho \frac{h(r)}{r} \,dr}
     \quad \text{ for any $\rho\in (0,1)$ and $\Theta \in \partial B_1^\alpha$} \, .
    \]

    Let us put $w(\Theta)=\bar u\left(\frac 12,\Theta\right)$ and $z(\rho):=e^{\int_{1/2}^\rho \frac{h(r)}{r} \,dr}$.
    Using the polar expression of the Grushin operator \eqref{GrPo} and recalling that $\bar u$ is a solution to \eqref{eq:baru} we have
    \[  \psi_\alpha\left(\frac{\partial^2 z }{\partial \rho^2} \, w(\Theta)+ \frac{Q-1}{\rho}\frac{\partial z }{\partial \rho} \, w(\Theta)+\frac{(\alpha+1)^2}{\rho^2}z(\rho)\mathcal{L}_\Theta w(\Theta) \right)=0 \, .
    \]
    Maintaining $\rho$ fixed and varying $\Theta \in \partial B_1^\alpha$, we see that the function $w=w(\Theta)$ is an eigenfunction of $-\mathcal L_\Theta$ and letting $\mu$ the corresponding eigenvalue, we deduce that
    \begin{equation*} %\label{eq:separation}
              -\mathcal L_\Theta w= \mu w \quad \text{in $\partial B_1^\alpha$} \qquad \text{and} \qquad \frac{\partial^2 z }{\partial \rho^2}+\frac{Q-1}{\rho}\frac{\partial z }{\partial \rho} - \frac{\mu(\alpha+1)^2}{\rho^2} z=0
               \quad \text{in $(0,1)$} \, .
    \end{equation*}
     The linear equation for $z=z(\rho)$ has solutions of the form $z(\rho)=c_1 \rho^{\beta_1} + c_2\rho^{\beta_2}$ where
     $$
       \beta_1 = -\frac12(Q-2)-\frac12\sqrt{(Q-2)^2+4\mu(\alpha+1)^2}
       \quad \text{and} \quad
       \beta_2 = -\frac12(Q-2)+\frac12\sqrt{(Q-2)^2+4\mu(\alpha+1)^2} \, .
     $$
     We show that $\rho^{\beta_1}w(\Theta) \not\in  W^{1,2}_\alpha (B_1^\alpha)$. Indeed, by \eqref{eq:gradu-gradv} in the case $h,k \ge 2$ and the corresponding identities when at least one between $h$ or $k$ equals $1$, we have that
     $$
        |\nabla_\alpha (\rho^{\beta_1} w(\Theta))|^2=\rho^{2\beta_1-2} \, W(\Theta)
     $$
     for some function $W$ depending only on the angular part $\Theta \in \partial B_1^\alpha$. Looking at the spherical representation of the volume element $dxdy$ in \eqref{eq:volume-integral}, we conclude that $|\nabla_\alpha (\rho^{\beta_1} w(\Theta))|^2 \in L^1(B_1^\alpha)$ if and only if $\int_0^1 \rho^{2\beta_1+Q-3} \, d\rho<\infty$ and this is impossible as one can see substituting the expression of $\beta_1$.

     Therefore the expression of $\bar u$ si given by $\bar u(\rho,\Theta)= c_2 \, \rho^{\beta_2} w(\Theta)$.
The value of $c_2$ is given by $\left(\int_{\partial B_1^\alpha} \psi_\alpha \, \bar u^2 \, d\Ha\right)^{-\frac 12}$ so that we can now put $\Psi(\Theta)=c_2 \, w(\Theta)$.

     Putting the expression $\bar u(\rho,\Theta)=\rho^{\beta_2} \Psi(\Theta)$ in $D_{\bar u}$ and $H_{\bar u}$, by
     \eqref{eq:gradu-gradv} and corresponding identities when at least one between $h$ and $k$ equals $1$ and \eqref{eq:b-alpha-eig}, we obtain
     \begin{align*}
     &    D_{\bar u}(\rho) =\rho^{2-Q} \beta_2^2 \int_{0}^{\rho} r^{2\beta_2+Q-3}
        \left(\int_{\partial B_1^\alpha} \psi_\alpha \, \Psi^2 \, d\Ha \right)\, dr \\[7pt]
        & +\rho^{2-Q} \int_0^\rho (\alpha+1)^2\,
        r^{2\beta_2 +Q-3} \left( \int_{\partial B_1^\alpha} b_\alpha(\Psi,\Psi) \, d\Ha \right) \, dr =\frac{\beta_2^2+\mu(\alpha+1)^2}{2\beta_2+Q-2}  \, \rho^{2\beta_2}=\beta_2 \, \rho^{2\beta_2} \, ,  \\[7pt]
     &   H_{\bar u}(\rho) = \rho^{1-Q} \rho^{2\beta_2 +Q-1} \int_{\partial B_1^\alpha} \psi_\alpha \, \Psi^2 \, d\Ha
        =\rho^{2\beta_2} \, ,
     \end{align*}
      where we recall that $ b_\alpha(\cdot,\cdot)$ is defined in \eqref{eq:f-b-sferica}.
    Therefore, a straightforward computation yields $\mathcal N_{\bar u}\equiv \beta_2$ and hence $\beta_2=\ell$.

    It remains to prove the uniform convergence in \eqref{eq:u-c}. It follows from the fact that $\|u_\varepsilon\|_{ W^{1,2}_\alpha (B_1^\alpha)} = O(1)$ as $\varepsilon \to 0^+$ combined with Proposition \ref{p:L-infty} in Section \ref{s:regularity},  Battaglia and Bonfiglioli \cite[Theorem 4.1]{BatBon} and the Ascoli-Arzelà Theorem.
    This completes the proof of the lemma.
\end{proof}

A crucial part for concluding the proof of Theorem \ref{t:main} is to show that the limit in \eqref{eq:exist-lim-H} is strictly positive.
Recalling the notations introduced for \eqref{eq:eig-L-Theta} and \eqref{eq:eigf-L-Theta}, we may expand in Fourier series the function $u$:   $ u(\delta_\eps (\Theta)) = \sum_{i=0}^{+\infty} \varphi_i(\eps) \omega_i(\Theta)$, $\Theta \in \SA$,
with $\delta_\eps$ as in \eqref{eq:delta-lambda}, where $\varphi_i$ is given by
\begin{equation} \label{eq:phi-i}
   \varphi_i(\eps) := \int_{\SA}  \psi_\alpha(\Theta)  u(\delta_\eps(\Theta)) \omega_i(\Theta) \, d\Ha \, .
\end{equation}
We also define
\begin{equation*} % \label{eq:Ups-i}
   \Upsilon_i(\eps) := \int_{B_\eps^\alpha} V(x,y) u(x,y) \omega_i \left(\frac{x}{d_\alpha(x,y)},\frac{y}{[d_\alpha(x,y)]^{\alpha+1}}\right) \, dxdy
\end{equation*}
and
\begin{equation}  \label{eq:zeta-i}
    \zeta_i(\eps) := \eps^{1-Q} \, \Upsilon_i'(\eps) = \int_{\SA} V(\delta_\eps(\Theta)) u(\delta_\eps(\Theta)) \omega_i(\Theta) \, d\Ha
\end{equation}
thanks to Lemma \ref{l:der-B} and Lemma \ref{l:rescale-partial} in the Appendix.
%In particular we have the following Fourier expansion
%\begin{equation} % \label{eq:Fourier}
%       V(\delta_\eps(\Theta)) u(\delta_\eps(\Theta)) = \sum_{i=0}^{+\infty} \zeta_i(\eps) \omega_i(\Theta)
%    \, , \qquad \Theta \in \SA \, .
%\end{equation}
By \eqref{GrPo}, \eqref{eq:phi-i}, \eqref{eq:zeta-i}, the fact that $u$ solves \eqref{eq:u}, the symmetry of the bilinear form $b_\alpha(\cdot,\cdot)$ defined in \eqref{eq:f-b-sferica} and the fact that $\omega_i$ is an eigenfunction of $\mu_i$, we obtain the following equation for $\varphi_i$
\begin{equation} \label{eq:equation-phi-i}
     -\varphi_i''(\eps)-\frac{Q-1}{\eps} \, \varphi_i'(\eps) + \frac{(\alpha+1)^2 \mu_i}{\eps^2} \, \varphi_i(\eps) = \zeta_i(\eps) \quad \text{in the sense of distributions in $(0,\bar r)$}\, .
\end{equation}

\begin{lemma} \label{l:est-phi-i} Suppose that all the assumptions of Lemma \ref{l:N'} hold true and let $\ell$ be the limit defined in Lemma \ref{l:def-ell}. Let $m\ge 1$ be the multiplicity of the eigenvalue $\mu_{j_0} = \mu_{j_0+1} = \dots = \mu_{j_0+m-1}$ of $-\mathcal L_{\Theta}$ found in Lemma \ref{l:blow-up} and satisfying identity \eqref{eq:def-ell}. Then for any $R \in (0,\bar r)$, with $\bar r$ as in Lemma \ref{l:H>0}, and $i \in \{j_0,\dots,j_0+m-1\}$, we have
\begin{equation}\label{eq:est-phi-i}
  \varphi_i(\eps) =\eps^\ell \left( R^{-\ell} \varphi_i(R) \! + \!\frac{2-Q-\ell}{2-Q-2\ell} \, \int_{\eps}^{R} s^{-Q+1-\ell} \Upsilon_i(s) \, ds \! - \! \frac{\ell R^{-Q+2-2\ell}}{2-Q-2\ell} \, \int_0^R s^{\ell-1} \Upsilon_i(s) \, ds \right) \! + \! O(\eps^{\ell + \delta})
\end{equation}
as $\eps \to 0^+$, with $\delta=\frac{2\sigma-Q}{\sigma}$ as in the proof of Lemma \ref{l:bound}.
\end{lemma}

\begin{proof} The proof of this lemma is based on the same argument used in the proof of \cite[Lemma 6.8]{FFT3} and we refer to \cite{FFT3} for the details. By direct computation one can check that \eqref{eq:equation-phi-i} may rewritten in the form
\begin{equation*}
  -\left[ \eps^{Q-1+2\ell} \left( \eps^{-\ell} \varphi_i(\eps) \right)' \right]' = \eps^{Q-1+\ell} \, \zeta_i(\eps) \, .
\end{equation*}
Proceeding as in the proof of \cite[Lemma 6.8]{FFT3} we obtain, for a suitable constant $c_i$,
\begin{align} \label{eq:v-phi-1}
    \varphi_i(\eps) & = \eps^\ell \left[ R^{-\ell} \varphi_i(R) + \frac{2-Q-\ell}{2-Q-2\ell} \int_{\eps}^{R}
     s^{-Q+1-\ell} \, \Upsilon_i(s) \, ds + \frac{\ell c_i R^{-Q+2-2\ell}}{2-Q-2\ell} \right]  \\[10pt]
 \notag  & \qquad  +\frac{\ell \eps^{-Q+2-\ell}}{Q-2+2\ell} \left(c_i + \int_{\eps}^{R} s^{\ell-1} \Upsilon_i(s) \, ds \right) \, .
\end{align}
By Lemma \ref{l:change-B}, \eqref{eq:conv-V}, Proposition \ref{p:embedding} we have that
\begin{align*}
   & \frac{\eps^2}{\sqrt{H(\eps)}} \, |\Upsilon_i(\eps)| \le C(\Omega,\alpha,h,k,\sigma)  \|V\|_{L^\sigma(\Omega)} \, \|\omega_i\|_{L^{\frac{2\sigma}{\sigma-1}}(\SA)} \eps^{Q + \frac{2\sigma-Q}{\sigma}} \, \|u_\eps\|_{ W^{1,2}_\alpha (B_1^\alpha)}
\end{align*}
for some constant $C(\Omega,\alpha,h,k,\sigma) $ depending only on $\Omega,\alpha,h,k,\sigma$. Then, by \eqref{eq:bound-W12}, Lemma \ref{l:bound}, \eqref{eq:exist-lim-H}, we deduce that $\Upsilon_i(\eps) = O(\eps^{Q-2+\ell+\delta})$ as $\eps \to 0^+$ and in particular $s^{-Q+1-\ell} \Upsilon_i(s) \in L^1(0,R)$.

This implies that the first line in \eqref{eq:v-phi-1} is a $O(\eps^\ell)   \subseteq o(\eps^{-Q+2-\ell})$ as $\eps \to 0^+$.

From the estimate $\Upsilon_i(\eps) = O(\eps^{Q-2+\ell+\delta})$ we also have $s^{\ell -1 } \Upsilon_i(s) \in L^1(0,R)$ so that if by contradiction we have $\ell\left(c_i + \int_{0}^{R} s^{\ell-1} \Upsilon_i(s) \, ds \right)\neq 0$ then   $\varphi_i(\eps) \asymp \eps^{-Q+2-\ell}$ as $\eps \to 0^+$ contradicting the fact that $\int_0^R \eps^{Q-1} |\varphi_i(\eps)|^q d\eps<+\infty$ for any $q \ge 1$ if $Q=2$ and $1\le q \le \frac{2Q}{Q-2}$ if $Q>2$, as a consequence of Sobolev embedding (see Proposition \ref{p:embedding}) and \eqref{eq:phi-i}.

Now if $\ell = 0$, the conclusion of the lemma follows immediately, otherwise from what we have obtained above we infer $c_i =- \int_{0}^{R} s^{\ell-1} \Upsilon_i(s) \, ds$ so that \eqref{eq:v-phi-1} becomes
\begin{align*} % \label{eq:v-phi-2}
    \varphi_i(\eps) & = \eps^\ell \left[ R^{-\ell} \varphi_i(R) + \frac{2-Q-\ell}{2-Q-2\ell} \int_{\eps}^{R}
     s^{-Q+1-\ell} \, \Upsilon_i(s) \, ds - \frac{\ell R^{-Q+2-2\ell}}{2-Q-2\ell} \int_0^R s^{\ell-1} \Upsilon_i(s) \, ds \right]  \\[10pt]
 \notag  & \qquad  -\frac{\ell \eps^{-Q+2-\ell}}{Q-2+2\ell}  \int_{0}^{\eps} s^{\ell-1} \Upsilon_i(s) \, ds \, .
\end{align*}
The last term in the above identity satisfies $O(\eps^{\ell+\delta})$ as a consequence of the estimate $\Upsilon_i(s) = O(s^{Q-2+\ell+\delta})$. This completes the proof of the lemma also in the case $\ell \neq 0$.
\end{proof}

 Now we are ready to prove that the limit in  \eqref{eq:exist-lim-H} is strictly positive.

\begin{lemma} \label{l:lim-H>0}
    Suppose that all the assumptions of Lemma \ref{l:N'} hold true and let $\ell$ be the limit defined in Lemma \ref{l:def-ell}. Then the limit \eqref{eq:exist-lim-H} is strictly positive.
\end{lemma}

\begin{proof} Suppose by contradiction that the limit in \eqref{eq:exist-lim-H} equals $0$ so that by the Parseval identity we also have that $\lim_{\eps \to 0^+} \eps^{-\ell} \, \varphi_i(\eps)=0$ for any $i \in \{j_0,\dots,j_0+m-1\}$, with $j_0$ and $m$ as in Lemma \ref{l:est-phi-i}. Therefore, by Lemma \ref{l:est-phi-i}, letting $\eps \to 0^+$, we obtain
\begin{equation*}
   R^{-\ell} \varphi_i(R) \! + \!\frac{2-Q-\ell}{2-Q-2\ell} \, \int_{0}^{R} s^{-Q+1-\ell} \, \Upsilon_i(s) \, ds \! - \! \frac{\ell R^{-Q+2-2\ell}}{2-Q-2\ell} \, \int_0^R s^{\ell-1} \, \Upsilon_i(s) \, ds = 0
\end{equation*}
 which  replaced in the asymptotic estimate \eqref{eq:est-phi-i} for  $\varphi_i$ implies
\begin{equation} \label{eq:est-phi-3}
    \varphi_i(\eps) = - \frac{2-Q-\ell}{2-Q-2\ell} \, \eps^\ell \int_0^\eps s^{-Q+1-\ell} \, \Upsilon_i(s) \, ds  + O(\eps^{\ell + \delta}) \, .
\end{equation}
In the proof of Lemma \ref{l:est-phi-i} we have shown that $\Upsilon_i(s) = O(s^{Q-2+\ell+\delta})$ so that by \eqref{eq:est-phi-3} we obtain $\varphi_i(\eps) = O(\eps^{\ell + \delta})$ as $ \eps \to 0^+$.

Now, by \eqref{eq:est-H-ge} with $\eta=\delta$, we have that $\sqrt{H(\eps)} \ge \sqrt{K(\delta)} \, \eps^{\ell +\frac \delta 2}$. Recalling the definition of $u_\eps$ in \eqref{eq:u-eps}, for any $i \in \{j_0,\dots,j_0+m-1\}$, we then have
\begin{align} \label{eq:eigenspace-1}
   &  \left|\int_{\SA} \psi_\alpha \, u_\eps \, \omega_i \, d\Ha \right|= \frac{|\varphi_i(\eps)|}{\sqrt{H(\eps)}}
   \le \frac{O(\eps^{\ell +\delta})}{\sqrt{K(\delta)} \, \eps^{\ell + \frac \delta 2}} = O(\eps^{\frac \delta 2}) = o(1)
   \qquad \text{as} \ \eps \to 0^+ \, .
\end{align}
On the other hand, by \eqref{eq:u-c} and \eqref{eq:conv-Boundary}, for any $i \in \{j_0,\dots,j_0+m-1\}$
\begin{equation} \label{eq:eigenspace-2}
   \int_{\SA} \psi_\alpha \, u_\eps \, \omega_i \, d\Ha \to \int_{\SA} \psi_\alpha \, \Psi \, \omega_i \, d\Ha
\end{equation}
where $\Psi$ belongs to the eigenspace generated by $\omega_{j_0},\dots,\omega_{j_0+m-1}$.
Combining \eqref{eq:eigenspace-1} and \eqref{eq:eigenspace-2} we infer $\int_{\SA} \psi_\alpha \, \Psi \, \omega_i \, d\Ha = 0$ for any $i \in \{j_0,\dots,j_0+m-1\}$ in contradiction with the fact that $\Psi \not\equiv 0$ belongs to the eigenspace generated by $\omega_{j_0},\dots,\omega_{j_0+m-1}$.
\end{proof}

\subsection{End of the proof of Theorem \ref{t:main}} Using the notations of Lemma \ref{l:est-phi-i} for $\mu_{j_0},\dots, \mu_{j_0+m-1}$ and $\omega_{j_0},\dots,\omega_{j_0+m-1}$, letting $\eps_n \downarrow 0$, by Lemma \ref{l:blow-up} and Lemma \ref{l:lim-H>0}, along a subsequence $\eps_{n_j}$ we have
\begin{align*}
   & \eps_{n_j}^{-\ell} \, u(\delta_{\eps_{n_j}}(x,y)) \rightharpoonup %(d_\alpha(x,y))^\ell \,
     %\Psi\left(\frac{x}{d_\alpha(x,y)},\frac{y}{[d_\alpha(x,y)]^{\alpha+1}}\right)
      (d_\alpha(x,y))^\ell \ \sum_{i=j_0}^{j_0+m-1} \beta_i  \, \omega_i\left(\frac{x}{d_\alpha(x,y)},\frac{y}{[d_\alpha(x,y)]^{\alpha+1}}\right)
\end{align*}
weakly in $ W^{1,2}_\alpha (B_1^\alpha)$ and strongly in $L^2_\alpha(\SA)$, where $\sum_{i=j_0}^{j_0+m-1} \beta_i \, \omega_i = \Psi$ with $\Psi$ as in Lemma \ref{l:blow-up}. To complete the proof of the theorem it remains to prove that the function $\Psi$ and hence the coefficients $\beta_{j_0},\dots,\beta_{j_0+m-1}$  do not depend on the sequence $\eps_n$, thus showing that the convergence above actually holds as $\eps \to 0^+$.

By \eqref{eq:phi-i} and the orthonormality of $\omega_{j_0},\dots,\omega_{j_0+m-1}$, for any $i \in \{j_0,\dots,j_0+m-1\}$, we then have
\begin{equation*}
  \eps_{n_j}^{-\ell} \, \varphi_i(\eps_{n_j}) \to \beta_i  \qquad \text{as } j \to +\infty \, .
\end{equation*}
On the other hand, by Lemma \ref{l:est-phi-i} we have that
\begin{equation*}
   \eps^{-\ell} \varphi_i(\eps) \to R^{-\ell} \varphi_i(R) \! + \!\frac{2-Q-\ell}{2-Q-2\ell} \, \int_{0}^{R} s^{-Q+1-\ell} \Upsilon_i(s) \, ds \! - \! \frac{\ell R^{-Q+2-2\ell}}{2-Q-2\ell} \, \int_0^R s^{\ell-1} \Upsilon_i(s) \, ds \, ,
\end{equation*}
thus proving that $\beta_i$ does not depend on the sequence $\eps_n$ and
\begin{equation*}
    \beta_i = R^{-\ell} \varphi_i(R) \! + \!\frac{2-Q-\ell}{2-Q-2\ell} \, \int_{0}^{R} s^{-Q+1-\ell} \Upsilon_i(s) \, ds \! - \! \frac{\ell R^{-Q+2-2\ell}}{2-Q-2\ell} \, \int_0^R s^{\ell-1} \Upsilon_i(s) \, ds \, .
\end{equation*}
To complete the proof of Theorem \ref{t:main} we need to prove that  the  strong convergence in $ W^{1,2}_\alpha (B_r^\alpha)$ and the uniform convergence in $B_r^\alpha$ contained in \eqref{eq:as-for} holds for any $r>0$ and not only for $r \in (0,1)$ as already shown in Lemma \ref{l:blow-up}. Letting $r>0$ and $(X,Y) \in B_r^\alpha$, one can define $(x,y) = \delta_{1/(2r)}(X,Y) \in B_{1/2}^\alpha$ and observe that
\begin{align*}
  & \rho^{-\ell} u(\delta_\rho(X,Y)) = (2 r)^\ell \ (2 r\rho)^{-\ell} u(\delta_{2 r\rho}(x,y)) \to
  (2 r)^{\ell} (d_\alpha(x,y))^{-\ell} \, \Psi \left(\frac{x}{d_\alpha(x,y)}, \frac{y}{[d_\alpha(x,y)]^{\alpha+1}}\right) \\[7pt]
  & \qquad =  (d_\alpha(X,Y))^{-\ell} \, \Psi \left(\frac{X}{d_\alpha(X,Y)}, \frac{Y}{[d_\alpha(X,Y)]^{\alpha+1}}\right) \, \qquad \text{as }  \rho \to 0^+ ,
\end{align*}
   that completes the proof of Theorem \ref{t:main}.

\subsection{Proof of Corollary \ref{c:1}} We divide the proof in three cases.

\textit{\underline{The case $z_0 \in \Sigma$.}} Letting $z_0=(0,y_0)$, up to the translation $(x,y) \mapsto (x,y-y_0)$ we may assume that $z_0$ is the origin of $\R^N$. Suppose by contradiction that $u\not\equiv 0$ in $\Omega$ and apply Theorem \ref{t:main} to get
\begin{equation} \label{eq:app-t}
   \rho^{-\ell} \, u(\delta_\rho(\Theta)) \to \Psi(\Theta) \qquad \text{as } \rho \to 0^+
\end{equation}
uniformly in $\SA$ for some eigenfunction $\Psi \not\equiv 0$ of $-\mathcal L_\Theta$. Taking $n > \ell$, with $n$ as in the statement of the corollary, we have that $|u(\delta_\rho(\Theta))| \le C \rho^n$ for any $\rho>0$ sufficiently small. This, combined with \eqref{eq:app-t} implies $\Psi \equiv 0$ in $\SA$, a contradiction.

\textit{\underline{The case $z_0 \not\in \Sigma$, $h\ge 2$.}} The assumption $h \ge 2$ implies that $\Omega \setminus \Sigma$ is connected and the proof follows from the classical unique continuation principle for uniformly elliptic operators (see Sogge \cite{So}) applied outside an arbitrarily small neighborhood of the singular set $\Sigma$, as already observed in Remark \ref{r:cor}.

\textit{\underline{The case $z_0 \not\in \Sigma$, $h=1$.}} In this case, if $0 \in \Omega$, the set $\Omega \setminus \Sigma$ is not connected. We denote by $H^\pm := \{(x,y) \in \R^{h+k}: x\gtrless 0\}$ and by
$\Omega^\pm := \Omega \cap H^\pm$ the two connected components of $\Omega \setminus \Sigma$ .
We may suppose that $z_0 \in \Omega^+$, being the other case completely equivalent. The classical unique continuation principle for uniformly elliptic operators implies that $u \equiv 0$ in $\Omega^+$. Suppose by contradiction that $u\not \equiv 0$ in $\Omega$. Then we may apply Theorem \ref{t:main} around $0$ and conclude that the eigenfunction $\Psi \not \equiv 0$ vanishes in $\SA \cap H^+$. Then we may define the function $W \in C^\infty(\R^N \setminus \{0\})$ as in the proof of Proposition \ref{p:reg-eig}, solution of the equation \eqref{eq:W-equation} with a potential in $C^\infty(\R^N \setminus \{0\})$. Choose a point $\bar z = (0,\bar y) \in \Sigma \setminus \{0\}$ so that all derivatives of any order of $W$ vanish at $\bar z$, being $W \equiv 0$ in $H^+$. By Taylor formula with Peano remainder term we deduce that $W(z) = O(|z-\bar z|^n)$ for any $n\in \N$. Applying to $W$ the result of Corollary \ref{c:1} itself, already proved in the case $z_0 = \bar z \in \Sigma$, we may conclude that $W \equiv 0$ in $\R^N \setminus \{0\}$ and in particular $\Psi \equiv 0$ in $\SA$, thus reaching a contradiction.

\subsection{Proof of Corollary \ref{c:2}} For the same reason already explained in the proof of Corollary \ref{c:1}, we have to distinguish three cases.

\textit{\underline{The case $z_0 \in \Sigma$.}} Up to translation we may assume that $z_0=0$. Suppose by contradiction that $u\not\equiv 0$ in $\Omega$.  By  estimate \eqref{eq:est-H-ge} of Lemma \ref{l:H-2ell}
 for any $\eta>0$ there exist $K(\eta)>0$ and $r_\eta>0$ such that, after integrating  the boundary integral:
\begin{equation*}
    \int_{B_r^\alpha} \psi_\alpha u^2 \, dxdy \ge \frac{K(\eta)}{2\ell + \eta +Q} \, r^{2\ell + \eta +Q} \quad  \mbox{ for any } r< r_\eta,
\end{equation*}
thus contradicting the assumption of the corollary one we choose $n>2\ell + \eta +Q$.

\textit{\underline{The case $z_0 \not\in \Sigma$.}} We can treat both the cases $h \ge 2$ and $h=1$ exactly as in the proof of Corollary \ref{c:1} by applying the unique continuation principle with the integral assumption which is precisely the type of condition assumed in Sogge \cite[Theorem 2.1]{So}, that is implied by our one.

\subsection{Proof of Corollary \ref{c:positive-measure}}
When $h \ge 2$ we know that $\Omega \setminus \Sigma$ is connected so that, as already observed in Remark \ref{r:cor} the conclusion follows from the unique continuation principle for uniformly elliptic operators in the version with the solution $u$ vanishing on a set of positive measure, see de Figueiredo and Gossez \cite{deFiGo} and Sogge \cite{So}. When $h=1$, we already observed in Remark \ref{r:cor} that in general we can only conclude that $u$ vanishes on at least one connected component of $\Omega \setminus \Sigma$. But in this case we can apply the argument already used in the proof of Corollary \ref{c:1} and complete the proof also in this situation.

\section{Regularity results and properties of weighted Sobolev spaces} \label{s:regularity}

This section collects some fundamental properties of the weighted Sobolev spaces $ W^{1,p}_\alpha $ and some local regularity results in  $W^{k,p}_\alpha$ and in fractional Sobolev spaces for solutions of equations with the Grushin operator $ \Delta_\alpha $.  We note that parts of the results contained in this section, or their variants in more general settings, can be found in the literature  (see e.g.  Rothschild and  Stein \cite{RoSt},  Garofalo and Nhieu \cite{GaNh98}, Danielli, Garofalo and Nhieu \cite{DaGaNh98},  Monti and Morbidelli \cite{MoMo02}, and Kogoj and Lanconelli \cite{KoLa12}). Nevertheless, to the best of our knowledge, no existing references fully address the exact assumptions required here. For these reasons we decided to provide an exposition including detailed proofs, which we believe  can be of interest for people working with the Grushin operator.

\subsection{Properties of the Sobolev spaces $ W^{1,p}_\alpha (\Omega)$ } \label{s:properties-Sobolev}

In this section we will state and prove some basic results of weighted Sobolev spaces $ W^{1,p}_\alpha (\Omega)$ concerning extension theorems, density, trace operators and embeddings.
In all these statements we assume the validity of the following condition on the domain $\Omega$:

\begin{definition} \label{d:condition-D} Let $h,k,\alpha \in \N$, $h,k,\alpha \ge 1$ and let $N=h+k$. We say that a domain $\Omega\subset \R^N$ satisfies condition $(D)$ if it is a domain with Lipschitz boundary and moreover $\partial \Omega\cap \Sigma$, with $\Sigma$ as in \eqref{eq:def-sigma},
may be covered by a finite family $\{V_j\}_{j=1}^s$ of cuboids (i.e. rotations of rectangle parallelepipeds in $\R^N$) such that the following assumptions hold:
\begin{itemize}
  \item[$(i)$] there exists a corresponding family of rotations $\{r_j\}_{j=1}^s$ in $\R^N=\R^h \times \R^k$, $r_j=r_j(x,y)$ independent of $x\in \R^h$ (and hence depending only on $y\in \R^k$), such that for any $j\in \{1,\dots,s\}$
      $$
         r_j(V_j)=\{(x,y)\in \R^h \times \R^k:\alpha_{\iota j}<x_\iota<\beta_{\iota j}, \iota=1,\dots,h, \
         a_{ij}<y_i<b_{ij}, i=1,\dots,k  \};
      $$

  \item[$(ii)$] for any $j\in \{1,\dots,s\}$
  $$
     r_j(\Omega \cap V_j)=\{(x,y)\in \R^h\times \R^k: (x,y') \in W_j, \ g_j(x,y')<y_k<b_{kj}\}
  $$
  where we wrote $y'=(y_1,\dots,y_{k-1})$ and
  $$
    W_j:=\{(x,y')\in \R^h\times \R^{k-1}:\alpha_{\iota j}<x_\iota<\beta_{\iota j}, \iota=1,\dots,h, \
         a_{ij}<y_i<b_{ij}, i=1,\dots,k-1   \};
  $$

  \item[$(iii)$] for any $j\in \{1,\dots,s\}$ we have that $g_j\in C^{0,1}(\overline{W_j})$ and moreover
  $$
      a_{kj}<g_j(x,y')<b_{kj} \qquad \text{for any } (x,y')\in W_j \, ,
  $$
  and there exists a constant $C_j$ such that
  $$
     |\nabla_x g_j(x,y')| \le C_j |x|^{\alpha} \qquad \text{for any } (x,y')\in W_j \, .
  $$
\end{itemize}

\end{definition}

\begin{remark} We point out that conditions (i)-(iii) in Definition \ref{d:condition-D} not necessarily involve the whole boundary of $\Omega$ but only the part of it intersecting the singular set $\Sigma$. In other words the family of cuboids $\{V_j\}_{j=1}^s$ does not in general cover the whole set $\partial \Omega$ but only $\partial\Omega \cap \Sigma$.
However, for the remaining part of $\partial \Omega$ one can find in Definition \ref{d:condition-D} a Lipschitzianity condition.
\end{remark}

This definition excludes for example domains $\Omega$ for which relatively open sets in $\Sigma$ are entirely contained in $\partial \Omega$. In other words, the set $\Sigma$ intersects transversally $\partial\Omega$. Moreover, the second condition in (iii) states that partial derivatives with respect to the variables $x_i$ of the functions $g_j$ have to vanish at least like $|x|^\alpha$ or faster than it when $(x,y')$ is close to the set $W_j \cap \Sigma$.  Note that the same condition for domains in $\R^2$ has been already introduced by Monti and Morbidelli \cite[p. 763]{MoMo02}.

Definition \ref{d:condition-D} seems to be quite restrictive but for our purposes this is not a great problem since the next results about the properties of the weighted Sobolev spaces  $ W_\alpha^{1,p} (\Omega)$ will be essentially applied to the case $\Omega=B_r^\alpha$ with $B_r^\alpha$ as in \eqref{eq:B-r-alpha}. One can verify that for this kind of balls $B_r^\alpha$, condition $(D)$ is actually satisfied.

\bigskip

We start with the following extension result:

\begin{proposition} \label{p:extension}
  Let $\Omega\subset \R^N$ be a bounded domain satisfying condition $(D)$ introduced in Definition \ref{d:condition-D}.
  Then there exists a linear continuous operator $\mathcal E: W^{1,p}_\alpha (\Omega)\to  W^{1,p}_\alpha (\R^N)$ such that $(\mathcal E v)_{|\Omega}=v$ for any $v\in  W^{1,p}_\alpha (\Omega)$.
\end{proposition}

\begin{proof} The proof follows exactly the same argument used for classical Sobolev spaces based on partition of unity, local deformation of the domain and symmetric extension. Let us emphasize the main differences with the classical case.
Up to the rotation $r_j$ introduced in Definition \ref{d:condition-D}, we may assume that
\begin{equation*}
  \Omega \cap V_j=\{(x,y)\in \R^h\times \R^k: (x,y') \in W_j, \ g_j(x,y')<y_k<b_{kj}\}
\end{equation*}
and if $v\in  W^{1,p}_\alpha (\Omega)$, up to consider a partition of unity, we may assume that
\begin{equation*}
   {\rm supp}(v) \subset \{(x,y)\in \R^h\times \R^k: (x,y') \in W_j, \ g_j(x,y')\le y_k<b_{kj}\}
\end{equation*}
so that by implementing a trivial extension of $v$ we may put $v(x,y',y_k)=0$ for any $(x,y')\in W_j$ and $y_k\ge b_{kj}$.

Then we define the deformation
\begin{align*}
   & \Phi_j:W_j\times \R   \to W_j\times \R   \\
   & \Phi_j(x,y',y_k)=(x,y',y_k-g_j(x,y')) \qquad \text{for any } (x,y',y_k)\in W_j \times \R
\end{align*}
and the function  $\mathbf{v}:W_j\times (0,+\infty)\to \R$ given by $\mathbf{v}(x,y',y_k):=v(\Phi_j^{-1}(x,y',y_k))$ for any
$(x,y',y_k) \in W_j \times (0,+\infty)$.

By a direct computation based on Definition \ref{d:condition-D} (iii), we infer
\begin{align}  \label{eq:prel-est}
   & |\nabla_\alpha \mathbf{v}(x,y',y_k)|^2 \le 2|\nabla_x v(\Phi_j^{-1}(x,y',y_k))|^2 +
      2 \left|\frac{\partial v}{\partial y_k}(\Phi_j^{-1}(x,y',y_k))\right|^2 |\nabla_x g_j(x,y')|^2 \\[5pt]
 \notag      & \quad + 2|x|^{2\alpha}|\nabla_{y'} v(\Phi_j^{-1}(x,y',y_k))|^2 +2|x|^{2\alpha}
       \left|\frac{\partial v}{\partial y_k}(\Phi_j^{-1}(x,y',y_k))\right|^2 |\nabla_{y'} g_j(x,y')|^2 \\[5pt]
  \notag      & \quad +|x|^{2\alpha}
       \left|\frac{\partial v}{\partial y_k}(\Phi_j^{-1}(x,y',y_k))\right|^2  \\[5pt]
   \notag    & \le (2C_j^2+2 \sup |\nabla g_j|^2+3) \, |\nabla_\alpha v(\Phi_j^{-1}(x,y',y_k))|^2
   % \left[|\nabla_x v(\Phi_j^{-1}(x,y',y_k))|^2+ |x|^{2\alpha} |\nabla_y v(\Phi_j^{-1}(x,y',y_k))|^2\right]
   \quad \text{for a.e. $(x,y',y_k) \in W_j \times (0,+\infty)$. }
\end{align}

By \eqref{eq:prel-est} and a change of variables we obtain
\begin{align*}
   & \int_{W_j \times (0,+\infty)}  |\nabla_\alpha  \mathbf{v}|^p dxdy'dy_k \le
   (2C_j^2+2 \sup |\nabla g_j|^2+3)^{\frac p2} \int_{\Omega \cap V_j} |\nabla_\alpha v|^p dxdy \, .
\end{align*}
 With a similar but easier procedure one can show that $\|\mathbf{v}\|_{L^p(W_j \times (0,+\infty))}=\|v\|_{L^p(\Omega \cap V_j)}$
 so that $\mathbf{v}\in  W^{1,p}_\alpha (W_j \times (0,+\infty))$ and
 \begin{equation} \label{eq:V1}
     \|\mathbf{v}\|_{ W^{1,p}_\alpha (W_j \times (0,+\infty))}  \le \left[1+(2C_j^2+2 \sup |\nabla g_j|^2+3)^{\frac p2}\right]^{\frac 1p}
     \, \|v\|_{ W^{1,p}_\alpha (\Omega\cap V_j)} \, .
 \end{equation}

The next step is to extend the function $\mathbf{v}$ by a symmetric reflection with respect to the hyperplane $y_k=0$. This extension gives rise to a function $\widetilde {\mathbf{v}}:W_j\times \R\to \R$ such that $\widetilde {\mathbf{v}}\in  W^{1,p}_\alpha (W_j \times \R)$ and
\begin{equation*} % \label{eq:V2}
     \|\widetilde { \mathbf{v}}\|_{  W^{1,p}_\alpha (W_j \times \R)}=2^{\frac 1p}  \, \| \mathbf{v}\|_{  W^{1,p}_\alpha (W_j \times (0,+\infty))} \, .
\end{equation*}
For more details about the extension by symmetric reflection and the proof of the fact that $X_i \widetilde { \mathbf{v}}$ and $Y_{J,\ell}\widetilde { \mathbf{v}}$ exist in a weak sense and are $L^p$-functions, see Brezis \cite[Lemma 9.2]{Brezis}.

The final step is to define the function $\widetilde v:W_j \times \R \to \R$ by putting $\widetilde v(x,y',y_k):=\widetilde { \mathbf{v}}(\Phi(x,y',y_k))$ for any $(x,y',y_k)\in W_j \times \R$. With the same procedure used in \eqref{eq:prel-est} one can verify that $\widetilde v\in  W^{1,p}_\alpha (W_j \times \R)$ and
\begin{equation} \label{eq:V3}
  \|\widetilde v\|_{ W^{1,p}_\alpha (W_j \times \R)} \le \left[1+(2C_j^2+2 \sup |\nabla g_j|^2+3)^{\frac p2}\right]^{\frac 1p}  \,
  \|\widetilde { \mathbf{v}}\|_{ W^{1,p}_\alpha (W_j \times \R)} \, .
\end{equation}
Since $\widetilde v$ is compactly supported in $W_j\times \R$, it can be trivially extended to the whole $\R^N$. Moreover, combining \eqref{eq:V1}-\eqref{eq:V3} we can write
\begin{equation*}
   \|\widetilde v\|_{ W^{1,p}_\alpha (\R^N)}\le 2^{\frac 1p} \left[1+(2C_j^2+2 \sup |\nabla g_j|^2+3)^{\frac p2}\right]^{\frac 2p}
    \|v\|_{ W^{1,p}_\alpha (\Omega)} \, .
\end{equation*}
This completes the proof of the proposition.
\end{proof}

\begin{remark} \label{r:compact-support} We observe that under the assumptions of Proposition \ref{p:extension}, the extension operator can be constructed in such a way that every function $\mathcal E v$, $v \in  W^{1,p}_\alpha (\Omega)$, is compactly supported in $\R^N$ and its support is contained in a ball $B_R(0)$ with a radius $R$ independent of $v \in  W^{1,p}_\alpha (\Omega)$. Indeed, suppose that is not the case; then it is sufficient to choose a cutoff function $\zeta \in C^\infty_c(\R^N)$ with $\zeta\equiv 1$ in $\Omega$ and multiply every extended function by $\zeta$. In this way, we may construct a new extension operator which now satisfies the desired property.
\end{remark}

The next purpose is to show that it is possible to define a trace operator from the space $ W^{1,p}_\alpha (\Omega)$ to the weighted $L^p_\alpha(\partial\Omega)$ of functions defined on $\partial\Omega$ which belongs to the $L^p$-space with respect to the boundary measure $|x|^{p\alpha} d\mathcal H^{N-1}$.

A possible strategy is to define the trace on smooth functions and to extend it by density and continuity to the whole space $ W^{1,p}_\alpha (\Omega)$. In order to prove density of $C^\infty(\overline \Omega)$ in $ W^{1,p}_\alpha (\Omega)$, we assume again condition $(D)$ introduced in Definition \ref{d:condition-D}.  Since the proof follows the classical argument by Friedrichs in \cite{Friedrichs}, we therefore omit it.

\begin{proposition} \label{p:density}
Let $\Omega\subset \R^N$ be a bounded domain satisfying condition $(D)$ introduced in Definition \ref{d:condition-D}
% and suppose that $h\ge 2$.
Then $C^\infty(\overline \Omega)$ is dense in $ W^{1,p}_\alpha (\Omega)$.
\end{proposition}

Let us proceed with the definition of the trace operator.

\begin{proposition} \label{p:trace-alpha} Let $\Omega\subset \R^N$ be a bounded domain satisfying condition $(D)$ introduced in Definition \ref{d:condition-D}. Then there exists a continuous linear operar ${\rm Tr_\alpha}: W^{1,p}_\alpha (\Omega)\to L^p_\alpha(\partial\Omega)$ such that ${\rm Tr}\, v=v_{|\partial\Omega}$ for any $v\in C^\infty(\overline \Omega)$.
\end{proposition}

\begin{proof} Let us consider the linear operator which maps every function $v \in C^\infty(\overline \Omega)\subset  W^{1,p}_\alpha (\Omega)$ to $v_{|\partial\Omega}\in L^p_\alpha(\partial \Omega)$.
We want to prove that there exists a constant $C>0$ such that
\begin{equation} \label{eq:est-v-0}
     \|v_{|\partial \Omega}\|_{L^p_\alpha(\partial\Omega)}\le C \|v\|_{ W^{1,p}_\alpha(\Omega)} \qquad \text{for any } v\in C^\infty(\overline\Omega)
     \, .
\end{equation}
By using partition of unity and a local representation of $\partial \Omega$, as in Definition \ref{d:condition-D}, the problem can be reduced to prove the following estimate
\begin{equation} \label{eq:est-v}
   \|v_{|  W}\|_{L^p_\alpha(W)}\le C \|v\|_{ W^{1,p}_\alpha(W \times (0,+\infty))} \qquad \text{for any } v\in C^\infty_c(W\times [0,+\infty))
\end{equation}
with $W$ in the form
 $$
    W:=\{(x,y')\in \R^h\times \R^{k-1}:\alpha_{\iota}<x_\iota<\beta_{\iota}, \iota=1,\dots,h, \
         a_{i}<y_i<b_{i}, i=1,\dots,k-1   \} \, .
 $$
Indeed, one can proceed by exploiting the computations in \eqref{eq:prel-est} and the two conditions in Definition \ref{d:condition-D} (iii).
Let $R>0$ be such that the support of $v$ is contained in $W\times [0,R]$.

For any function $v$ as in \eqref{eq:est-v}, by the integration by parts formula we obtain
\begin{align*}
   & \int_{W} |x|^{p\alpha} \, |v(x,y',0)|^p \, dxdy' \\[7pt]
   & \ =\int_{W} |x|^{p\alpha} \, |v(x,y',y_k)|^p \, dxdy'-
   \int_{W \times (0,y_k)} p|x|^{p\alpha} \, |v(x,y',t)|^{p-2} v(x,y',t)\frac{\partial v}{\partial y_k}(x,y',t) \, dxdy'dt \\[7pt]
   & \ \le \int_{W} |x|^{p\alpha} \, |v(x,y',y_k)|^p \, dxdy'+(p-1)\int_{W\times (0,+\infty)}
   |x|^{p\alpha} \, |v|^p \, dxdy+\int_{W\times (0,+\infty)}
   |x|^{p\alpha} \, \left|\frac{\partial v}{\partial y_k}\right|^p \, dxdy \, .
\end{align*}
Integrating with respect to $y_k$ in the interval $(0,R)$ we get
\begin{align} \label{eq:W-0-infty}
   & R \int_{W} |x|^{p\alpha} \, |v(x,y',0)|^p \, dxdy' \\[7pt]
  \notag &   \le \int_{W \times (0,R)} |x|^{p\alpha} \, |v|^p \, dxdy+R(p-1)\int_{W\times (0,+\infty)}
   |x|^{p\alpha} \, |v|^p \, dxdy+R \int_{W\times (0,+\infty)}
   |x|^{p\alpha} \, \left|\frac{\partial v}{\partial y_k}\right|^p \, dxdy  \\[10pt]
  \notag  &  \le [1+R(p-1)] \left(\sup_W |x|^{p\alpha} \right) \ \int_{W\times (0,+\infty)} |v|^p \, dxdy
   +R \int_{W\times (0,+\infty)} |\nabla_\alpha v|^p \, dxdy \, .
\end{align}
This proves \eqref{eq:est-v}. Coming back to the original domain $\Omega$ and exploiting the assumptions contained in Definition \ref{d:condition-D}, one can easily complete the proof of \eqref{eq:est-v-0}.

Combining the estimate \eqref{eq:est-v-0} with the fact that $C^\infty(\overline\Omega)$ is dense in $ W^{1,p}_\alpha (\Omega)$, as shown in Proposition \ref{p:density}, we can extend by continuity and density the restriction operator to the whole $ W^{1,p}_\alpha (\Omega)$.

The operator defined in this way will be called trace operator and it will be denoted with the notation introduced in the statement of the proposition.
\end{proof}

As a corollary of Proposition \ref{p:trace-alpha} we have the following,  where we recall that $Q = h +(\alpha+1)k$ is the homogeneous dimension defined in  Section \ref{s:functional-setting}.

\begin{corollary} \label{c:trace} Let $\Omega\subset \R^N$ be a bounded domain satisfying condition $(D)$ introduced in Definition \ref{d:condition-D}.  Then the operator ${\rm Tr}_\alpha$ introduced in Proposition \ref{p:trace-alpha} is also a linear continuous operator from $ W^{1,p}_\alpha (\Omega)$ to $L^q_\alpha(\partial\Omega)$ for any $1\le q<\infty$ if $Q\le p$ and
for any $1\le q\le \frac{p(Q-1)}{Q-p}$ if $Q>p$.
\end{corollary}

\begin{proof} It is enough to prove the counterpart of \eqref{eq:est-v}, i.e.
\begin{equation} \label{eq:est-v-cor}
   \|v_{| W}\|_{L^q_\alpha(W)}\le C \|v\|_{ W^{1,p}_\alpha (W \times (0,+\infty))} \qquad \text{for any } v\in C^\infty_c(W\times [0,+\infty)) \, .
\end{equation}
For $1\le q<\infty$, by proceeding as in the proof of Proposition \ref{p:trace-alpha} and using first H\"older inequality and then Young inequality, we obtain the generalization of \eqref{eq:W-0-infty}, i.e.
\begin{equation} \label{eq:crit-trace}
    R \int_{W} |x|^{q\alpha} \, |v(x,y',0)|^q \, dxdy'
\end{equation}
\begin{equation*}
 \le \! \left( \sup_W |x|^{q\alpha} \right) \! \left( \|v\|_{L^q(W\times (0,+\infty))}^q
   \! + \! R(q-1) \|v\|_{L^{\frac{p(q-1)}{p-1}}(W\times (0,+\infty))}^q \right) \!
   + \! R \left(\int_{W\times (0,+\infty)} \! |\nabla_\alpha v|^p \, dxdy \right)^{\frac qp} \! \! .
\end{equation*}

In order to estimate the right hand side of \eqref{eq:crit-trace} in terms of $\|v\|_{ W^{1,p}_\alpha (W\times (0,+\infty))}^q$
the only restriction occurs when $Q>p$ and it reads $q\le p_\alpha^*$ and $\frac{p(q-1)}{p-1}\le p_\alpha^*$, as a consequence of Proposition \ref{p:embedding} below. The more restrictive condition is the second one and it can be rewritten as $q\le \frac{p(Q-1)}{Q-p}$. We have so completed the proof of \eqref{eq:est-v-cor}.
\end{proof}

We now recall an embedding for the weighted Sobolev spaces $ W^{1,p}_\alpha (\Omega)$.

\begin{proposition} \label{p:embedding} Let $\Omega\subset \R^N$ be a bounded domain satisfying condition $(D)$ introduced in Definition \ref{d:condition-D}.

\begin{itemize}
  \item[$(i)$] If $1\le p<Q$ then, letting $p^*_\alpha:=\frac{pQ}{Q-p}$, for any $1\le q\le p_\alpha^*$ we have that $ W^{1,p}_\alpha (\Omega)\subset L^q(\Omega)$ and there exists a constant $C(\Omega,\alpha,h,k,p,q)$ depending only on $\Omega$, $\alpha$, $h$, $k$, $p$ and $q$ such that
  \begin{equation*}
     \|u\|_{L^q(\Omega)} \le C(\Omega,\alpha,h,k,p,q) \|u\|_{ W^{1,p}_\alpha (\Omega)} \qquad
      \text{for any } u\in  W^{1,p}_\alpha (\Omega) \, .
   \end{equation*}
   Moreover, if $1\le q < p^*_\alpha$, the embedding $ W^{1,p}_\alpha (\Omega)\subset L^q(\Omega)$ is also compact.

  \item[$(ii)$] If $Q\le p<\infty$ then for any $1\le q<\infty$ we have that $ W^{1,p}_\alpha (\Omega)\subset L^q(\Omega)$ and there exists a constant $C(\Omega,\alpha,h,k,p,q)$ depending only on $\Omega$, $\alpha$, $h$, $k$, $p$ and $q$ such that
  \begin{equation*}
     \|u\|_{L^q(\Omega)} \le C(\Omega,\alpha,h,k,p,q) \|u\|_{ W^{1,p}_\alpha (\Omega)} \qquad
      \text{for any } u\in  W^{1,p}_\alpha (\Omega) \, .
   \end{equation*}
   Moreover, the embedding $ W^{1,p}_\alpha (\Omega)\subset L^q(\Omega)$ is also compact for any $1\le q<\infty$.
\end{itemize}

\end{proposition}

\begin{proof} The proof of the proposition is essentially contained in Section 3 of Kogoj and Lanconelli \cite{KoLa12} where it is proved the continuous embedding $ W^{1,p}_{\alpha,0} (\Omega)\subset L^{p^*_\alpha}(\Omega$).
For proving part $(i)$ we first apply the extension operator constructed in Proposition \ref{p:extension}, see also Remark \ref{r:compact-support}. Let $R>0$ be as in Remark \ref{r:compact-support} so that the embedding $ W^{1,p}_\alpha (\Omega)\subset L^{p^*_\alpha}(\Omega)$ and the relative continuity estimate easily follow from the continuous embedding $ W^{1,p}_{\alpha,0} (B_R(0))\subset L^{p^*_\alpha}(B_R(0))$, see (1.18) in \cite{KoLa12}. The continuous embedding $ W^{1,p}_{\alpha,0} (B_R(0))\subset L^q(B_R(0))$ for $1\le q<p^*_\alpha$ then follows by the H\"older inequality.

The validity of the embedding $ W^{1,p}_\alpha (\Omega)\subset L^q(\Omega)$ for $1\le q\le p^*_\alpha$ and its continuity then follow immediately; at the same time, its compactness for $1\le q<p^*_\alpha$ follows from the compactness of the embedding
 $ W^{1,p}_{\alpha,0} (\Omega)\subset L^q(\Omega)$.

Let us proceed with the proof of $(ii)$. Since $p\ge Q$, for any $0<\eps\le Q-1$, by H\"older inequality we have that $ W^{1,p}_\alpha (\Omega)$ is continuously embedded in $ W^{1,Q-\varepsilon}_\alpha (\Omega)$ which, in turn, is continuously embedded in $L^q(\Omega)$ for any
$1\le q \le (Q-\varepsilon)^*_\alpha=\frac{Q(Q-\eps)}{\eps}$ and compactly embedded in $L^q(\Omega)$ for any
$1\le q < (Q-\varepsilon)^*_\alpha=\frac{Q(Q-\eps)}{\eps}$. For $\eps$ small we see that the critical exponent $(Q-\varepsilon)^*_\alpha$ becomes larger and larger, thus showing the validity of the compact embedding $ W^{1,p}_\alpha (\Omega)\subset L^q(\Omega)$ for any $1\le q<\infty$.

\end{proof}

\subsection{Local $ W^{2,p}_\alpha $-regularity results}

We now state a regularity result for weak solutions of the Poisson equation
\begin{equation}\label{eq:u-f}
  - \Delta_\alpha  \, u=f \, .
\end{equation}
Since all the statements of this subsection are all local regularity results, no restrictions are actually needed on the regularity of domain $\Omega$. Up to translation, it is not restrictive assuming that $\Omega$ contains the origin.

The next proposition provides regularity for a solution $u$ of \eqref{eq:u-f} in terms of both weighted Sobolev spaces and  classical (without weights) fractional Sobolev spaces.

\begin{proposition} \label{p:reg} For $1<p<\infty$, let $f\in L^p_{{\rm loc}}(\Omega)$ and let $u \in   W^{1,2}_{\alpha,\, {\rm loc}}(\Omega)\cap L^p_{{\rm loc}}(\Omega)$ be a weak solution of \eqref{eq:u-f} in the sense that
$$
    \int_\Omega \nabla_\alpha u \cdot \nabla_\alpha v \, dxdy=\int_\Omega fv \, dxdy  \qquad
    \text{for any } v\in  W^{1,2}_{\alpha,c} (\Omega) \cap L^{\frac{p}{p-1}}(\Omega) \, .
$$
Then the following assertions hold true:
% and suppose that $V$ satisfies \eqref{eq:V}.
\begin{itemize}
  \item[$(i)$] if $f\in  W^{m,p}_{\alpha,{\rm loc}}(\Omega)$, $m\ge 0$ integer, then $u\in  W^{m+2,p}_{\alpha,{\rm loc}}(\Omega)$. Moreover, for any open  sets $\omega_1, \omega_2$ such that $\overline{\omega_1} \subset \omega_2 \subset \overline{\omega_2}  \subset \Omega$, there exists a constant $C(\omega_1,\omega_2,\Omega,\alpha,h,k,m,p)$ depending only on $\omega_1$, $\omega_2$ $\Omega$, $\alpha$, $h$, $k$, $m$, $p$ such that
\begin{equation} \label{eq:reg-est}
   \|u\|_{W^{m+2,p}_\alpha(\omega_1)}\le C(\omega_1,\omega_2,\Omega,\alpha,h,k,m,p) \, \Big(\|u\|_{L^p(\omega_2)}+\|f\|_{W^{m,p}_\alpha(\omega_2)}\Big) \, ;
\end{equation}

\item[$(ii)$] if $f\in W^{\gamma,p}_{{\rm loc}}(\Omega)$, $\gamma \in [0,\infty)$, then
$u\in W^{\gamma+\frac{2}{\alpha+1},p}_{{\rm loc}}(\Omega)$. Moreover, for any open   sets $\omega_1, \omega_2$ such that $\overline{\omega_1} \subset \omega_2 \subset \overline{\omega_2}  \subset \Omega$ there exists a constant $C(\omega_1,\omega_2,\Omega,\alpha,h,k,\gamma,p)$ depending only on $\omega_1$, $\omega_2$, $\Omega$, $\alpha$, $h$, $k$, $\gamma$, $p$ such that
\begin{equation} \label{eq:reg-est-2}
   \|u\|_{W^{\gamma+\frac{2}{\alpha+1},p}(\omega_1)}\le C(\omega_1,\omega_2,\Omega,\alpha,h,k,\gamma,p) \, \Big(\|u\|_{L^p(\omega_2)}+\|f\|_{W^{\gamma,p}(\omega_2)}\Big) \, ;
\end{equation}
\end{itemize}

\end{proposition}

\begin{proof}
  The proof of the proposition is a consequence of a classical regularity result for linear second order operators admitting a representation as the sum of squares of vector fields like in \eqref{eq:hormander}.
  For more details see Rothschild and  Stein \cite[Theorem 16]{RoSt} and its proof. For the proofs of \eqref{eq:reg-est}-\eqref{eq:reg-est-2} one can combine the following results contained in \cite{RoSt}: Theorem 11, Propositions 17.3-17.4, Corollary 17.14 and the proof of Theorem 16.
\end{proof}

Combining the assumptions on the potential $V$ contained in \eqref{eq:V} with the results of Propositions \ref{p:embedding}-\ref{p:reg} we obtain  the following regularity results for weak solutions of \eqref{eq:u}.

\begin{proposition} \label{p:L^q-estimate} Let $u\in   W^{1,2}_{\alpha,\, {\rm loc}}(\Omega)$ be a weak solution of \eqref{eq:u} with $V$ satisfying \eqref{eq:V}. Then $u\in L^q_{{\rm loc}}(\Omega)$ for any $1\le q<\infty$ and moreover for any open  sets $\omega_1, \omega_2$ such that $\overline{\omega_1} \subset \omega_2 \subset \overline{\omega_2}  \subset \Omega$, there exists a constant $C(\omega_1,\omega_2,\Omega,\alpha,h,k,\sigma,q)$ depending only on $\omega_1$, $\omega_2$ $\Omega$, $\alpha$, $h$, $k$, $\sigma$, $q$ such that
\begin{equation*}
   \|u\|_{L^q(\omega_1)} \le C(\omega_1,\omega_2,\Omega,\alpha,h,k,\sigma,q) \Big(\|u\|_{ W^{1,2}_\alpha}(\Omega)+\|V\|_{L^\sigma(\Omega)}\Big) \, .
\end{equation*}
\end{proposition}

\begin{proof} We may assume in the rest of the proof that $Q>2$ otherwise the proof of the proposition is an immediate consequence of Proposition \ref{p:embedding} (ii) with $p=Q=2$. Let us define the following sequences of exponents:
\begin{align*}
   & q_0=2^*_\alpha=\tfrac{2Q}{Q-2} \, , \quad p_n=\tfrac{q_n \sigma}{q_n+\sigma} \ \ \ \text{for $n\ge 0$} \, , \quad q_{n+1}=\tfrac{Qp_n}{Q-2p_n} \ \ \ \text{for $n\ge 0$}
\end{align*}
until $p_n$ satisfies the condition $p_n<\frac Q2$.

With this definition, combining iteratively Proposition \ref{p:embedding} and Proposition \ref{p:reg}, we may construct the following scheme: starting from $u\in L^{q_n}_{{\rm loc}}(\Omega)$, by H\"older inequality we first have $Vu\in L^{p_n}_{{\rm loc}}(\Omega)$, then by Proposition \ref{p:reg} (i) we have $u\in  W^{2,p_n}_{\alpha,{\rm loc}}(\Omega)$ and by an iteration of Proposition \ref{p:embedding} we finally obtain $u\in L^{q_{n+1}}_{{\rm loc}}(\Omega)$ and the procedure can restart again.

The iteration can be repeated until condition $p_n<\frac Q2$ is satisfied. If after a finite number of iterations we obtain $p_n\ge \frac Q2$ then we are done since $u\in  W^{2,p_n}_{\alpha,{\rm loc}}(\Omega)$ implies $u\in L^q_{{\rm loc}}(\Omega)$ for any $1\le q<\infty$.

After a simple check one can verify that the sequence of exponents $p_n$ is increasing since $\sigma>\frac Q2$ and since until $p_n<\frac Q2$ we also have
$$
    \tfrac{p_{n+1}}{p_n}=\tfrac{Q\sigma}{Q\sigma-(2\sigma-Q)p_n}>\tfrac{Q\sigma}{Q\sigma-(2\sigma-Q)p_0}>1 \, .
$$
This shows that after a finite number of steps we find $n\ge 0$ such that $p_n<\frac Q2$ but $p_{n+1}\ge \frac{Q}{2}$.

The conclusion is that $u\in  W^{2,p_{n+1}}_{\alpha,{\rm loc}}(\Omega)$ with $p_{n+1}\ge \frac{Q}{2}$ which, in turn, implies $u\in L^q_{{\rm loc}}(\Omega)$ for any $1\le q<\infty$. The corresponding estimate appearing in the statement of the proposition follows as well.
\end{proof}

 By Battaglia and Bonfiglioli \cite{BatBon} it is possible to obtain also an $L^\infty_{\rm loc}$ regularity results.

\begin{proposition} \label{p:L-infty} Let $u\in   W^{1,2}_{\alpha,\, {\rm loc}}(\Omega)$ be a weak solution of \eqref{eq:u} with $V$ satisfying \eqref{eq:V}. Then $u$ is locally H\"older continuous in $\Omega$ and in particular $u\in L^\infty_{{\rm loc}}(\Omega)$ and moreover for any open   sets $\omega_1, \omega_2$ such that $\overline{\omega_1} \subset \omega_2 \subset \overline{\omega_2}  \subset \Omega$, there exists a constant $C(\omega_1,\omega_2,\Omega,\alpha,h,k,\sigma)$ depending only on $\omega_1$, $\omega_2$, $\Omega$, $\alpha$, $h$, $k$, $\sigma$ such that
\begin{equation*}
   \|u\|_{L^\infty(\omega_1)} \le C(\omega_1,\omega_2,\Omega,\alpha,h,k,\sigma) \Big(\|u\|_{ W^{1,2}_\alpha(\omega_2)}+\|V\|_{L^\sigma(\omega_2)}\Big) \, .
\end{equation*}
\end{proposition}

\begin{proof} See the statement of  \cite[Theorem 4.1]{BatBon}. In order to verify that the assumptions of  \cite[ Theorem 4.1]{BatBon} are really satisfied by our problem, see \cite[Section 6]{BatBon} and \eqref{eq:vector-fields}, \eqref{eq:hormander}, \eqref{eq:delta-lambda} in the present article.
Finally we have to show that $Vu \in L^p_{{\rm loc}}(\Omega)$ for some $p>Q/2$; this is an immediate consequence of Proposition \ref{p:L^q-estimate}, from which we have that $u\in L^q_{{\rm loc}}(\Omega)$ for any $1\le q <\infty$, and of \eqref{eq:V}, which states that $V\in L^\sigma(\Omega)$ for some $\sigma>\frac{Q}{2}$.
\end{proof}

 Now we have all the ingredients to prove the following regularity results, that will be necessary to prove the Pohozaev-type identity of Proposition \ref{p:Pohozaev}.

\begin{proposition} \label{p:reg-2} Let $u\in   W^{1,2}_{\alpha,\, {\rm loc}}(\Omega)$ be a weak solution of \eqref{eq:u} with $V$ satisfying \eqref{eq:V}. Then the following conclusions hold true:
\begin{itemize}
  \item[$(i)$] $u \in W^{1+\frac{2}{\alpha+1},\sigma}_{{\rm loc}}(\Omega) \cap  W^{2,2}_{\alpha,{\rm loc}} (\Omega)$ and moreover
  for any open   sets $\omega_1, \omega_2$ such that $\overline{\omega_1} \subset \omega_2 \subset \overline{\omega_2}  \subset \Omega$, there exists a constant $C(\omega_1,\omega_2,\Omega,\alpha,h,k,\sigma)$ depending only on $\omega_1$, $\omega_2$, $\Omega$, $\alpha$, $h$, $k$, $\sigma$ such that
  \begin{equation}\label{eq:2-f-s}
     \|u\|_{W^{1+\frac{2}{\alpha+1},\sigma}(\omega_1)} + \|u\|_{W^{2,2}_{\alpha}(\omega_1)}
      \le C(\omega_1,\omega_2,\Omega,\alpha,h,k,\sigma) \Big(\|u\|_{ W^{1,2}_\alpha(\omega_2)}+\|V\|_{W^{1,\sigma}(\omega_2)}\Big) \, ;
  \end{equation}

  \medskip

  \item[$(ii)$] $\nabla u \in L^2_{{\rm loc}}(\Omega; \R^N)$ and the traces on $\partial B_r^\alpha$ of first order derivatives of $u$ are in $L^2(\partial B_r^\alpha)$ for any $r>0$ such that $\overline{B_r^\alpha} \subset \Omega$.
      Moreover   for any open  sets $\omega_1, \omega_2$ such that $\overline{\omega_1} \subset \omega_2 \subset \overline{\omega_2}  \subset \Omega$, there exists a constant $C(\omega_1,\omega_2,\Omega,\alpha,h,k,\sigma)$ depending only on $\omega_1$, $\omega_2$, $\Omega$, $\alpha$, $h$, $k$, $\sigma$ and a constant $C(\omega_1,\omega_2,r,\alpha,h,k,\sigma)$ depending only on $\omega_1$, $\omega_2$, $\Omega$, $r$, $\alpha$, $h$, $k$, $\sigma$ such that
  \begin{align}\label{eq:2-f-s-bis}
     & \|\nabla u\|_{L^2(\omega_1; \R^N)}
      \le C(\omega_1,\omega_2,\Omega,\alpha,h,k,\sigma) \Big(\|u\|_{ W^{1,2}_\alpha (\omega_2)}+\|V\|_{W^{1,\sigma}(\omega_2)}\Big) \, , \\
  \label{eq:2-f-s-ter}   & \|\nabla u\|_{L^2(\partial B_r^\alpha; \R^N)}
      \le  C(\omega_1,\omega_2,r,\alpha,h,k,\sigma) \Big(\|u\|_{ W^{1,2}_\alpha (\omega_2)}+\|V\|_{W^{1,\sigma}(\omega_2)}\Big) \, .
  \end{align}

\end{itemize}

\end{proposition}

\begin{proof} {\bf Proof of (i).} We begin by proving that $u \in W^{1+\frac{2}{\alpha+1},\sigma}_{{\rm loc}}(\Omega)$. Our purpose is to apply  Brezis and Mironescu \cite[Lemma 6.1]{BrMi} with $f=u$ and $g=V$.

If $N>\sigma$ we choose $p=\sigma$, $t=\frac{N\sigma}{N-\sigma}$, and $\sigma<r\le \frac{N\sigma}{N-\sigma}$ such that the number  $\theta$  defined by the identity $\frac{1}{r}+\frac{\theta}{t}=\frac{1}{p}$
satisfies $0<\theta<1$. We observe that choosing $r$ sufficiently close to $N<\frac{N\sigma}{N-\sigma}$, the number $0<\theta<1$ becomes closer and closer to $1$. Note that the inequality $N<\frac{N\sigma}{N-\sigma}$ is satisfied as a consequence of \eqref{eq:V},  being $Q>N$.

If $N\le \sigma$ we choose $p=t=\sigma$, $r>\sigma$ and $\theta=\frac{r-\sigma}{r}$. We observe that choosing $r>\sigma$ sufficiently large, the number $0<\theta<1$ becomes closer and closer to $1$.

%The previous choices of $p,t,r$ were made possible by \eqref{eq:V} and Proposition \ref{p:embedding}.

We now put $s_1=\frac{2}{\alpha+1}$ and, up to shrink $\theta$ if necessary, we may assume that $\theta s_1\le 1$.

In this way, by \eqref{eq:V},  Proposition \ref{p:L^q-estimate}, Proposition \ref{p:reg} (ii) and Proposition \ref{p:L-infty}, $u$ and $V$ satisfy the assumptions of \cite[Lemma 6.1]{BrMi} with $f=u$ and $g=V$ in the following sense:
\begin{equation*} %\label{eq:fractional-spaces}
   u \in W^{s_1,t}_{{\rm loc}}(\Omega)\cap L^\infty_{{\rm loc}}(\Omega) \quad
   \text{and} \quad V\in W^{\theta s_1,\sigma}_{{\rm loc}}(\Omega) \cap L^r_{{\rm loc}}(\Omega) \, .
\end{equation*}
Actually, in order to apply the result in \cite{BrMi}, one needs to have functions defined on the whole $\R^N$ and belonging in the corresponding fractional spaces defined over $\R^N$; in our case, this may be done by multiplying $u$ and $V$ by a cut-off function and after that by extending them trivially to the whole $\R^N$. In this way, the extended functions that we still denote for simplicity by $u$ and $V$, belong to $W^{s_1,t}(\R^N)\cap L^\infty(\R^N)$ and $W^{\theta s_1,\sigma}(\R^N) \cap L^r(\R^N)$ respectively.
As a consequence of Lemma 6.1 in \cite{BrMi}, we obtain $Vu\in W^{\theta s_1,\sigma}_{{\rm loc}}(\Omega)$ and moreover for any couple of open sets $\omega_1, \omega_2$ such that  $\overline \omega_1 \subset \omega_2$ and $\overline \omega_2 \subset \Omega$, there exist two constants $C_1(\omega_1,\omega_2,\Omega,\alpha,h,k,\sigma,r)$ and $C_2(\omega_1,\omega_2,\Omega,\alpha,h,k,\sigma,r)$ depending only on $\omega_1,\omega_2,\Omega,\alpha,h,k,\sigma,r$ such that
\begin{align}  \label{eq:Vu-bis}
 &  \|Vu\|_{W^{\theta s_1,\sigma}_{{\rm loc}}( \omega_1)} \\
\notag & \qquad   \le C_1(\omega_1,\omega_2,\Omega,\alpha,h,k,\sigma,r)
     \Big(\|u\|_{L^\infty( \omega_2)} \, \|V\|_{W^{\theta s_1,\sigma}( \omega_2)}+\|V\|_{L^r( \omega_2)} \|u\|_{W^{s_1,t}( \omega_2)}^\theta \, \|u\|_{L^\infty( \omega_2)}^{1-\theta}\Big)  \\
\notag  & \qquad \le C_2  (\omega_1,\omega_2,\Omega,\alpha,h,k,\sigma,r)  \Big(\|u\|_{ W^{1,2}_\alpha(\omega_2)}+\|V\|_{L^\sigma( \omega_2)}\Big) \, .
\end{align}
By Proposition \ref{p:reg} (ii) we infer $u\in W^{\theta s_1 +s_1,\sigma}_{{\rm loc}}(\Omega)$. Now, we can proceed iteratively maintaining the same choice of $r$ and $\theta$ and by putting at each step $s_{n+1}=\theta s_n+s_1$ for any $n\ge 1$ until $\theta s_n \le 1$.

Suppose by contradiction that $\theta s_n \le 1$ for any $n\ge 1$. By direct computation we see that $s_n=s_1 \sum_{j=0}^{n-1} \theta^j=\frac{s_1(1-\theta^n)}{1-\theta}$  and that $\lim_{n\to +\infty} \theta s_n=\frac{\theta s_1}{1-\theta}$.

It is easy to verify that with an appropriate choice of $\theta\in (0,1)$ the limit $\frac{\theta s_1}{1-\theta}>1$ thus showing that $\theta s_n$ becomes eventually larger than $1$, a contradiction.

We may conclude that after a finite number of steps we find $n\ge 1$ such that $Vu \in W^{\theta s_n,\sigma}_{{\rm loc}}(\Omega)$, $\theta s_n\le 1$ and $u \in W^{s_{n+1},\sigma}_{{\rm loc}}(\Omega)$ with $s_{n+1}>\theta s_{n+1}>1$.

In particular we obtain that $u\in W^{s,\sigma}_{{\rm loc}}(\Omega)$ for some $s>1$. Now, we may choose $p$ and $t$ as above, $\theta=\frac 1s$ and $r$ such that the identity $\frac 1r +\frac \theta t=\frac 1p$ holds true. Such an $r$ always satisfies $r>\sigma$ and if $N>\sigma$ it also satisfies $r<\frac{N\sigma}{N-\sigma}$ as one can deduce from of the fact
that $\sigma >\frac{Q}{2}\ge \frac{N}{2}$. In this way we obtain
\begin{equation} \label{eq:prel-est-*}
   Vu \in W^{1,\sigma}_{{\rm loc}}(\Omega)
\end{equation}
and hence $u\in W^{1+\frac{2}{\alpha +1},\sigma}_{{\rm loc}}(\Omega)$  by Proposition \ref{p:reg} (ii). This completes the proof of the first part of (i).

Let us conclude the proof of (i) by showing that $u\in  W^{2,2}_{\alpha,{\rm loc}} (\Omega)$. Since $Vu\in L^{\sigma}_{{\rm loc}}(\Omega)$, by Proposition \ref{p:reg} (i) we deduce that $u \in  W^{2,\sigma}_{\alpha,{\rm loc}}(\Omega)$.

Let us consider now the three cases $Q\ge 4$, $Q=3$ and $Q=2$ separately.

\textit{Tha case $Q\ge 4$}. Since $u \in  W^{2,\sigma}_{\alpha,{\rm loc}}(\Omega)$ we immediately have that $\sigma>\frac Q2\ge 2$ thanks to \eqref{eq:V} and the conclusion readily follows.

\textit{The case $Q=3$}. Since $h,k\ge 1$ and $\alpha$ is a nonnegative integer we have the three following subcases: $h=1$, $k=2$ and $\alpha=0$ or $h=2$, $k=1$ and $\alpha=0$ or $h=k=1$ and $\alpha=1$. By \eqref{eq:prel-est-*} and Sobolev embedding we have that $Vu \in L^q_{{\rm loc}}(\Omega)$ for any $1\le q<\infty$ if $N\le \sigma$ or $Vu \in L^{\frac{N\sigma}{N-\sigma}}_{{\rm loc}}(\Omega)$ if $N>\sigma$.   When $N>\sigma$, in all the three subcases we have that $\frac{N\sigma}{N-\sigma}>2$ as one can check by direct computation after recalling that $\sigma>\frac{Q}{2}=\frac 32$.

In both situations $N>\sigma$ and $N\le \sigma$ we may conclude that $Vu \in L^q_{{\rm loc}}(\Omega)$ for some $q>2$ and the conclusion readily follows by Proposition \ref{p:reg} (i).

\textit{The case $Q=2$}. Since $h,k\ge 1$ and $\alpha$ is a nonnegative integer the only possibility is that $h=k=1$ and $\alpha=0$. Therefore the conclusion follows by proceeding as in the case $Q=3$. Note that in the case $N>\sigma$ the critical Sobolev exponent $\frac{N\sigma}{N-\sigma}=\frac{2\sigma}{2-\sigma}>2$.

The corresponding estimate \eqref{eq:2-f-s} follows by iteration of Sobolev embeddings, Proposition \ref{p:reg} and \eqref{eq:Vu-bis}.

 \medskip

{\bf Proof of (ii).} We start by proving that for any $r>0$ such that $\overline{B_r^\alpha} \subset \Omega$, the trace on $\partial B_r^\alpha$ of the first order derivatives of $u$ belong to $L^2(\partial B_r^\alpha)$. To see this, it is sufficient to recall that $u\in W^{1+\frac{2}{\alpha +1},\sigma}_{{\rm loc}}(\Omega)$ from part (i) of this proposition and hence its first order derivatives belong to $W^{\frac{2}{\alpha +1},\sigma}(B_r^\alpha)$. Applying Adams \cite[Theorem 7.58]{Adams} we deduce that the traces of functions in $W^{\frac{2}{\alpha +1},\sigma}(B_r^\alpha)$ belong to
$L^{\frac{\sigma(N-1)}{N-\frac{2\sigma}{\alpha+1}}}(\partial B_r^\alpha)$ if $N>\frac{2\sigma}{\alpha+1}$ and to $L^q(\partial B_r^\alpha)$ for any $1\le q<\infty$ if $N\le \frac{2\sigma}{\alpha+1}$. In any case we see that first order partial derivatives belong to $L^2(\partial B_r^\alpha)$. Indeed, since $\sigma>\frac Q2$, for the critical exponent in the case $N>\frac{2\sigma}{\alpha+1}$ we have
\begin{align} \label{eq:ge-than-2}
   & \tfrac{\sigma(N-1)}{N-\frac{2\sigma}{\alpha+1}}>\tfrac{Q(N-1)}{2\left(N-\frac{Q}{\alpha+1}\right)}\ge 2
\end{align}
where the last inequality is equivalent to the inequality
$$
  (h+k+\alpha k)(h+k-1)\ge \tfrac{4h\alpha}{\alpha+1} \, .
$$
The last one is readily verified for any $h,k\ge 1$ and $\alpha\ge 0$ integer.

It remains to show that $\nabla u \in L^2_{{\rm loc}}(\Omega;\R^N)$.  It is enough to show that $\nabla u \in L^2_{{\rm loc}}(\Omega_1;\R^N)$ for any smooth subset $\Omega_1$ such that  $\overline{\Omega_1} \subset \Omega$. We recall that $u \in W^{1+\frac{2}{\alpha +1},\sigma}(\Omega_1)$ so that its first order derivatives belong to $W^{\frac{2}{\alpha +1},\sigma}(\Omega_1)$. Now the conclusion follows immediately by Sobolev embedding both in the case $N\le \frac{2\sigma}{\alpha+1}$ and in the case $N>\frac{2\sigma}{\alpha+1}$: in the first case we have that first order derivatives are in $L^q(\Omega_1)$ for any $1\le q<\infty$ while in the second case we have that first order derivatives are in $L^q(\Omega_1)$ for any $1\le q\le \frac{\sigma N}{N-\frac{2\sigma}{\alpha+1}}$ as one can deduce by applying again    \cite[Theorem 7.58]{Adams}. Finally, by \eqref{eq:ge-than-2} we observe that $\frac{\sigma N}{N-\frac{2\sigma}{\alpha+1}}\ge 2$.

The corresponding estimates \eqref{eq:2-f-s-bis} and \eqref{eq:2-f-s-ter} follow by Sobolev embedding and \eqref{eq:2-f-s}.
 \end{proof}

We conclude this section by observing that, as a consequence of the regularity results for $ \Delta_\alpha $, the eigenfunctions of the spherical operator $-\mathcal L_\Theta$ are smooth on $\SA$:

\begin{proposition} \label{p:reg-eig} Let $\Psi$ be an eigenfunction of $-\mathcal L_\Theta$. Then $\Psi \in C^\infty(\SA)$.
\end{proposition}

\begin{proof} Let $\lambda$ be the eigenvalue of $-\mathcal L_\Theta$ corresponding to $\Psi$ and denote by $W$ the function defined in $\R^{h+k}\setminus \{0\}$ by $W(x,y)=\Psi\left(\frac{x}{d_\alpha(x,y)},\frac{y}{[d_\alpha(x,y)]^{\alpha+1}}\right)$. Then by \eqref{GrPo}, $W$ solves the equation
\begin{equation} \label{eq:W-equation}
  - \Delta_\alpha  W = \frac{\lambda (\alpha+1)^2 |x|^{2\alpha}}{(d_\alpha(x,y))^{2\alpha+2}} \, W \qquad \text{in } \R^{h+k}\setminus \{0\} \, .
\end{equation}
Applying the same kind of argument used in the proof of Proposition \ref{p:reg-2}, based on Proposition \ref{p:reg-2} and    \cite[Lemma 6.1]{BrMi}, we infer that $W \in C^\infty(\R^{h+k} \setminus \{0\})$ and in particular $\Psi \in C^\infty(\SA)$.
\end{proof}

\section{Appendix}
 In this appendix we collect some scaling properties of integrals on $B_\lambda^\alpha$ and $\partial B_\lambda^\alpha$ (see \eqref{eq:B-r-alpha} for the definition). Moreover we prove the Pohozaev-type identity of Proposition \ref{p:Pohozaev}.

Let $\delta_\lambda$ be the anisotropic dilation defined in \eqref{eq:delta-lambda}. As we have already observed one has that    $ \Delta_\alpha  (u \circ \delta_\lambda) = \lambda^2 ( \Delta_\alpha   u) \circ \delta_\lambda$.
This means that, up to a constant multiplier, the Grushin Laplacian commutes with the group of operators $\{T_t\}_{t\in \R}$ defined by $T_t u=u\circ \delta_{e^t}$  for every sufficiently smooth function $u$.
The group $\{T_t\}_{t\in \R}$ admits as an infinitesimal operator the operator $X_G$ defined in \eqref{eq:def-XG}.

Being $X_G$ the infinitesimal generator of the group of operators $\{T_t\}_{t\in \R}$, we have that
\begin{equation*}
  \begin{cases}
      \frac{\partial w}{\partial t}=X_G w \, , \qquad t \in \R \, , \\
      w(\cdot,\cdot,0)=u   \, ,
  \end{cases}
\end{equation*}
where $w(\cdot,\cdot,t):=T_t u$ where $u$ is a function of class $C^1$. In particular letting $u_\lambda:= u\circ \delta_\lambda$ and $t=\log \lambda$, we have that
\begin{equation} \label{eq:vect-field}
   \frac{\partial u_\lambda}{\partial \lambda}=\frac{dt}{d\lambda} \frac{\partial}{\partial t} (T_t u)=\frac 1 \lambda \, X_G (T_t u)= \frac{1}{\lambda} \, X_G u_\lambda \, .
\end{equation}

We can now state some basic properties which will be used in the monotonicity argument and the subsequent blow-up procedure.

%\begin{lemma} \label{l:int-parts} Let $u,v \in ????$ % C^2(\overline B_r^\alpha)$
% Then
% \begin{align*} % \label{eq:int-parts}
%  & \int_{B_r^\alpha} (- \Delta_\alpha  u) v \, dxdy=\int_{B_r^\alpha}  \nabla_\alpha u \cdot \nabla_\alpha v \, dxdy-\int_{\partial B_r^\alpha} \frac{\psi_\alpha}{d_\alpha} \, v X_G u \, d\Ha \, .
%\end{align*}
%\end{lemma}

\begin{lemma} \label{l:der-B} Let $\Omega\subseteq \R^N$ be a domain containing the origin and let $v\in L^1_{{\rm loc}}(\Omega)$. Then
\begin{equation*}
  \frac{d}{d\lambda} \int_{B_\lambda^\alpha} v(x,y) \, dxdy=\int_{\partial B_\lambda^\alpha} v \, d\Ha
\end{equation*}
for a.e. $\lambda>0$ such that $\overline{B_\lambda^\alpha}\subset \Omega$.
\end{lemma}

\begin{proof} The proof is immediate consequence of the coarea formula being $\partial B_\lambda^\alpha$ the level sets of the function $d_\alpha$.
\end{proof}

We recall once again the definition of the homogenous dimension $Q=h+(\alpha+1)k$ introduced in Section \ref{s:functional-setting} before \eqref{eq:V}. In the next lemmas it is clarified its role when dealing with changes of variables for both volume and surface integrals.

\bigskip

\begin{lemma} \label{l:change-B} Let $\Omega\subseteq \R^N$ be a domain containing the origin and let $v\in L^1_{{\rm loc}}(\Omega)$. Then
\begin{equation*}
   \int_{B_\lambda^\alpha} v(x,y) \, dxdy=\lambda^Q \int_{B_1^\alpha} v(\delta_\lambda(x,y)) \, dxdy
\end{equation*}
for any $\lambda>0$ such that $\overline{B_\lambda^\alpha}\subset \Omega$.
\end{lemma}

\begin{proof} The proof follows by the definition of $\delta_\lambda$, \eqref{eq:delta-B} and a change of variables.
\end{proof}

\begin{lemma} \label{l:rescale-partial} Let $v\in L^1(\partial B_\lambda^\alpha)$ for some $\lambda>0$. Then
\begin{equation*}
   \int_{\partial B_\lambda^\alpha} v\, d\Ha=\lambda^{Q-1} \int_{\SA} v(\delta_\lambda(\Theta)) \, d\Ha.
\end{equation*}
\end{lemma}

\begin{proof} It can be proved by combining the result of Lemma \ref{l:der-B}, which characterizes a boundary integral as the derivative of a volume integral, with the scaling properties of the volume integral as stated in Lemma \ref{l:change-B}.
\end{proof}

We give now a rigorous version of an integration by parts formula on  $B_\lambda^\alpha$.
\begin{proposition} \label{p:int-parts}
Let $\Omega\subseteq \R^N$ be a domain containing the origin.
Let $u \in  W^{2,2}_{\alpha,{\rm loc}} (\Omega) \cap W^{1+\frac{2}{\alpha+1},\sigma}_{{ \rm loc}}(\Omega)$ for some $\sigma > \frac Q 2$ and let $v \in  W^{1,2}_{\alpha,{\rm loc}}(\Omega)$. Then for any $r>0$ such that $\overline{B_r^\alpha} \subset \Omega$ the following conclusions hold: $ \Delta_\alpha  u \in L^2(B_r^\alpha)$, the trace on $\partial B_r^\alpha$ of the first order derivatives of $u$ belong to $L^2(\partial B_r^\alpha)$, the trace of $v$ on $\partial B_r^\alpha$ belongs to $L^2_\alpha(\partial B_r^\alpha)$ and the following identity holds true
\begin{equation*}
    \int_{B_r^\alpha}  v  \Delta_\alpha  u \, dxdy=\int_{\partial B_r^\alpha} \frac{\psi_\alpha}{d_\alpha} \, v X_G u  \, d\Ha
    -\int_{B_r^\alpha} \nabla_\alpha u \cdot \nabla_\alpha v \, dxdy
\end{equation*}
where $\nu$ denotes the outer unit normal to $\partial B_r^\alpha$.
\end{proposition}

\begin{proof} The proof can be obtained by density applying  Proposition \ref{p:density} on Grushin balls $B_r^\alpha$.
In order to show that first order derivatives of $u$ belong to $L^2(\partial B_r^\alpha)$ one can repeat the argument used in the proof of Proposition \ref{p:reg-2}. On the other hand, the trace of $v$ belongs to $L^2_\alpha(\partial B_r^\alpha)$ thanks to Proposition \ref{p:trace-alpha}. Then one can use the classical integration by parts formula on smooth functions combined with the identity
$ %\label{eq:A-X}
    A_\alpha(x,y) \nabla d_\alpha(x,y)=\frac{|x|^{2\alpha}}{(d_\alpha(x,y))^{2\alpha+1}} \, X_G(x,y)\, ,
    \ (x,y)\neq (0,0)
$
and pass to the limit.
\end{proof}

\subsection{Proof of Proposition \ref{p:Pohozaev}} We recall that by Sobolev embedding $u \in  W^{1,2}_\alpha (\Omega)\subset L^q(\Omega)$ for any $1\le q<\infty$ if $Q\le 2$ and for any $1\le q\le \frac{2Q}{Q-2}$ if $Q>2$. Hence, by \eqref{eq:V}, we infer that $Vu^2 \in L^1(\Omega)$ so that $\int_{B_r^\alpha} Vu^2 \, dxdy$ is well defined for any $r>0$ such that $\overline{B_r^\alpha} \subset \Omega$ and moreover
$\int_{\partial B_r^\alpha} Vu^2 \, d\Ha$ is defined almost everywhere and
\begin{equation} \label{eq:der-Vu^2}
  \frac{d}{dr}\int_{B_r^\alpha} Vu^2 \, dxdy=\int_{\partial B_r^\alpha} Vu^2 \, d\Ha
\end{equation}
almost everywhere, see Lemma \ref{l:der-B} for more details.
Applying again Sobolev embedding and \eqref{eq:V} we also deduce that $(\nabla V \cdot X_G) u^2 \in L^1_{{\rm loc}}(\Omega)$.

By \eqref{eq:der-Vu^2}, Lemma \ref{l:change-B} and simple computations we obtain
\begin{align} \label{eq:est-prel}
   & -\frac{Q}{r} \int_{B_r^\alpha} \frac 12 \, Vu^2 \, dxdy+\int_{\partial B_r^\alpha} \frac 12 \, Vu^2 \, d\Ha
     -\frac{1}{2r} \int_{B_r^\alpha} (\nabla V \cdot X_G) u^2 \, dxdy \\[10pt]
  \notag &  \qquad =r^Q  \int_{B_1^\alpha} \frac{\partial}{\partial r} \left(\frac 12 \, V_r u_r^2 \right) \, dxdy   -r^Q \int_{B_1^\alpha} \frac 12\,  \nabla V(rx,r^{\alpha+1} y) \cdot (x,(\alpha+1)r^\alpha y) u^2_r(x,y) \, dxdy \\[10pt]
   \notag & \qquad = r^Q \int_{B_1^\alpha} V_r u_r \frac{\partial u_r}{\partial r} \, dxdy
   =r^Q \int_{B_1^\alpha} (- \Delta_\alpha  u)_r \frac{\partial u_r}{\partial r} \, dxdy
   %=r^{Q-2} \int_{B_1^\alpha} - \Delta_\alpha  u_r \frac{\partial u_r}{\partial r} \, dxdy
\end{align}
where for any function $w=w(x,y)$ we use the notation $w_r(x,y)=w(\delta_r(x,y))$ in accordance with \eqref{eq:delta-lambda}. We observe that the right hand side of identity \eqref{eq:est-prel} is well defined since by Proposition \ref{p:reg-2} we have that
$V_r u_r = (Vu)_r= (- \Delta_\alpha  u)_r\in L^2(B_1^\alpha)$ being $u\in  W^{2,2}_{\alpha,{\rm loc}} (\Omega)$ and $\frac{\partial u_r}{\partial r}=\frac 1r \, \nabla u_r \cdot X_G \in L^2(B_1^\alpha)$.
Combining the above regularity of $u$ with an approximation argument like in Proposition \ref{p:density}, we may find a sequence $\{u^n\} \subset C^\infty(\overline{B_r^\alpha})$ such that
\begin{equation} \label{eq:conv-un}
   u^n \to u \quad \text{in }   W^{2,2}_\alpha (B_r^\alpha) \qquad \text{and} \qquad u^n \to u \quad \text{in } W^{1,2}(B_r^\alpha)
   \qquad \text{as } n\to +\infty \, .
\end{equation}
Applying Proposition \ref{p:int-parts} along the sequence $\{u^n\}$ and using \eqref{eq:psi-alpha}, \eqref{eq:vect-field},
Lemma \ref{l:der-B}, Lemma \ref{l:change-B} and Lemma \ref{l:rescale-partial}, we may write
\begin{align}  \label{eq:l-pa}
   & r^Q \int_{B_1^\alpha} (- \Delta_\alpha  u^n)_r \frac{\partial u^n_r}{\partial r} \, dxdy
   =r^{Q-2} \int_{B_1^\alpha} - \Delta_\alpha  u^n_r \, \frac{\partial u^n_r}{\partial r} \, dxdy
\end{align}
\begin{align*}
 \notag  & \qquad = -r^{Q-2} \int_{\partial B_1^\alpha} \psi_\alpha \, X_G u^n_r \, \frac{\partial u^n_r}{\partial r} d\Ha
     +r^{Q-2} \int_{B_1^\alpha} \nabla_\alpha u^n_r \cdot \nabla_\alpha \left(\frac{\partial u^n_r}{\partial r}\right) \, dxdy \\[10pt]
 \notag  & \qquad =-r^{Q-2} \int_{\partial B_1^\alpha} \psi_\alpha \, X_G u^n_r \, \frac 1r X_G u^n_r \, d\Ha
    +r^{Q-2} \int_{B_1^\alpha} \frac 12 \frac{\partial}{\partial r} \left(|\nabla_\alpha u^n_r|^2\right)\, dxdy \\[10pt]
  \notag & \qquad =-r^{Q-2} \int_{\partial B_1^\alpha} r^{-1} \psi_\alpha \, [(X_G u^n)^2]_r \, d\Ha
    +\frac{r^{Q-2}}2 \frac{\partial}{\partial r} \int_{B_1^\alpha}  |\nabla_\alpha u^n_r|^2 \, dxdy  \\[10pt]
 \notag  & \qquad =-r^{-2}  \int_{\partial B_r^\alpha} \psi_\alpha (X_G u^n)^2 \, d\Ha
   +\frac{r^{Q-2}}2 \frac{\partial}{\partial r} \left(r^{2-Q} \int_{B_r^\alpha}  |\nabla_\alpha u^n|^2 \, dxdy  \right) \\[10pt]
\notag   & \qquad =-r^{-2}  \int_{\partial B_r^\alpha} \psi_\alpha (X_G u^n)^2 \, d\Ha
   -\frac{Q-2}{2} \, r^{-1}  \int_{B_r^\alpha}  |\nabla_\alpha u^n|^2 \, dxdy
    +\frac 12 \int_{\partial B_r^\alpha}  |\nabla_\alpha u^n|^2 \, d\Ha\,.
\end{align*}
Passing to the limit as $n\to +\infty$ in the left and right hand sides of identity \eqref{eq:l-pa}, by \eqref{eq:conv-un} we obtain
\begin{align}  \label{eq:l-pa-bis}
   & r^Q \int_{B_1^\alpha} (- \Delta_\alpha  u)_r \frac{\partial u_r}{\partial r} \, dxdy  \\[10pt]
   \notag & \qquad = -\frac{Q-2}{2r} \,  \int_{B_r^\alpha}  |\nabla_\alpha u|^2 \, dxdy
     +\frac 12 \int_{\partial B_r^\alpha}  |\nabla_\alpha u|^2 \, d\Ha
     -\frac{1}{r^2} \int_{\partial B_r^\alpha} \psi_\alpha (X_G u)^2 \, d\Ha \, .
\end{align}
The proof of the proposition then follows combining \eqref{eq:est-prel} with \eqref{eq:l-pa-bis}.

\section*{Acknowledgments}

The authors are members of the ``Gruppo Nazionale per l’Analisi Matematica, la Probabilit\`a e le loro Applicazioni'' (GNAMPA) of the ``Istituto Nazionale di Alta Matematica'' (INdAM). L.A. is partially supported by the PRIN 2022 project 2022R537CS \emph{$NO^3$ - Nodal Optimization, NOnlinear elliptic equations, NOnlocal geometric problems, with a focus on regularity}, founded by the European Union - Next Generation EU.  A.F. and P.L.  acknowledge the support  from the project \emph{Perturbation problems and asymptotics for elliptic differential equations: variational and potential theoretic methods} funded by the MUR Progetti di Ricerca di Rilevante Interesse Nazionale (PRIN) Bando 2022 grant 2022SENJZ3.

\end{document}